\documentclass[11pt]{amsart}

\usepackage[text={14.06cm,22.84cm},centering]{geometry}

\usepackage{color}
\usepackage{esint,amssymb}
\usepackage{graphicx}
\usepackage{mathtools}
\usepackage[colorlinks=true, pdfstartview=FitV, linkcolor=blue, citecolor=blue, urlcolor=blue,pagebackref=false]{hyperref}
\usepackage{microtype}

\usepackage{bm}
\usepackage{scalerel} 
\usepackage{mathrsfs}
\usepackage[font={footnotesize}]{caption}

\usepackage{enumitem}

\usepackage{comment}

\usepackage{scalerel,stackengine}
\stackMath
\newcommand\reallywidehat[1]{%
\savestack{\tmpbox}{\stretchto{%
  \scaleto{%
    \scalerel*[\widthof{\ensuremath{#1}}]{\kern-.6pt\bigwedge\kern-.6pt}%
    {\rule[-\textheight/2]{1ex}{\textheight}}
  }{\textheight}%
}{0.5ex}}%
\stackon[1pt]{#1}{\tmpbox}%
}
\definecolor{darkgreen}{rgb}{0,0.5,0}
\definecolor{darkblue}{rgb}{0,0,0.7}
\definecolor{darkred}{rgb}{0.9,0.1,0.1}
\definecolor{lightblue}{rgb}{0,0.51,1}

\newtheorem{theorem}{Theorem}
\newtheorem{proposition}{Proposition}
\newtheorem{lemma}[proposition]{Lemma}
\newtheorem{corollary}[proposition]{Corollary}

\theoremstyle{remark}
\newtheorem{remark}[proposition]{Remark}

\theoremstyle{definition}

\numberwithin{equation}{section}
\numberwithin{proposition}{section}

\newcommand{\Z}{\mathbb{Z}}
\newcommand{\N}{\mathbb{N}}
\newcommand{\R}{\mathbb{R}}
\newcommand{\C}{\mathbb{C}}

\newcommand{\ep}{\varepsilon}

\renewcommand{\subset}{\subseteq}

\DeclareMathOperator{\Rel}{Re}

\renewcommand{\tilde}{\widetilde}

\begin{document}

\title[Navier-Stokes equations in the half-space for non localized data]{Estimates for the Navier-Stokes equations in the half-space for non localized data}

\begin{abstract}
This paper is devoted to the study of the Stokes and Navier-Stokes equations, in a half-space, for initial data in a class of locally uniform Lebesgue integrable functions, namely $L^q_{uloc,\sigma}(\R^d_+)$. We prove the analyticity of the Stokes semigroup $e^{-t{\bf A}}$ in $L^q_{uloc,\sigma}(\R^d_+)$ for $1<q\leq\infty$. This follows from the analysis of the Stokes resolvent problem for data in $L^q_{uloc,\sigma}(\R^d_+)$, $1<q\leq\infty$. We then prove bilinear estimates for the Oseen kernel, which enables to prove the existence of mild solutions. The three main original aspects of our contribution are: (i) the proof of Liouville theorems for the resolvent problem and the time dependent Stokes system under weak integrability conditions, (ii) the proof of pressure estimates in the half-space and (iii) the proof of a  concentration result for blow-up solutions of the Navier-Stokes equations. This concentration result improves a recent result by Li, Ozawa and Wang and provides a new proof.
\end{abstract}

\author[Y. Maekawa]{Yasunori Maekawa}
\address[Y. Maekawa]{Kyoto University, Department of Mathematics, Kyoto, Japan}
\email{maekawa@math.kyoto-u.ac.jp}

\author[H. Miura]{Hideyuki Miura}
\address[H. Miura]{Tokyo Institute of Technology, Department of 
Mathematical and Computing Sciences, Tokyo, Japan}
\email{miura@is.titech.ac.jp}

\author[C. Prange]{Christophe Prange}
\address[C. Prange]{Universit\'e de Bordeaux, CNRS, UMR [5251], IMB, Bordeaux, France}
\email{christophe.prange@math.u-bordeaux.fr}

\keywords{}
\subjclass[2010]{}
\date{\today}

\maketitle

\section{Introduction}

This paper is devoted to the study of fluid equations in the half-space $\R^d_+$. Our goal is twofold. First we show the analyticity of the Stokes semigroup for data belonging to the locally uniform Lebesgue space $L^q_{uloc,\sigma}(\R^d_+)$ for $1<q\leq \infty$ (uniformly locally in $L^p$, divergence-free and normal component zero on the boundary; a precise definition is given below in Section \ref{sec.setting}). Second we prove optimal bounds for the Oseen kernel $e^{-t{\bf A}}\mathbb P\nabla\cdot$ and get as a by-product the short time existence of mild solutions to the Navier-Stokes system with noslip boundary conditions
\begin{equation}\label{eq.ns.intro}
  \left\{
\begin{aligned}
 \partial_t u - \Delta u + \nabla p & = -\nabla \cdot   (u \otimes u ) , \quad \nabla \cdot u = 0  \qquad \mbox{in}~ (0,T)\times \R^d_+\,, \\
u & = 0\quad \mbox{on}~ (0,T)\times \partial\R^d_+,  \qquad u|_{t=0}=u_0 \quad \mbox{in}~\partial \R^d_+,
\end{aligned}\right.
\end{equation}
for non localized initial data $u_0\in L^q_{uloc,\sigma}(\R^d_+)$, $q\geq d$. Our results directly yield the concentration of the scale critical $L^d$ norm for blow-up solutions of the Navier-Stokes equations. 

\subsection{Outline of our results}

Our analysis relies on the study of the Stokes resolvent problem. The first paragraph below contains our main result in this direction. This enables (see second paragraph) to prove the analyticity of the Stokes semigroup in $L^q_{uloc,\sigma}$ for $q\in(1,\infty]$. We then (third paragraph below) state the bilinear estimates for the Oseen kernel, which allow the study of mild solutions in a way that is standard since the work of Fujita and Kato \cite{FK64}. We state the concentration result. The fourth paragraph is devoted to the Liouville theorems proved in Appendix \ref{sec.liouville} and Appendix \ref{sec.liouville.nonsteady}.

Let us emphasize three aspects of our results, which we believe are the most original. First, our Liouville theorems for the resolvent problem and the time dependent Stokes system hold under weak integrability conditions. Second, we prove pressure estimates in the half-space, which are key to our analysis of local energy weak solutions in \cite{MMP17b}. Third, we show a concentration phenomenon for blow-up solutions of the Navier-Stokes equations. Our result improves a recent result by Li, Ozawa and Wang \cite{LOW16} and provides a new proof. These aspects are discussed more extensively in the proceedings paper \cite{P18Xedp}.

\subsubsection*{Stokes resolvent problem}
The following statements are the main tools of the rest of the paper.  A considerable part of our work is concerned with estimates for the resolvent problem for the (stationary) Stokes system
\begin{equation}
\label{e.resol}
\left\{ 
\begin{aligned}
& \lambda u -\Delta u+ \nabla p = f, \quad  \nabla\cdot u=0  & \mbox{in} &\ \R^d_+, \\
& u = 0  & \mbox{on} & \ \partial\R^d_+. 
\end{aligned}
\right.
\end{equation}
for non localized data $f$ in the class $L^q_{uloc,\sigma}(\R^d_+)$, for $1<q\leq\infty$.

\begin{theorem}
\label{prop.fresolvent}
Let $1 < q \leq \infty$, $\ep >0$. Let $\lambda$ be a complex number in the sector $S_{\pi-\ep}$.  Let $f\in L^q_{uloc,\sigma}(\R^d_+)$. Then there exist $C(d,\ep,q)<\infty$ (independent of $\lambda$) and  a unique solution $(u,\nabla p)\in L^q_{uloc}(\R^d_+)^d \times L^1_{uloc} (\R^d_+)^d$ to \eqref{e.resol} in the sense of distributions such that 
\begin{align} 
|\lambda| \| u\|_{L^q_{uloc}} + |\lambda|^\frac12 \| \nabla u \|_{L^q_{uloc}} & \leq C\| f\|_{L^q_{uloc}},\label{est.prop.fresolvent.1}\\
\| \nabla^2 u \|_{L^q_{uloc}} + \| \nabla p \|_{L^q_{uloc}} & \leq C \left(1+e^{-c|\lambda|^\frac12} \log |\lambda| \right)\| f\|_{L^q_{uloc}},\quad\mbox{for}\quad q\neq \infty\label{est.prop.fresolvent.2}
\end{align}
and 
\begin{equation}\label{e.condpressure}
\lim_{R \rightarrow \infty} \| \nabla' p \|_{L^1(|x'|<1, R<x_d<R+1)}=0.
\end{equation}
Moreover, for $1<q=p\leq \infty$ or $1 \leq q < p\leq \infty$ satisfying $\frac1q-\frac1p<\frac1d$, there exists a constant $C(d,\ep,q,p)<\infty$ (independent of $\lambda$) such that
\begin{align}
\| u\|_{L^p_{uloc}} & \leq C | \lambda|^{-1} \left(1+|\lambda|^{\frac{d}{2}(\frac1q-\frac1p)}\right) \| f \|_{L^q_{uloc}}, \label{est.prop.fresolvent.3}\\
\| \nabla u\|_{L^p_{uloc}} & \leq C |\lambda|^{-\frac12} \left( 1+ |\lambda|^{\frac{d}{2}(\frac1q-\frac1p)}\right) \| f \|_{L^q_{uloc}}.\label{est.prop.fresolvent.4}
\end{align} 
\end{theorem}

Theorem \ref{prop.fresolvent} is proved in Section \ref{sec.resolest}.

Uniqueness comes from the condition \eqref{e.condpressure} which eliminates the parasitic solutions of our Liouville-type result, Theorem \ref{thm.unique} given in the Appendix. 
Condition \eqref{e.condpressure} is easily verified for the pressure represented via the integral formulas of Section \ref{sec.intrepr}. With Theorem \ref{prop.fresolvent}, one can define the resolvent operator $R(\lambda)=(\lambda+\bf A)^{-1}$ on the sector $S_{\pi-\ep}$ for given $\ep>0$. As is classical, the bounds on the solution to the resolvent problem are crucial to estimate the semigroup. The mixed $p,\, q$ bounds \eqref{est.prop.fresolvent.3} and \eqref{est.prop.fresolvent.4} are particularly important in view of the study of the nonlinear term in the Navier-Stokes equations. Let us comment on two points. First, we are not able to remove the $|\log(\lambda)|$ loss for small $\lambda$ in \eqref{est.prop.fresolvent.2}. We are ignorant whether there is a real obstruction or if this is just a technical issue. Second, the estimate \eqref{est.prop.fresolvent.1} fails for $q=1$. This is a fundamental point, which was already observed in the case of $L^1(\R^d_+)$ by Desch, Hieber and Pr\"uss in \cite{DHP01}. It is specific to the case of $\R^d_+$ as opposed to $\R^d$. We comment more on this below.

\subsubsection*{Stokes semigroup}

Let $\bf A$ be the Stokes operator realized in $L^q_{uloc,\sigma}(\R^d_+)$ (for precise definitions, see Section \ref{sec.semigroup}). Time dependent estimates for the semigroup are derived from the resolvent bounds using classical techniques from complex analysis. 

\begin{theorem}\label{theo.analyticity}
Let  $1<q \leq \infty$. Then $-{\bf A}$ generates a bounded analytic semigroup in $L^q_{uloc,\sigma}(\R^d_+)$.
\end{theorem}

More precise statements (and their proofs) along with global in time estimates for the linear Stokes dynamic are given in Section \ref{sec.semigroup}, see in particular propositions \ref{prop.analyticity} and \ref{prop.L^p_u-L^q_u.semigroup}. Again, because of the failure of \eqref{est.prop.fresolvent.1} when $q=1$, $-{\bf A}$ fails to generate an analytic semigroup in $L^1_{uloc,\sigma}(\R^d_+)$. This is due to the presence of the boundary.

\subsubsection*{Bilinear estimates, mild solutions and concentration for blow-up solutions}

Our main result is the following bilinear estimate, from which the 
short time existence of mild solutions follows as a corollary.

\begin{theorem}\label{prop.L^p_u-L^q_u.semigroup.inhomo}  
Let  $1<q \leq p \leq \infty$ or $1\leq q<p\leq\infty$ and let ${\bf A}$ be the Stokes operator realized in $L^q_{uloc,\sigma}(\R^d_+)$.  Then for $\alpha=0,1$, and $t>0$,
\begin{align}\label{est.prop.L^p_u-L^q_u.semigroup.inhomo.1} 
\| \nabla^\alpha  e^{-t {\bf A}} \mathbb{P} \nabla \cdot (u\otimes v) \|_{L^p_{uloc}}  \leq Ct^{-\frac{1+\alpha}{2}}  \big (t^{-\frac{d}{2}(\frac1q-\frac1p)} +1  \big ) \| u\otimes v \|_{L^q_{uloc}}, 
\end{align}
and also 
\begin{align}\label{est.prop.L^p_u-L^q_u.semigroup.inhomo.2} 
\| \nabla  e^{-t {\bf A}} \mathbb{P} \nabla \cdot (u\otimes v) \|_{L^q_{uloc}}  \leq Ct^{-\frac{1}{2}}  \big ( \| u \nabla v \|_{L^q_{uloc}} + \| v\nabla u \|_{L^q_{uloc}} \big ).
\end{align}
Estimates \eqref{est.prop.L^p_u-L^q_u.semigroup.inhomo.1} and \eqref{est.prop.L^p_u-L^q_u.semigroup.inhomo.2} are valid also for $q=1$.
\end{theorem}

The proof of Theorem \ref{prop.L^p_u-L^q_u.semigroup.inhomo} is given in Section \ref{sec.bilinse}. 
The function $e^{-t {\bf A}} \mathbb{P} \nabla \cdot (u\otimes v)$ is expressed in terms of the integral over $\R^3_+$ with some kernels satisfying suitable pointwise estimates, and such a representation itself is well-defined and satisfies \eqref{est.prop.L^p_u-L^q_u.semigroup.inhomo.1} and \eqref{est.prop.L^p_u-L^q_u.semigroup.inhomo.2} even for the case $q=1$, though we do not have the analyticity of the semigroup $\{e^{-t{\bf A}}\}_{t\geq 0}$ in $L^1_{uloc,\sigma} (\R^3_+)$.

There is quite some work to go from the semigroup bounds on $e^{-t{\bf A}}$ to the bounds on $e^{-t{\bf A}}\mathbb P\nabla\cdot$ of Theorem \ref{prop.L^p_u-L^q_u.semigroup.inhomo}, which are needed to solve the Navier-Stokes equations. 
Two of the key difficulties are: (i) making sense of the action of $e^{-t{\bf A}}\mathbb P\nabla\cdot$ on $L^q_{uloc,\sigma}$ functions, since 
the Helmholtz-Leray projection does not act as a bounded operator on $L^q_{uloc}$, (ii) the non commutativity of vertical derivatives and the symbol (in tangential Fourier) for the  projection $\mathbb P$.

The existence of mild solutions for initial data $u_0\in L^q_{uloc,\sigma}$, $q\geq d$, is stated in Proposition \ref{prop.mild.ns} and Proposition \ref{thm.mild.ns} in Section \ref{sec.nse}. Once the bilinear estimate of Theorem \ref{prop.L^p_u-L^q_u.semigroup.inhomo} is established, the local in time existence of mild solutions
can be shown by the standard arguments. 
It is not our objective to include a deeper discussion of the mild solutions here. 
We only note that the pressure $p$ associated to the mild solution $u$ can be constructed 
so that the pair $(u, p)$ satisfies  \eqref{eq.ns.intro} in the sense of distributions;
see \cite{MMP17b} for the details. 
Many other issues, such as the convergence to the initial data, are similar in the half-space and the whole space. Hence we simply refer to \cite{MT06} where a thorough discussion is carried out.

As an application of the well-posedness in $L^q_{uloc,\sigma}$, $q\geq d$, 
we study the behavior of the blow-up solution.
Here a solution $u \in C((0,T_*);L^\infty_\sigma(\R^d_+))$ 
blows up at $t = T_*$
if $\displaystyle{\limsup_{t\uparrow T_*} \|u(t)\|_{L^\infty}}=\infty$. 
Leray \cite{Leray} proved a lower bound  
for the blow-up solutions in $\R^d$ of the form
\begin{equation}
\label{est:rate}
\|u(t)\|_{L^\infty} \ge \frac{\varepsilon}{\sqrt{T_*-t}}  \qquad {\rm for}
\ t <T_*. 
\end{equation}  
More recently lower bounds for scale critical norms of the 
blow-up solution have been also studied extensively.
It is shown in \cite{BS15} that 
$\lim_{t\rightarrow T_*} \|u(t)\|_{L^3(\R^3_+)}=\infty$.
On the other hand, the authors in the recent work \cite{LOW16} showed a lower bound for 
the local $L^d$ norm along the parabolic cone for the blow-up solutions. 
More precisely, it is proved that there exists 
a sequence $(x_n,t_n) \in \R^3 \times (0,T_*)$ such that 
$\|u(t_n)\|_{L^d(x_n+(0,c_0\sqrt{T_*-t_n})^d)} \ge \varepsilon$. 
This can be seen as a concentration phenomenon of the 
critical norm at the blow-up time. In the same work, concentration results along the parabolic cone are also proved for $L^q$ norms, $q\geq d$, along a discrete sequence of times. To the best of our knowledge, this type of concentration results is new for the Navier-Stokes equations, even in the whole space $\R^d$. While in \cite{LOW16} the concentration results are proved via a clever use of frequency decomposition techniques, it occurred to us that they are simple consequences of the existence of mild solutions in $L^q_{uloc}$. As a direct consequence of the well-posedness result, 
we thus have the following 
corollary.

\begin{corollary}
\label{cor.concentration}
For all $q\geq d$, there exists a positive constant $\gamma<\infty$ such that for all $0<T_*<\infty$, for all $u \in C((0,T_*);L^\infty_\sigma(\R^d_+))$ mild solution to \eqref{eq.ns.intro}, if $u$ blows up at $t = T_*$, then for all $t\in(0,T_*)$, there exists $x(t)\in\R^d_+$ with the following estimate
\begin{equation*}
\|u(t)\|_{L^q(|x(t)-\cdot|\leq \sqrt{T_*-t})} \ge \frac{\gamma}{(T_*-t)^{\frac12(1-\frac dq)}}.\\
\end{equation*}
\end{corollary}

This corollary is proved in Section \ref{sec.concentration}. 
Our result also holds for $\R^d$ instead of $\R^d_+$ and seems to be new even in that case. It obviously includes the concentration of the scale-critical norm $L^d$. It improves Theorem 1.2 in \cite{LOW16} in the sense that the concentration holds not only along a sequence of times $t_n\rightarrow T_*$, but for all times $t\in(0,T_*)$. Moreover, our method gives a new and simple proof of this type of results, which appear to be nice applications of the existence of mild solutions in the $L^q_{uloc}$ setting.

\subsubsection*{Liouville theorems}

Here we give a uniqueness result to \eqref{e.resol} in our functional framework. This Liouville theorem is the counterpart for the resolvent system to the Liouville theorem for the unsteady Stokes system in the half-space proved in \cite{JSS12} and crucial for the uniqueness part of Theorem \ref{prop.fresolvent}.

\begin{theorem}\label{thm.unique} {\rm (i)} Let $\lambda\in S_{\pi-\ep}$ with $\ep \in (0,\pi)$. Let $(u,\nabla p)\in  L^1_{uloc}(\R^d_+)^d \times L^1_{uloc} (\R^d_+)^d$ with $p\in L^1_{loc}(\R^d_+)$ be a solution to \eqref{e.resol} with $f=0$ in the sense of distributions. Then $(u,\nabla p)$ is a parasitic solution, i.e., $u=(a'(x_d),0)^\top$ and $p=D \cdot x'+c$. 
Here $a'(x_d)=(a_1(x_d),\cdots, a_{d-1}(x_d))^\top$ is smooth and bounded and its trace at $x_d=0$ is zero, while $D\in \C^{d-1}$ is a constant vector and $c\in \C$ is a constant.
If  either $\displaystyle \lim_{R\rightarrow \infty} \| \nabla' p  \|_{L^1 (|x'|<1, R<x_d<R+1)}=0$ or $\displaystyle \lim_{|y'| \rightarrow \infty} \| u  \|_{L^1 (|x'-y'|<1, 1<x_d<2)}=0$ in addition, then $p$ is a constant and $u=0$. 

\noindent {\rm (ii)}  Let $(u,\nabla p)\in  L^1_{uloc}(\R^d_+)^d \times L^1_{uloc} (\R^d_+)^d$ with $p\in L^1_{loc}(\R^d_+)$ be a solution to \eqref{e.resol} with $\lambda=0$ and $f=0$ in the sense of distributions. Then $u=0$ and $p$ is a constant. 
\end{theorem}

The proof of this theorem is given in Appendix \ref{sec.liouville}. By using a similar argument  we shall show a uniqueness result for the time-dependent Stokes problem:
\begin{equation}\label{eq.s.appendix}
  \left\{
\begin{aligned}
 \partial_t u - \Delta u + \nabla p & = f , \quad \nabla \cdot u = 0  \qquad \mbox{in}~ (0,T)\times \R^d_+\,, \\
u & = 0\quad \mbox{on}~ (0,T)\times \partial\R^d_+,  \qquad u|_{t=0}=u_0 \quad \mbox{in}~\partial \R^d_+,
\end{aligned}\right.
\end{equation}
Compared with the known Liouville theorem by \cite{JSS12} for bounded solutions, our result imposes the condition on the pressure, while the regularity condition on the velocity  is weaker than in \cite{JSS12}.
This framework will be useful in the study of the local energy weak solutions.

\begin{theorem}\label{thm.unique.time} Let $(u, \nabla p)$ be a solution to \eqref{eq.s.appendix} in the sense of distributions with $u_0=f=0$. Then $(u,\nabla p)$ is a parastic solution, i.e., $u (t,x) =(a'(t,x_d),0)^\top$ and $p (t,x) =D (t) \cdot x' + c(t)$. Here $a'(t,x_d) = (a_1(t,x_d), \cdots,a_{d-1}(t,x_d))^\top$ with $a_j\in L^1_{loc}([0,T)\times \overline{\R^d_+})$, while $D\in L^1_{loc} (0,T)^{d-1}$ and $c\in L^1_{loc} (0,T)$. If either 
$$\displaystyle \lim_{R\rightarrow \infty} \int_\delta^T \| \nabla'p (t) \|_{L^1(|x'|<1,R<x_d<R+1)} d t =0\quad \text{for all} ~~\delta \in (0,T)$$
or 
$$\displaystyle \lim_{|y'|\rightarrow \infty} \int_0^T \| u(t) \|_{L^1 (|x'-y'|<1,\, 1<x_d<2)} d t =0$$
in addition, then $p$ is a constant and $u=0$.  
\end{theorem}

This theorem is proved in Appendix \ref{sec.liouville.nonsteady}. There we state precisely the notion of solutions to \eqref{eq.s.appendix}. To our knowledge, these two Liouville-type results, Theorem \ref{thm.unique} and Theorem \ref{thm.unique.time}, are new under these assumptions. 

\subsection{Comparison to other works}

We give an overview of some works related to our result. Although we try to give a faithful account of the state of the art of the study of fluid equations (mainly Stokes and Navier-Stokes systems) with infinite energy or non localized data, we are very far from being exhaustive in our description. We divide the description into three parts: first we deal with the class of bounded functions, second we handle the class of locally uniform Lebesgue spaces and finally we describe some differences between the whole space and the half-space.

A common feature of the analysis in $L^\infty_\sigma$ and $L^q_{uloc,\sigma}$ is the failure of classical techniques used in works on Stokes and Navier-Stokes equations. This appears at several levels. Firstly, there is of course no global energy inequality. The substitute is a local energy inequality, which involves the pressure. Hence one has to obtain precise information on the pressure. Secondly, there is obviously no uniqueness for flows with infinite energy. This is due to flows driven by the pressure (solving the Stokes and Navier-Stokes equations), such as in the whole space $\R^d$
\begin{equation*}
u(x,t):=f(t)\quad\mbox{and}\quad p(x,t):=-f'(t)\cdot x, 
\end{equation*}
or in the half-space $\R^d_+$
\begin{equation*}
u(x,t):=(v_1(x_d,t),\ldots\, v_{d-1}(x_d,t), 0)\quad\mbox{and}\quad p(x,t):=-f(t)\cdot x', 
\end{equation*}
where $f\in C^{\infty}_0((0,\infty);\R^{d-1})$ and $v(x_d,t)$ solves the heat equation $\partial_tv-\partial_d^2v=f$ with $v(0,t)=0$. Hence, one has to handle or eliminate these parasitic solutions. Thirdly, even in the instance where flows driven by the pressure are ruled out, one needs to make sense of a representation formula for the pressure. Indeed the source term in the elliptic equation for the pressure
\begin{equation*}
-\Delta p=\nabla\cdot\big(\nabla\cdot (u\otimes u)\big),
\end{equation*}
is non localized, thus a priori non decaying at infinity. Fourthly, the Helmholtz-Leray projection is not bounded on $L^\infty$ nor on $L^q_{uloc}$, basically because the Riesz transform does not map $L^\infty$ into itself. This makes the study of the mapping properties of the Stokes operator ${\bf A}$, which is usually defined as ${\bf A} = - \mathbb{P} \Delta_D$, where $\mathbb P$ is the Helmholtz-Leray projection and $\Delta_D$ the Dirichlet Laplacian, particularly delicate.

\subsubsection*{Bounded functions}
From the point of view of both the results and the techniques, the main source of inspiration of the linear Stokes estimates is the paper by Desch, Hieber and Pr{\"u}ss \cite{DHP01}. This paper is concerned with the study of the Stokes semigroup in the half-space in particular in the class of bounded functions. The authors prove the analyticity of the Stokes semigroup in $L^q_\sigma(\R^d_+)$ for $1<q\leq\infty$. The case of $1<q<\infty$ was previously known. Their approach is based on the study of the Stokes resolvent problem \eqref{e.resol}. In order to circumvent the use of the Helmholtz-Leray projection, one of the key ideas is to decompose the resolvent operator into a part corresponding to the Dirichlet-Laplace part and another part associated with the nonlocal pressure term
\begin{align*}
(\lambda+{\bf A})^{-1} = {\bf R}_{D.L.}(\lambda) + {\bf R}_{n.l.} (\lambda).
\end{align*}
Our work uses the same idea, but we need more precise estimates on the kernels than the mere $L^1$ bounds proved in \cite[Section 3]{DHP01}, which are not enough for our purposes.

The $L^\infty$ theory for the Stokes equations has recently been advanced thanks to a series of works by Abe, Giga and Hieber. In \cite{AG13}, the Stokes semigroup is proved to be analytic via an original (in this context) compactness (or blow-up) method in admissible domains, which include bounded domains and the half-space. In these domains a bound on the pressure holds, which excludes the parasitic solutions previously mentioned. However, an intrinsic drawback of the compactness argument, is that it only gives an $L^\infty$ bound on the solution for times $0<t<T_0$, with $T_0$ depending only on the domain. The papers \cite{Abe16,AG14} build on the same method. Concerning the resolvent problem, it was considered in \cite{AGH15} by a localization argument, which boils down to applying locally the $L^p$ theory and interpolating to get a control in $L^\infty$. These developments enable to investigate blow-up rates \eqref{est:rate}
in $L^\infty$ for potential singularities in the solutions of the Navier-Stokes equations \cite{Abe15}.

\subsubsection*{Locally uniform Lebesgue spaces}

The locally uniform Lebesgue spaces $L^q_{uloc}$ form a wider class of functions than $L^\infty$. They include a richer spectrum of behaviors, obviously allowing for some singular behavior (homogeneous functions slowly decaying at $\infty$) or non decaying functions such as locally $L^p$ periodic or almost-periodic functions. The main advantage of this class is that it is easy to define and visualize, while it includes various class of functions as mentioned above. In the operator theoretical point of view, another advantage is that we can characterize the domain of the Stokes operator in the $L^q_{uloc}$ spaces if $1<q<\infty$ (see Section \ref{sec.semigroup}), which is hard to expect in the $L^\infty$ framework even for the Laplace operator. 
On the other hand, the main and major drawback is that the $L^q_{uloc}$ functions are difficult to handle in the context of pseudodifferential calculus. Indeed, there is no obvious characterization in Fourier space, which makes it difficult to straightly analyze the action of Fourier symbols. Most of the time, one has to first derive kernel bounds for the symbols in physical space, before estimating the $L^q_{uloc}$ norms. Many equations have been studied in the framework of loc-uniform spaces. Without aiming at exhaustivity, let us mention some works parabolic-type equations: on linear parabolic equations \cite{ARCD04}, on Ginzburg-Landau equations \cite{MS95,GV97} and on reaction-diffusion equations \cite{CD04}. We refer to these works for basic properties of the space $L^q_{uloc}$.

The study of the Stokes semigroup and the application to the existence of mild solutions to the Navier-Stokes equations with initial data $u_0\in L^q_{uloc,\sigma}(\R^d)$, $q\geq d$, was carried out in \cite{MT06}. 
As regards the existence of weak solutions for initial data in $L^2_{uloc,\sigma}(\R^d)$ satisfying the local energy inequality, the so-called suitable weak solutions 
it was handled by Lemari\'e-Rieusset \cite[Chapter 32 and 33]{lemariebook}. This result was also worked out in \cite{KS07}. Existence of these local energy solutions is the key to the blow-up of the $L^3$ norm criteria at the blow-up time proved by Seregin in \cite{Ser12} for the three-dimensional Navier-Stokes equations and also to the construction of 
the forward self-similar solutions by Jia and Sverak in \cite{JS}.

\subsubsection*{Whole space vs. half-space}

Fewer results are proved for the half-space and more generally for unbounded domains with (unbounded) boundaries. Let us emphasize some phenomena related to the presence of boundaries.

A striking feature of the half-space case is the failure of $L^1(\R^d_+)$ estimates, for the resolvent problem as well as the semigroup. This fact was proved in \cite[Section 5]{DHP01}. In the whole space, the Stokes semigroup is known to be analytic even for $q=1$ (see \cite{MT06}). As underlined in \cite{DHP01}, one should relate this lack of analyticity in $L^1$ to the non existence of local mild solutions to the Navier-Stokes equations in $L^1$ for an exterior domain \cite{Koz98}. Existence of such solutions would imply that the total force acting on the boundary is zero.

On a different note, the Helmholtz decomposition may fail even in $L^q$, for some $1<q<\infty$ for smooth sector-like domains with sufficiently large opening (see \cite[Remark III.1.3]{Galdi_book}). On the contrary, the decomposition is known to hold for any $1<q<\infty$, for any smooth domain with compact boundary, for the half-space and the whole space (see \cite[Theorem III.1.2]{Galdi_book}). The definition of the Stokes semigroup in $L^q$ spaces for finite $q$ in non Helmholtz sector-like domains was recently addressed in \cite{AGSS15}.

The works of Abe and Giga, notably \cite{AG13}, aim at extending results known for the Stokes semigroup in $L^\infty_\sigma(\R^d_+)$ to more general domains with boundaries. They introduce a class of admissible domains (which includes smooth bounded domains and $\R^d_+$) in which the analyticity of the Stokes semigroup holds in $L^\infty_\sigma$. This work however says nothing in general about the longtime behavior of the linear Stokes dynamics. Indeed the $L^\infty$ bound for the Stokes dynamics is true only on a time interval $(0,T_0)$, with $T_0$ depending only on the domain. Notice that $T_0=\infty$ for smooth bounded domains and for the half-space. As regards the existence of mild solutions to the Navier-Stokes equations in the half-space for bounded data, let us mention the work of Solonnikov \cite{Sol03} (initial data bounded and continuous) and of Bae and Jin \cite{BJ12} (initial data in just bounded). These works are based on direct estimates on the kernels of the non stationary Stokes system. Our approach for the solvability of the Navier-Stokes equations in the non localized class $L^q_{uloc,\sigma}(\R^d_+)$ is based on the analysis of the bilinear operator 
\begin{equation}\label{e.bilintro}
(u,v)\mapsto(\lambda+{\bf A})^{-1}\mathbb P\nabla\cdot(u\otimes v).
\end{equation}
from which we derive bounds for the unsteady problem. A key issue of the half-space as opposed to the whole space is the non commutativity of $\mathbb P$ and vertical derivatives $\partial_d$ (the commutation with tangential derivatives does work). This prompts the need to integrate by parts in the vertical direction so as to analyze \eqref{e.bilintro} (see Section \ref{sec.leraydiv} below).

To finish this overview, let us mention that stationary Stokes, Stokes-Coriolis and Navier-Stokes-Coriolis systems with infinite energy Dirichlet boundary condition were also considered in the context of boundary layer theory. The domain is usually a perturbed half-space with a highly oscillating boundary $x_d\geq \omega(x')$. The results of \cite{DP14,DGV17} are well-posedness results in the class of Sobolev functions with locally uniform $L^2$ integrability in the tangential variable and $L^2$ integrability in the vertical variable. The main challenges are first the bumpiness of the boundary, which prevents from using the Fourier transform close to the boundary and second the lack of a priori bounds on the function itself, which requires to rely on Poincar\'e type inequalities. 
The reader is also referred to Geissert and Giga \cite{GG08}, where the Stokes resolvent equations in the exterior domain are analyzed in the $L^p_{uloc}$ space. In \cite{GG08} the compactness of the boundary is essentially used.

\subsection{Overview of the paper}

In Section \ref{sec.prelim}, the reader can find standard notations used throughout the paper, the definitions of the functional spaces as well as the computation of the Fourier symbols for the resolvent problem. As stated above, we rely on the decomposition of the solution to the resovent problem into a part corresponding to the solution of the Dirchlet-Laplace problem and a part associated with the nonlocal pressure. Section \ref{sec.resolhp} is devoted to getting pointwise bounds on the kernels for the resolvent problem defined in the physical space. In this regard, Lemma \ref{lem.optdecay} is the basic tool so as to get the optimal pointwise estimates. These bounds on the kernels stated in Lemma \ref{l.kernel} (local Dirichlet-Laplace part), Proposition \ref{prop.estkernel} (nonlocal pressure part) and Proposition \ref{prop.kernpressure} (pressure) form an essential part of our work. They are indispensable for the estimates in $L^q_{uloc}$ obtained in Section \ref{sec.resolest}. In this section, Theorem \ref{prop.fresolvent} is proved. The next section, Section \ref{sec.semigroup} establishes the analyticity of the Stokes semigroup (Theorem \ref{theo.analyticity}) along with the bounds on the longtime dynamic of the linear Stokes equation stated in Proposition \ref{prop.L^p_u-L^q_u.semigroup}. Section \ref{sec.bilinse} contains the crucial bilinear estimates (Proposition \ref{thm.estbilinop} and Theorem \ref{prop.L^p_u-L^q_u.semigroup.inhomo}) needed to prove the existence of mild solutions to the Navier-Stokes equations \eqref{eq.ns.intro}. 
The proofs of Proposition \ref{prop.mild.ns}, Proposition \ref{thm.mild.ns} and Corollary 
\ref{cor.concentration} are given in Section \ref{sec.nse}. 
Appendix \ref{sec.liouville} is concerned with the proof of the Liouville-type result of Theorem \ref{thm.unique} for the resolvent problem \eqref{e.resol}. The Liouville theorem for the non steady Stokes system, Theorem \ref{thm.unique.time}, is proved in Appendix \ref{sec.liouville.nonsteady}.

\section{Preliminaries}
\label{sec.prelim}

\subsection{Notations}
\label{sec.notations}

Throughout the paper (unless stated otherwise), the small greek letters $\alpha,\, \beta,\, \gamma,\, \iota,\, \eta$ usually denote integers or multi-indices, and $\ep,\, \delta,\, \kappa>0$ denote small positive real numbers. When it is clear from the context, we also sometimes use Einstein's summation convention for repeated indices. For $(x',x_d)\in\R^d_+$, $x'\in\R^{d-1}$ is the tangential component, while $x_d$ is the vertical one. The complex scalar number $\lambda\in\mathbb C$ belongs to the sector
\begin{equation*}
S_{\pi-\ep}:=\{\rho e^{i \theta}:\ \rho>0,\, \theta\in[-\pi+\ep,\pi-\ep]\}\subset\mathbb C,
\end{equation*}
with $\ep$ fixed in $(0,\pi)$. For $\xi\in\R^{d-1}$, we define 
\begin{align}\label{def.omega.lamxi}
\omega_\lambda(\xi) :=\sqrt{\lambda+|\xi|^2}.
\end{align}
The following inequality is used repeatedly in the paper: there exists a constant $C(\ep)<\infty$, such that for all $\lambda\in S_{\pi-\ep}$, for all $\xi\in\R^{d-1}$, 
\begin{equation*}
|\omega_\lambda(\xi)|\geq \Rel(\omega_\lambda(\xi))\geq C(|\lambda|^{\frac12}+|\xi|).
\end{equation*}
Finally, let us fix our convention for Fourier transform: for $\xi\in\mathbb R^{d-1}$
\begin{equation*}
\widehat{u}(\xi):=\int_{\mathbb R^{d-1}}e^{-i\xi\cdot x'}u(x')dx'
\end{equation*}
for $u\in \mathcal S(\R^{d-1})$. The inverse Fourier transform is defined by
\begin{equation*}
u(x'):=\frac{1}{(2\pi)^{d-1}}\int_{\mathbb R^{d-1}}e^{i\xi\cdot x'}\widehat{u}(\xi)d\xi
\end{equation*}
for $\widehat{u}\in \mathcal S(\R^{d-1})$. Both the Fourier transform and its inverse are naturally extended to $\mathcal S'(\R^{d-1})$ by duality. The definitions of the functional spaces are given in the next paragraph.

\subsection{Functional setting and notion of solutions}
\label{sec.setting}

The results of the paper take place in the class $L^p_{uloc}(\R^d_+)$ of uniformly locally $L^p$ functions. More precisely,
\begin{equation*}
L^p_{uloc}(\R^d_+):=\left\{f\in L^1_{loc}(\R^d_+)~|~ \sup_{\eta\in\mathbb Z^{d-1}\times\mathbb Z_{\geq 0}}\|f\|_{L^p(\eta+(0,1)^d)}<\infty\right\}.
\end{equation*}
Let us define the space $L^p_{uloc,\sigma} (\R^d_+)$ of solenoidal vector fields in $L^p_{uloc}$ as follows:
\begin{align}
L^p_{uloc,\sigma} (\R^d_+):= \left\{ f\in L^p_{uloc}(\R^d_+)^d ~|~ \int_{\R^d_+} f \cdot \nabla \varphi \,  d x =0~~{\rm for~any}~\varphi \in C_0^\infty (\overline{\R^d_+})\right\}.
\end{align} 
Notice that this definition encodes both the fact that $f$ is divergence-free in the sense of distributions (take test functions $\varphi\in C_0^\infty (\R^d_+)$) and the fact that $f_d$ vanishes on $\partial\R^d_+$. As usual, $WL^{p'}_{uloc}(\R^d_+)$ for $1<p'\leq\infty$ denotes the dual space of $L^{p}_{uloc}(\R^d_+)$, where $p'$ is the H\"older conjugate of $1\leq p<\infty$, $1=\frac 1p+\frac1{p'}$. It is defined in the following way:
\begin{equation*}
WL^{p'}_{uloc}(\R^d_+):=\bigg\{g\in L^1_{loc}(\R^d_+)~|~ \sum_{\eta\in\mathbb Z^{d-1}\times\mathbb Z_{\geq 0}}\|g\|_{L^{p'}(\eta+(0,1)^d)}<\infty\bigg\}.
\end{equation*}
As is usual $BUC(\R^d_+)$ denotes the space of bounded uniformly continuous functions, and
\begin{align*}
BUC_\sigma (\R^d_+) =\Big \{ f\in BUC(\R^d_+)^d~|~{\rm div}\, f =0\,, \quad f|_{x_d=0} =0\Big \}\,.
\end{align*}
Note that any function in $BUC(\R^d_+)$ is uniquely extended as a bounded uniformly continuous function on $\overline{\R^d_+}$, and thus the trace is defined as a restriction on $\partial \R^d_+$ of this extended continuous function on $\overline{\R^d_+}$.

Let us also fix the notion of solutions to \eqref{e.resol}. Let $f\in L^1_{uloc} (\R^d_+)^d$. 
We say that $(u,\nabla p)\in L^1_{uloc}(\R^d_+)^d \times L^1_{uloc}(\R^d_+)^d$ with $p\in L^1_{loc}(\R^d_+)$ is a solution to \eqref{e.resol} in the sense of distributions if 
\begin{align}\label{def.weak.u1}
\int_{\R^d_+} u \cdot (\lambda \varphi - \Delta \varphi) + \nabla p \cdot \varphi \, d x = \int_{\R^d_+} f \cdot \varphi \, d x, \quad \varphi \in C_0^\infty (\overline{\R^d_+})^d ~{\rm with}~ \varphi|_{x_d=0} =0,
\end{align}
and 
\begin{align}\label{def.weak.u2}
\int_{\R^d_+} u \cdot \nabla \phi \, d x=0,\qquad \phi \in C_0^\infty (\overline{\R^d_+}).
\end{align} 
Let us notice that the notion of solutions defined by \eqref{def.weak.u1} and \eqref{def.weak.u2} is enough to apply our uniqueness result, Theorem \ref{thm.unique}. Moreover, we emphasize that the solution $u$ of the resolvent problem \eqref{e.resol} given by Theorem \ref{prop.fresolvent} is a strong solution, thanks to the estimates \eqref{est.prop.fresolvent.1} and \eqref{est.prop.fresolvent.2}. Hence the trace of $u$ is well defined in the sense of the trace of $W^{1,p}_{loc}(\R^3_+)$ functions.

\begin{remark}\label{rem.unique}{\rm  In the definition of the solution in the sense of distributions, since $(u,\nabla p)\in L^1_{uloc}(\R^d_+)^d\times L^1_{uloc} (\R^d_+)^d$, the class of test functions is easily relaxed as follows:
$\varphi \in C^2 (\overline{\R^d_+})^d$ with $\varphi|_{x_d=0}=0$ and $\phi \in C^1 (\overline{\R^d_+})$ such that, for $\alpha=0,\, 1,\, 2$ and $\beta=0,\, 1$,
\begin{align*}
\nabla^\alpha \varphi (x), \quad \nabla^\beta \phi (x)  ~\sim ~ \mathcal{O} (|x|^{-d-\kappa}),  \qquad |x|\gg 1,
\end{align*}
for some $\kappa>0$.
}
\end{remark}

\subsection{Integral representation formulas for the resolvent system}
\label{sec.intrepr} 

The solution to the resolvent problem \eqref{e.resol} can be computed in Fourier space. We build on the formulas for the symbols, which were derived in \cite{DHP01}. In particular in this paper, the authors show that the symbol associated with the Dirichlet Stokes resolvent problem can be decomposed into: one part corresponding to the symbol of the Dirichlet-Laplace resolvent problem and a remainder term due to the pressure. Hence, the solution in the Fourier side about the tangential variables can be decomposed into $\widehat{u}=\widehat{v} + \widehat{w}$, with for all $\xi\in\R^{d-1}$, $y_d>0$,
\begin{subequations}\label{e.vw}
\begin{equation}
\widehat{v}(\xi,y_d)
=\frac{1}{2\omega_\lambda(\xi)} \int^{\infty}_0 \left(e^{-\omega_\lambda(\xi)|y_d-z_d|} -  e^{-\omega_\lambda(\xi)(y_d+z_d)}\right)
\widehat{f}(\xi,z_d)dz_d,
\end{equation}
\begin{equation}
\widehat{w'}(\xi,y_d)
=\int^{\infty}_0 \frac{\xi}{|\xi |}
\frac{e^{-|\xi|y_d} -  e^{-\omega_\lambda(\xi)y_d}} {\omega_\lambda(\xi)(\omega_\lambda(\xi)-|\xi|)}
\,e^{-\omega_\lambda(\xi)z_d} \xi \cdot \widehat{f}'(\xi,z_d)dz_d
\end{equation}
and
\begin{equation}
\widehat{w_d}(\xi,y_d)=i\int^{\infty}_0
\frac{e^{-|\xi|y_d} -  e^{-\omega_\lambda(\xi)y_d}} {\omega_\lambda(\xi)(\omega_\lambda(\xi)-|\xi|)} e^{-\omega_\lambda(\xi)z_d}\xi \cdot \widehat{f}'(\xi,z_d)dz_d
\end{equation}
The solution of the form $u=v+w$ is then obtained by taking the inverse Fourier transform.
Notice that $v$ is the solution to the Dirichlet-Laplace resolvent problem, while 
the remainder term $w$ comes from the contribution of the nonlocal pressure term. The above formulas are derived for $f\in C_0^\infty (\R^d_+)^d$ satisfying $\nabla\cdot f =0$ in $\R^d_+$ and $f_d=0$ on $\partial\R^d_+$. But as is seen below, these formulas are also well-defined for any $L^p_{uloc}$ function $f$. If moreover $f$ is solenoidal, i.e. $\nabla\cdot f=0$ in the sense of distributions and $f_d=0$ on $\partial \R^d_+$ in the sense of the generalized trace (see \cite[(III.2.14) p. 159]{Galdi_book}), the velocity $u$, together with the pressure $p$ defined below, is a solution to \eqref{e.resol}. 

The formula for the pressure in the Fourier variables is written as follows: for all $\xi\in\R^{d-1}$, $y_d>0$,
\begin{align}\label{pressure.velocity.1}
\widehat{p}(\xi,y_d) & =-\int_0^\infty e^{-|\xi|y_d}e^{-\omega_\lambda(\xi)z_d}\frac{\omega_\lambda(\xi)+|\xi|}{|\xi|}\widehat{f}_d(\xi,z_d)dz_d\\
& =\int_0^\infty e^{-|\xi|y_d}e^{-\omega_\lambda(\xi)z_d}\left(\frac{1}{|\xi|}+\frac{1}{\omega_\lambda(\xi)}\right)i\xi\cdot\widehat{f}'(\xi,z_d)dz_d.\nonumber
\end{align}
\end{subequations}
Another useful representation of $\widehat{p}$ is 
\begin{subequations}\label{pressure.velocity.2}
\begin{align}
\widehat{p}(\xi, y_d) = -\frac{i \xi}{|\xi|} e^{-|\xi| y_d} \cdot \partial_{y_d} \widehat{u}' (\xi,0),
\end{align}
which in particular leads to 
\begin{align}\label{pressure.velocity.3'}
i\xi_j \widehat{p}(\xi,y_d) & = \frac{\xi_j \xi}{|\xi|} e^{-|\xi|y_d} \cdot \partial_{y_d} \widehat{u'}(\xi,0),\\
\label{pressure.velocity.3''}
\partial_{y_d} \widehat{p} (\xi,y_d) & = i\xi e^{-|\xi| y_d} \cdot \partial_{y_d} \widehat{u'} (\xi,0),
\end{align} 
\end{subequations}
for all $y_d>0$. Formula \eqref{pressure.velocity.2} is important when one deals with the nondecaying solutions.
Indeed, it excludes the flow driven by the pressure, that is, the pressure is completely determined by the velocity. Notice that such a formula rules out the parasitic linearly growing solutions to the pressure equation (see the Liouville theorem, Theorem \ref{thm.unique} in the Appendix).
By using the integration by parts the formula  \eqref{pressure.velocity.3'} is also written as 
\begin{subequations}
\begin{align}
\label{pressure.velocity.4'}
i\xi_j \widehat{p}(\xi,y_d) & = \frac{\xi_j \xi}{|\xi|} \cdot \int_0^\infty e^{-|\xi|y_d} e^{-\omega_\lambda (\xi) z_d} \bigg ( \omega_\lambda (\xi) \partial_{z_d} \widehat{u'}(\xi,z_d) - \partial_{z_d}^2 \widehat{u'}(\xi,z_d) \bigg ) d z_d \nonumber \\
& = \frac{\xi_j \xi}{|\xi|} \cdot \int_0^\infty e^{-|\xi|y_d} e^{-\omega_\lambda (\xi) z_d} \bigg ( \omega_\lambda (\xi)^2 \widehat{u'}(\xi,z_d) - \partial_{z_d}^2 \widehat{u'}(\xi,z_d) \bigg ) d z_d,
\\
\label{pressure.velocity.5'}
\partial_{y_d} \widehat{p} (\xi,y_d) & = i\xi \cdot \int_0^\infty e^{-|\xi| y_d} e^{-\omega_\lambda (\xi) z_d} \bigg ( \omega_\lambda (\xi)^2  \widehat{u'} (\xi,0) - \partial_{z_d}^2 \widehat{u'}(\xi,z_d) \bigg ) d z_d.
\end{align}
\end{subequations} 
The expression \eqref{pressure.velocity.4'} is useful in obtaining the characterization of the domain of the Stokes operator in $L^q_{uloc}$ spaces; see Proposition \ref{prop.domain}.

We now define the kernels $k_{1,\lambda}$:\, $\R^d \rightarrow \mathbb{C}$ 
and $k_{2,\lambda}$:\, $\R^d \times \R_+ \rightarrow \mathbb{C}$ associated with the Dirichlet-Laplace part by: for all $y'\in\R^{d-1}$ and $y_d\in\R$,
\begin{subequations}
\begin{equation}\label{def:k_1}
k_{1,\lambda}(y',y_d):=\int_{\mathbb R^{d-1}} e^{iy'\cdot \xi} \frac{1}{2\omega_\lambda(\xi)} 
e^{-\omega_\lambda(\xi)|y_d|}d\xi,
\end{equation}
and for all $y'\in\R^{d-1}$ and $y_d,\ z_d>0$,
\begin{equation}\label{def:k_2}
k_{2,\lambda}(y',y_d,z_d):=\int_{\mathbb R^{d-1}} e^{iy'\cdot \xi} \frac{1}{2\omega_\lambda(\xi)} 
e^{-\omega_\lambda(\xi)(y_d+z_d)}d\xi.
\end{equation}
We also define the kernels $r'_\lambda$: 
$\R^{d}_+ \times \R_+ \rightarrow \mathbb{C}^{d-1}$
and $r_{d,\lambda}$: 
$\R^{d}_+ \times \R_+ \rightarrow \mathbb{C}$
associated with the nonlocal part as follows: for all $y'\in\R^{d-1}$ and $y_d,\ z_d>0$ 
\begin{align}
r'_\lambda(y',y_d,z_d) &:= 
\int_{\R^{d-1}} e^{iy'\cdot \xi} 
 \frac{e^{-|\xi| y_d} -  e^{-\omega_\lambda(\xi)y_d}} {\omega_\lambda(\xi)(\omega_\lambda(\xi)-|\xi|)}
\,e^{-\omega_\lambda(\xi)z_d} \frac{\xi\otimes\xi}{|\xi|}d\xi,\label{e.defgreenkernels}
\\
r_{d,\lambda}(y',y_d,z_d) &:=
\int_{\R^{d-1}} ie^{iy'\cdot \xi} 
 \frac{e^{-|\xi|y_d} -  e^{-\omega_\lambda(\xi)y_d}} {\omega_\lambda(\xi)(\omega_\lambda(\xi)-|\xi|)}
\,e^{-\omega_\lambda(\xi)z_d} \xi d\xi.\label{e.defgreenkernels'}
\end{align}
Moreover, the kernel associated with the pressure is defined as follows: for all $y'\in\R^{d-1}$ and $y_d,\ z_d>0$
\begin{align}\label{e.defqlambda}
q_\lambda(y',y_d,z_d):=i\int_{\R^{d-1}}e^{iy'\cdot\xi}e^{-|\xi|y_d}e^{-\omega_\lambda(\xi)z_d}\left(\frac{\xi}{|\xi|}+\frac{\xi}{\omega_\lambda(\xi)}\right) d\xi.
\end{align}
\end{subequations}
Notice that  for all $y'\in\R^{d-1}$, for all $y_d>0$,
\begin{subequations}\label{def:vw}
\begin{multline}
v(y',y_d)=\int_{\R^{d-1}}\int^{\infty}_0 k_{1,\lambda} (y'-z', y_d-z_d)f(z',z_d)dz_ddz'\\+\int_{\R^{d-1}}\int^{\infty}_0 k_{2,\lambda} (y'-z', y_d, z_d)f(z',z_d)dz_d dz',
\end{multline}
\begin{align}
w'(y',y_d)&=\int_{\R^{d-1}}\int^{\infty}_0 r'_\lambda(y'-z',y_d,z_d)f'(z',z_d)dz_ddz',\\
w_d(y',y_d)&=\int_{\R^{d-1}}\int^{\infty}_0 r_{d,\lambda}(y'-z',y_d,z_d)\cdot f'(z',z_d)dz_ddz'.
\end{align}
and
\begin{equation}
p(y',y_d)=\int_{\R^{d-1}}\int^{\infty}_0 q_{\lambda}(y'-z',y_d,z_d)\cdot f'(z',z_d)dz_ddz'.
\end{equation}
\end{subequations}
These integral representation formulas, together with pointwise estimates on the kernels, are the bases for estimate in  $L^p_{uloc}$ spaces.

Note that, in view of \eqref{pressure.velocity.2}, the pressure $\nabla p$ is also written as 
\begin{subequations}\label{pressure.velocity.3}
\begin{align}
\nabla' p (\cdot,y_d) & = - (\nabla' \nabla' (-\Delta')^{-\frac12}) \cdot P (y_d) \gamma \partial_{y_d} u' , \qquad y_d>0,\\
\partial_{y_d} p (\cdot,y_d) & = \nabla' \cdot P (y_d) \gamma \partial_{y_d} u' , \qquad y_d>0.
\end{align}
\end{subequations}
Here $\gamma$ is the trace operator on $\partial\R^d_+$ and $P(t)$ is the Poisson semigroup whose kernel is the Poisson kernel defined by $\mathcal{F}^{-1} [e^{-|\xi| t} ]$.
We note that, when $\gamma \partial_{y_d} u'$ belongs to $L^q_{uloc}(\R^{d-1})^{d-1}$ for some $q\in [1,\infty]$, the function $P(y_d)\gamma \partial_{y_d} u'$ is smooth and bounded including their derivatives in $\R^d_{+,\delta} = \{ (y',y_d)\in \R^d~|~y_d >\delta\}$ for each $\delta>0$. This can be proved from the pointwise estimate of the Poisson kernel (and its derivatives) and we omit the details here. 
Then, the action of $\nabla' \nabla' (-\Delta')^{-\frac12}$ or $\nabla'$ on $P (y_d) \gamma \partial_{y_d} u'$ makes sense for each $y_d>0$, when $\gamma \partial_{y_d} u' \in L^q_{uloc}(\R^{d-1})^{d-1}$.
Indeed, one natural way to realize the action of $\nabla' \nabla' (-\Delta')^{-\frac12}$ is to define it as 
\begin{align*}
\nabla' \nabla' (-\Delta')^{-\frac12} f = \int_0^\infty \nabla' \nabla' P (t) f \, d t,
\end{align*}
which is well-defined for any function $f\in BC^{2}(\R^{d-1})$ (or more sharply, for any $f\in C^{1+\ep}(\R^{d-1})$, $\ep>0$). The formula \eqref{pressure.velocity.3} will be used in Section \ref{sec.semigroup}.

We end this section with the following scaling properties of the kernels, which will be used to work with $|\lambda|=1$: for all $y'\in\R^{d-1}$, $y_d\in\R$,
\begin{subequations}\label{eq.scaling}
\begin{equation}
k_{1,\lambda}(y',y_d)=|\lambda|^{\frac{d}{2}-1}k_{1,\frac{\lambda}{|\lambda|}}(|\lambda|^{\frac{1}{2}}y',|\lambda|^{\frac{1}{2}}y_d)
\end{equation}
and for all $y'\in\R^{d-1}$, $y_d,\, z_d\in\R_+$,
\begin{align}
k_{2,\lambda}(y',y_d,z_d)&=|\lambda|^{\frac{d}{2}-1}k_{2,\frac{\lambda}{|\lambda|}}(|\lambda|^{\frac{1}{2}}y',|\lambda|^{\frac{1}{2}}y_d,|\lambda|^{\frac{1}{2}}z_d),\\
r'_\lambda(y', y_d,z_d)&
=|\lambda|^{\frac{d}{2}-1}r'_{\frac{\lambda}{|\lambda|}}(|\lambda|^{\frac{1}{2}}y',|\lambda|^{\frac{1}{2}}y_d,|\lambda|^{\frac{1}{2}}z_d),\\
r_{d,\lambda}(y', y_d,z_d)&=|\lambda|^{\frac{d}{2}-1}r_{d,\frac{\lambda}{|\lambda|}}(|\lambda|^{\frac{1}{2}}y',|\lambda|^{\frac{1}{2}}y_d,|\lambda|^{\frac{1}{2}}z_d),\\
q_\lambda(y',y_d,z_d)&=|\lambda|^{\frac{d-1}{2}}q_{\frac\lambda{|\lambda|}}(|\lambda|^{\frac{1}{2}}y',|\lambda|^{\frac{1}{2}}y_d,|\lambda|^{\frac{1}{2}}z_d).
\end{align}
\end{subequations}
There is no evident characterization of $L^p_{uloc}$ spaces in Fourier space. These spaces are easily defined in physical space. Therefore, a prominent task is to derive pointwise estimates on the kernels. The goal of the next section is to address this task.

\section{Pointwise kernel estimates}
\label{sec.resolhp}

Deriving pointwise bounds for the Dirichlet-Laplace part is rather classical. The nonlocal part requires a more refined analysis. 

\subsection{General ideas for the estimates}
\label{sec.galideas}

Before starting the estimates of the kernels derived in Section \ref{sec.prelim}, we give some general remarks, which serve as guidelines for this section. First, we always start by using the formulas \eqref{eq.scaling} in order to make $|\lambda|=1$. Second, integrability of derivatives of the Fourier multipliers are traded in decay of the kernels in physical space in the tangential direction. This is the role of Lemma \ref{lem.optdecay} below, which is central in our approach. Third, the analysis of the integrability of derivatives of the Fourier multipliers requires sometimes to analyze separately the low frequencies from the high frequencies, or small $y_d$ from large $y_d$. More heuristic explanations are given in \cite{P18Xedp}.

The following lemma is standard. Since we use it repeatedly, we state and prove it here.

\begin{lemma}\label{lem.optdecay}
Let $m\in C^\infty(\R^{d-1}\setminus\{0\})$ be a smooth Fourier multiplier. Let $K$ be the kernel associated with $m$: for all $y'\in\R^{d-1}$,
\begin{equation*}
K(y'):=\int_{\R^{d-1}}e^{iy'\cdot\xi}m(\xi)d\xi.
\end{equation*}
Assume that there exists $n>-d+1$, and positive constants $c_0(d,n,m),\ C_0(d,n,m)<\infty$ such that for all $\alpha\in\mathbb N$, $0\leq \alpha\leq n+d$, for all $\xi\in \R^{d-1}\setminus\{0\}$
\begin{equation}\label{e.boundm}
|\nabla^\alpha m(\xi)|\leq C_0|\xi|^{n-\alpha}e^{-c_0|\xi|}\quad \mbox{if}\ n>1-d,\\
\end{equation}
Then, there exists a constant $C(d,n,c_0,C_0)<\infty$ such that for all $y'\in\R^{d-1}\setminus\{0\}$,
\begin{equation}\label{e.optdecay}
|K(y')|\leq \frac{C}{|y'|^{n+d-1}}.
\end{equation}
\end{lemma}

In the definition of $K$ the integral is, as usual, considered as the oscillatory integral.
Note that $n$ is not needed to be an integer.
The lemma will be typically used to get bounds on the kernel when $y'$ is large, say $|y'|\geq1$. Let us now give the proof of the lemma.

\begin{proof}
The proof is by integration by parts. There are two steps. For the first step, due to the singularity of the multiplier $m$ at $0$, we can integrate by parts $[n+d-2]$ times, which yields the decay $|K(y')|\leq \frac{C}{|y'|^{n+d-2}}$. Here $[a]$ denotes the Gauss symbol, i.e., $[a]$ is the integer such that $a=[a]+\delta$ with $\delta\in [0,1)$. In the second step, we cut-off the singularity around $0$ at an ad hoc frequency $R$, and integrate by parts two more times in high frequencies. This makes it possible to get the optimal decay stated in display \eqref{e.optdecay}. Let $y'\in\R^{d-1}\setminus\{0\}$ be fixed.

\noindent{\emph{\underline{Step 1.}}} There is $j\in \{1,\cdots\, d-1\}$ such that $|y_j|\geq |y'|/(d-1)$. For such $j$ we have from the integration by parts, 
\begin{align*}
&(-iy_j)^{[n+d-2]} K(y') = \int_{\R^{d-1}}  e^{iy'\cdot \xi} 
\partial_{\xi_j}^{[n+d-2]} m(\xi) d\xi\\
& = \int_{\R^{d-1}} \chi_R  e^{iy'\cdot \xi} 
\partial_{\xi_j}^{[n+d-2]} m(\xi) d\xi + \int_{\R^{d-1}} (1-\chi_R) e^{iy'\cdot \xi} 
\partial_{\xi_j}^{[n+d-2]} m(\xi) d\xi \\
& =: I_R + II_R.
\end{align*}
Here $R\in (0,\infty)$ is fixed (and will be chosen below) and $\chi_R\in C_0^\infty (\R^{d-1})$ is a smooth radial cut-off function such that $\chi_R =1$ for $|\xi|\leq R$ and $\chi_R =0$ for $|\xi|\geq 2R$. Notice that using the bound \eqref{e.boundm} we get
\begin{equation*}
\left|(-i y_j)^{[n+d-2]} K(y')\right|\leq C_0\int_{\R^{d-1}}|\xi|^{n-[n+d-2]}e^{-c_0|\xi|}d\xi\leq C,
\end{equation*}
which is not optimal.

\noindent{\emph{\underline{Step 2.}}} We have 
\begin{align*}
|I_R |\leq C_0\int_{|\xi|\leq 2R} |\xi|^{n-[n+d-2]} d \xi \leq C_0R^{1+\delta},\qquad n+d-2=[n+d-2]+\delta,
\end{align*}
while
\begin{align*}
|(-i y_j)^2 II_R | & = \left|\int_{\R^{d-1}} e^{iy'\cdot \xi}  \partial_{\xi_j}^2 \bigg ( (1-\chi_R) \partial_{\xi_j}^{[n+d-2]} m(\xi) \bigg ) d\xi \right|\\
& \leq C_0 \int_{|\xi|\geq R} |\xi|^{n-[n+d-2]-2} e^{-c_0|\xi|} d\xi \leq C_0 R^{-1+\delta}.
\end{align*} 
Now we take $R=|y'|^{-1}$, which yields from $|y_j|\geq |y'|/(d-1)$,
\begin{align*}
|K(y')| \leq | y_j|^{-[n+d-2]} ( |I_R| + |II_R| ) \leq C |y'|^{-n-d+1}.
\end{align*}
This concludes the proof.
\end{proof}

\subsection{Kernel estimates for the Dirichlet-Laplace part}
\label{sec.dirlap}

The Dirichlet Laplace part of the operator is nothing but the part corresponding to the resolvent problem for the scalar Laplace equation in $\R^d_+$ with Dirichlet boundary conditions. The kernels 
$k_{1,\lambda}$ and $k_{2,\lambda}$ are given by the expressions \eqref{def:k_1} and \eqref{def:k_2} respectively. 
We recall and prove the following classical pointwise bounds.

\begin{proposition}
\label{l.kernel}
Let $\lambda\in S_{\pi-\ep}$. There exist constants $c(d,\varepsilon)$, $C(d,\varepsilon)<\infty$ such that 
 for  $y' \in \R^{d-1}$, \ $y_d \in\R$,
\begin{align} 
&|k_{1,\lambda}(y',y_d)| \nonumber\\
&\le 
\left\{\begin{array}{ll}
C e^{-c|\lambda|^{\frac{1}{2}}|y_d|}\min \!\left\{ 
\log \left(e+\displaystyle{\frac{1}{|\lambda|^{\frac12}(|y_d|+|y'|)}}\right), \,
 \displaystyle{\frac{1}{|\lambda|(|y_d|+|y'|)^{2}}}  \right\}, & d=2, 
\\
\displaystyle{\frac{C e^{-c|\lambda|^{\frac{1}{2}}|y_d|}}{(|y_d|+|y'|)^{d-2}(1+|\lambda|^{\frac{1}{2}}(|y_d|+|y'|))^{2}}},& d\ge 3,
\end{array}\right.
\label{est:k_1}
\end{align}
and for $\alpha \in \N$,
\begin{align}
|\nabla k_{1,\lambda}(y',y_d)| &\le 
\frac{Ce^{-c|\lambda|^{\frac{1}{2}}|y_d|}}{
(|y'|+|y_d|)^{d-1}(1+|\lambda|^{\frac{1}{2}}(|y_d|+|y'|))^{2}},
\label{est:dk_1}
\\
|\nabla^{\alpha} k_{1,\lambda}(y',y_d)| &\le 
\frac{Ce^{-c|\lambda|^{\frac{1}{2}}|y_d|}}{ 
(|y_d|+|y'|)^{d-2+\alpha}(1+|\lambda|^{\frac{1}{2}}(|y_d|+|y'|))}.
\label{est:dk_12}
\end{align}
Moreover we have for  $y' \in \R^{d-1}$, \ $y_d,\, z_d\in\R_+$,  $\alpha \in \N$, 
\begin{align} 
&|k_{2,\lambda}(y',y_d,z_d)|  \nonumber\\
&\le 
\left\{\begin{array}{l}
C e^{-c|\lambda|^{\frac{1}{2}}(y_d+z_d)}\min \!\left\{ 
\log \left(e+\displaystyle{\frac{1}{|\lambda|^{\frac12}(y_d+z_d+|y'|)}}\right), \,
 \displaystyle{\frac{1}{|\lambda|(y_d+z_d+|y'|)^{2}}}  \right\}, \\
 \qquad \qquad \qquad  \qquad \qquad \qquad \qquad \qquad \qquad \qquad \qquad \qquad\qquad\qquad d=2, \\
\displaystyle{\frac{C e^{-c|\lambda|^{\frac{1}{2}}(y_d+z_d)}}{(y_d+z_d+|y'|)^{d-2}(1+|\lambda|^{\frac{1}{2}}(y_d+z_d+|y'|))^{2}}}, \qquad d\geq 3, \\
\end{array}\right.
\label{est:k_2}
\end{align}
and for $\alpha \in \N$,
\begin{align}
|\nabla k_{2,\lambda}(y',y_d,z_d)| &\le 
\frac{Ce^{-c|\lambda|^{\frac{1}{2}}(y_d+z_d)}}{
(y_d+z_d+|y'|)^{d-1}(1+|\lambda|^{\frac{1}{2}}(y_d+z_d+|y'|))^{2}},
\label{est:dk_2}
\\
|\nabla^{\alpha} k_{2,\lambda}(y',y_d,z_d)| &\le 
\frac{Ce^{-c|\lambda|^{\frac{1}{2}}(y_d+z_d)} }{
(y_d+z_d+|y'|)^{d-2+\alpha}(1+|\lambda|^{\frac{1}{2}}(y_d+z_d+|y'|))}.\label{est:d2k_2}
\end{align}
\end{proposition}
\begin{remark}
\label{rem:k_2}
{\rm From \eqref{est:k_2}, it is clear that
\begin{align} 
&|k_{2,\lambda}(y',y_d,z_d)|  \nonumber\\
&\le 
\left\{\begin{array}{l}
C e^{-c|\lambda|^{\frac{1}{2}}|y_d-z_d|}\min \!\left\{ 
\log \left(e+\displaystyle{\frac{1}{|\lambda|^{\frac12}(|y_d-z_d|+|y'|)}}\right), \,
 \displaystyle{\frac{1}{|\lambda|(|y_d-z_d|+|y'|)^{2}}}  \right\}, \\
 \qquad \qquad \qquad  \qquad \qquad \qquad \qquad \qquad \qquad \qquad \qquad \qquad\qquad\qquad d=2, \\
\displaystyle{\frac{C e^{-c|\lambda|^{\frac{1}{2}}|y_d-z_d|}}{(|y_d-z_d|+|y'|)^{d-2}\left(1+|\lambda|^{\frac{1}{2}}(|y_d-z_d|+|y'|)\right)^{2}}}, \qquad d\geq 3, \\
\end{array}\right.
\end{align}
and similar estimates hold for the derivatives. Hence the integral operator associated with $k_{2, \lambda}$ 
can be estimated 
as a convolution kernel in $\R^d$ as  $k_{1, \lambda}$.}
\end{remark}

\begin{proof} Since these bounds can be estimated in the similar way, we only deal with \eqref{est:k_1}. 
The scaling property \eqref{eq.scaling} allows us to 
assume $|\lambda|=1$ in the following argument. 

We begin with the case $d\ge 3$.  First observe that 
\begin{align}
|k_{1,\lambda} (y',y_d)| 
 &\le C\int_{\R^{d-1}} \frac{e^{-c(1+|\xi|)y_d}}{|\xi|}d\xi
\label{est:y_dlarge}
 \\
 &\le Ce^{-cy_d}y_d^{-(d-2)}.
\notag
\end{align}
Secondly, for all $\alpha\in\N$ we have the following pointwise bound: 
\begin{equation}
\left| \nabla_\xi^\alpha
\left( \frac{1}{\omega_\lambda(\xi)} e^{-\omega_\lambda(\xi)|y_d|}\right)
\right|
\le
\frac{Ce^{-cy_d}e^{-c_0|\xi|y_d}}{(1+|\xi|)^{\alpha+1}}.
\label{est:dkernel}
\end{equation}
Therefore, applying Lemma \ref{lem.optdecay} with 
\begin{equation*}
m(\xi):=\frac{1}{\omega_\lambda(\xi)} e^{-\omega_\lambda(\xi)|y_d|},
\end{equation*}
$n=-1$ and $C_0:=Ce^{-cy_d}$ (remember that in this computation $y_d$ is a parameter),  we get 
\begin{equation*}
|k_{1,\lambda}(y',y_d)| \leq Ce^{-cy_d} |y'|^{-(d-2)}.
\end{equation*}
Combining the previous estimate with \eqref{est:y_dlarge} yields
\begin{align}
|k_{1,\lambda}(y',y_d)| 
\le  Ce^{-cy_d}(|y'|+y_d)^{-(d-2)}.
\label{est:k_11}
\end{align}
In the same way as above, the integration by parts and \eqref{est:dkernel} give 
\begin{align*}
|(-iy_j)^{d}k_{1,\lambda}(y',y_d)| 
&=\left|\int_{\mathbb R^{d-1}} e^{iy'\cdot \xi} \partial_{\xi_j}^{d}
\left( \frac{1}{2\omega_\lambda(\xi)} e^{-\omega_\lambda(\xi)|y_d|}\right)d\xi\right|
\\
&\le Ce^{-cy_d} \int_{\mathbb R^{d-1}} 
\frac{1}{(1+|\xi|)^{d+1}}d\xi
\\
&\le
Ce^{-cy_d}, 
\end{align*}
which together with \eqref{est:y_dlarge} implies
\begin{align} 
|k_{1,\lambda}(y',y_d)| 
&\le 
Ce^{-cy_d} (|y'|+y_d)^{-d}.
\label{est:ylarge}
\end{align} 
Combining this with \eqref{est:k_11}, 
we obtain the desired estimate \eqref{est:k_1} for $d\ge 3$.

For the case $d=2$,  it suffices to show  
\begin{align}
\label{est:d=2}
|k_{1,\lambda}(y',y_d)| &\le
C e^{-c|\lambda|^{\frac{1}{2}}|y_d|} 
\log (e+e|\lambda|^{-\frac12}(|y'|+|y_d|)^{-1}), 
\end{align}
since the other case in \eqref{est:k_1} for $d=2$
can be shown in the same manner as in \eqref{est:ylarge}.
Splitting the integral as in the proof of Lemma \ref{lem.optdecay}, we have 
\begin{align*}
k_{1,\lambda}(y',y_d) 
&= \int_{\mathbb R} \frac{1}{2\omega_\lambda(\xi)} e^{-\omega_\lambda(\xi)|y_d|}d\xi
\\
&= 
 \int_{\mathbb R} e^{iy'\cdot \xi}
  \frac{1}{2\omega_\lambda(\xi)} e^{-\omega_\lambda(\xi)|y_d|}
  \chi_R(\xi)d\xi
\\
& \quad +
 \int_{\mathbb R} e^{iy'\cdot \xi} 
 \frac{1}{2\omega_\lambda(\xi)} e^{-\omega_\lambda(\xi)|y_d|}(1-\chi_R(\xi))d\xi
\\
&=I+II.
\end{align*}
On the one hand \eqref{est:dkernel} with $\alpha=0$ gives
\begin{align*}
|I| &\le C\int_{|\xi'| \le 2R} 
  \frac{1}{1+|\xi|} e^{-(1+|\xi|)y_d} d\xi
\\
&\le 
Ce^{-cy_d}\int_{|\xi'| \le 2R} 
  \frac{1}{1+|\xi|} d\xi
\\
&\le 
Ce^{-cy_d}\log (1+R),
\end{align*}
and on the other hand, using the identity $i y_j e^{iy'\cdot \xi}=\partial_{\xi_j} e^{iy'\cdot \xi}$, an integration by parts and \eqref{est:dkernel} we obtain
\begin{align*}
|II| 
&\le
\frac{C}{|y'|} 
\left|\int_{\mathbb R} e^{iy'\cdot \xi} 
\partial_{\xi'}
\left( \frac{1}{2\omega_\lambda(\xi)} e^{-\omega_\lambda(\xi)|y_d|}(1-\chi_R(\xi))\right)
d\xi\right|
\\
&\le 
Ce^{-cy_d}(R|y'|)^{-1},
\end{align*}
from which we have
\begin{align*}
|k_{1,\lambda}(y',y_d)| \le Ce^{-cy_d} (\log (1+R)+(|y'|R)^{-1}).
\end{align*}
Hence taking  $R=|y'|^{-1}$ we obtain
\begin{align*}
|k_{1,\lambda}(y',y_d)| 
\le Ce^{-cy_d}(\log (1+|y'|^{-1})+1).
\end{align*}
Moreover, we have
\begin{align*}
|k_{1,\lambda}(y',y_d)|
&\le
 \left|\int_0^{y_d^{-1}} e^{iy'\cdot \xi}
  \frac{1}{2\omega_\lambda(\xi)} e^{-\omega_\lambda(\xi)|y_d|}
  d\xi\right|
\\
& \quad +
\left| \int_{y_d^{-1}}^{\infty} e^{iy'\cdot \xi} 
 \frac{1}{2\omega_\lambda(\xi)} e^{-\omega_\lambda(\xi)|y_d|}d\xi\right|
\\
&\le 
C
e^{-cy_d}\left( \int_0^{y_d^{-1}} 
  \frac{1}{1+|\xi|}d\xi
+ \int_{y_d^{-1}}^{\infty}  
  \frac{e^{-|\xi|y_d}}{|\xi|}   d\xi
  \right)
 \\
 &\le C e^{-cy_d} (\log (1+y_d^{-1}) +1).
 \end{align*}
Combining both cases, we obtain 
\begin{align}
|k_{1,\lambda}(y',y_d)| 
\le Ce^{-cy_d}(\log (1+(|y'|+y_d)^{-1})+1),
\end{align}
which immediately implies the desired estimate \eqref{est:d=2}.
This completes the proof of \eqref{est:k_1}.
\end{proof}

\begin{remark}[On the estimate of the tangential derivatives]\label{rem.tanderiv}
{\rm The tangential derivatives of $k_1$ or $k_2$ should a priori be better behaved than the vertical derivatives in $y_d$ or $z_d$, since differentiating in $y'$ brings a $\xi$ in the symbol. We were however unable to get an estimate of the type
\begin{equation*}
|\nabla^{2}_{y'} k_{1,\lambda}(y',y_d)| \le 
\frac{C|y_d|e^{-c|\lambda|^{\frac{1}{2}}|y_d|}}{ 
(|y_d|+|y'|)^{d+1}},
\end{equation*}
on the contrary of $\nabla_{y'}^2r'_\lambda$ for which this is true (see below \eqref{r.bound.2'}). A pointwise bound such as \eqref{r.bound.2'} makes it possible to prove uniform bounds in $\lambda$ on second order tangential derivatives in $L^q_{uloc}$, without loss of a factor $\log |\lambda|$ for small $|\lambda|$ (compare \eqref{e.unif2y'} to \eqref{est.nonlocal.18'}).
On a different note, we also note that the argument above also provide the estimate for the fractional derivative in the tangential variables. Indeed, if $m_\alpha (D')$ is any Fourier multiplier homogeneous of order $\alpha>0$, then we have for $\beta=0,1$,
\begin{align}\label{k.bound.frac}
\begin{split}
| m_\alpha (D') \nabla^\beta k_{1,\lambda}(y',y_d)|  &\le 
\frac{Ce^{-c|\lambda|^{\frac{1}{2}}|y_d|}}{ 
(|y_d|+|y'|)^{d-2+\alpha+\beta}(1+|\lambda|^{\frac{1}{2}}(|y_d|+|y'|))},\\
 | m_\alpha (D') \nabla^\beta k_{2,\lambda} (y',y_d,z_d)| &\le 
\frac{Ce^{-c|\lambda|^{\frac{1}{2}}(y_d+z_d)}}{ 
(y_d+z_d+|y'|)^{d-2+\alpha+\beta}(1+|\lambda|^{\frac{1}{2}}(y_d+z_d + |y'|))}.
\end{split}
\end{align}
}
\end{remark}

\subsection{Kernel estimates for the nonlocal part}
\label{sec.kernelnonlocal}

We now consider the nonlocal part $w$. We estimate the kernels 
$r'_{\lambda}$ and $r_{d,\lambda}$
defined by \eqref{e.defgreenkernels} and \eqref{e.defgreenkernels'} respectively. The nonlocal effects are due to the pressure of the Stokes equations. This part is the most difficult one. As above, our aim is to get pointwise estimates on the kernels following the general guidelines of Section \ref{sec.galideas}. We summarize our results in the following proposition.

\begin{proposition}\label{prop.estkernel}
Let $\lambda\in S_{\pi-\ep}$. There exist positive constants $c(d,\ep)$, $C(d,\ep)<\infty$  such that for all $y'\in\R^{d-1}$, $y_d,\, z_d>0$,
\begin{equation}\label{r.bound.1'}
\begin{split}
& \left|r'_\lambda(y',y_d,z_d)\right|+\left|r_{d,\lambda}(y',y_d,z_d)\right|\\
& \qquad \leq \frac{C y_d}{(y_d + z_d +|y'|)^{d-1}}\frac{e^{-c |\lambda|^{\frac 12}z_d}}{\big (1+|\lambda|^{\frac 12}(y_d+z_d + |y'|)\big ) \big (1+|\lambda|^{\frac 12}(y_d+z_d)\big ) }.
\end{split}
\end{equation}
Moreover, for $\alpha=1,\, 2$, 
\begin{align}\label{r.bound.2'}
\begin{split}
& \left| \nabla_{y'}^\alpha  r'_\lambda(y',y_d,z_d)\right|+\left|  \nabla_{y'}^\alpha r_{d,\lambda}(y',y_d,z_d)\right| \\
&\qquad  \leq \frac{Cy_d}{(y_d + z_d +|y'|)^{d-1+\alpha}}\frac{e^{-c|\lambda|^{\frac 12}z_d}}{\big (1+|\lambda|^{\frac 12}(y_d+z_d + |y'|)\big )\big (1+|\lambda|^{\frac 12}(y_d+z_d)\big )},
\end{split}
\end{align}
and for $\beta=0,\, 1$,
\begin{align}\label{r.bound.3'}
\begin{split}
& \left|  \nabla_{y'}^\beta \partial_{y_d} r'_\lambda(y',y_d,z_d)\right|+\left| \nabla_{y'}^\beta \partial_{y_d} r_{d,\lambda}(y',y_d,z_d)\right| \\
&\qquad  \leq \frac{C}{(y_d + z_d +|y'|)^{d-1+\beta}}\frac{e^{-c|\lambda|^{\frac 12}z_d}}{\big (1+|\lambda|^{\frac 12}(y_d+z_d + |y'|)  \big ) \big ( 1+ |\lambda|^\frac12 ( y_d +z_d) \big )},
\end{split}
\end{align}
and 
\begin{align}\label{r.bound.4'}
\begin{split}
& \left| \partial_{y_d}^2 r'_\lambda(y',y_d,z_d)\right|+\left| \partial_{y_d}^2 r_{d,\lambda}(y',y_d,z_d)\right|  \leq \frac{C e^{-c|\lambda|^{\frac 12}z_d}}{(y_d + z_d +|y'|)^{d} \big (1+|\lambda|^\frac12 (y_d +z_d)\big )}.
\end{split}
\end{align}
Finally, for $\beta=0,\, 1$,
\begin{align}\label{r.bound.5'}
\begin{split}
& \left|\nabla_{y'}^\beta \partial_{z_d}  r'_\lambda(y',y_d,z_d)\right|+\left|\nabla_{y'}^\beta \partial_{z_d} r_{d,\lambda}(y',y_d,z_d)\right| \\
&\qquad  \leq \frac{Cy_d}{(y_d + z_d +|y'|)^{d+\beta}}\frac{e^{-c|\lambda|^{\frac 12 z_d}}}{\big (1+|\lambda|^{\frac 12}(y_d+z_d)\big )},
\end{split}
\end{align}
and 
\begin{align}\label{r.bound.6'}
\begin{split}
& \left|\partial_{y_d }\partial_{z_d}  r'_\lambda(y',y_d,z_d)\right|+\left|\partial_{y_d} \partial_{z_d} r_{d,\lambda}(y',y_d,z_d)\right|   \leq \frac{C e^{-c |\lambda|^\frac12 z_d}}{(y_d + z_d +|y'|)^{d} \big ( 1+|\lambda|^\frac12 (y_d +z_d)\big )}.
\end{split}
\end{align}
\end{proposition}

\begin{remark}\label{rem.prop.estkernel} Related to \eqref{r.bound.2'}, as in case of the Dirichlet-Laplace kernel, we also have the estimate for the fractional derivative in the tangential variables.
Let $m_\alpha (D')$ be any Fourier multiplier of homogeneous order $\alpha>0$. Then we have 
\begin{align}\label{r.bound.2''}
\begin{split}
& \left| m_\alpha (D')   r'_\lambda(y',y_d,z_d)\right|+\left|  m_\alpha (D')  r_{d,\lambda}(y',y_d,z_d)\right| \\
&\qquad  \leq \frac{Cy_d}{(y_d + z_d +|y'|)^{d-1+\alpha}}\frac{e^{-c|\lambda|^{\frac 12}z_d}}{\big (1+|\lambda|^{\frac 12}(y_d+z_d + |y'|)\big ) \big ( 1+ |\lambda|^\frac12 (y_d +z_d)\big )},\\
& \mbox{}\\
& \left| m_\alpha (D')  \nabla  r'_\lambda(y',y_d,z_d)\right|+\left|  m_\alpha (D') \nabla r_{d,\lambda}(y',y_d,z_d)\right| \\
&\qquad  \leq \frac{C}{(y_d + z_d +|y'|)^{d-1+\alpha}}\frac{e^{-c|\lambda|^{\frac 12}z_d}}{\big (1+|\lambda|^{\frac 12}(y_d+z_d + |y'|)\big )}.
\end{split}
\end{align}
Estimate \eqref{r.bound.2''} is proved similarly as in \eqref{r.bound.2'}, and thus the proof of \eqref{r.bound.2''} is omitted in this paper.
\end{remark}

\begin{proof}[Proof of Proposition \ref{prop.estkernel}]
Using the scaling \eqref{eq.scaling}, we assume $|\lambda|=1$ for the remainder of this section. We give the proof only for $r'_\lambda$. Indeed from the representations \eqref{e.defgreenkernels} and \eqref{e.defgreenkernels'} it is clear that the estimate of $r_{d,\lambda}$ is obtained in the similar manner. By using the identity 
\begin{align*}
\frac{1}{\omega_\lambda (\xi) - |\xi|}  = \frac{\omega_\lambda (\xi) + |\xi|}{\lambda},
\end{align*}
we rewrite $r'_\lambda$ as 
\begin{align}\label{est.r'_lambda.y}
\begin{split}
r'_\lambda (y',y_d,z_d) & = \frac{1}{\lambda} \int_{\R^{d-1}} e^{iy'\cdot \xi} 
 \big ( e^{-|\xi| y_d} -  e^{-\omega_\lambda(\xi)y_d} \big )
\,e^{-\omega_\lambda(\xi)z_d} \frac{\xi\otimes\xi}{|\xi|}d\xi \\
& \quad + \frac{1}{\lambda} \int_{\R^{d-1}} e^{iy'\cdot \xi} 
 \big ( e^{-|\xi| y_d} -  e^{-\omega_\lambda(\xi)y_d} \big )
\,e^{-\omega_\lambda(\xi)z_d} \frac{\xi\otimes\xi}{\omega_\lambda (\xi)}d\xi\\
& = r'_{\lambda,1} (y',y_d,z_d)  + r'_{\lambda,2} (y',y_d,z_d).
\end{split}
\end{align}
Since $\lambda\in S_{\pi-\ep}$ and $|\lambda|=1$ the factor $\frac{1}{\omega_\lambda(\xi)}$ is more regular than $\frac{1}{|\xi|}$. Therefore we focus on the pointwise estimate of $r'_{\lambda,1}$, that is automatically satisfied by $r'_{\lambda,2}$ as well.
Again from $|\lambda|=1$ it suffices to consider the estimate of 
\begin{align}\label{est.s_lambda.y.0}
s_\lambda (y',y_d,z_d) = \int_{\R^{d-1}} e^{iy'\cdot \xi} 
 \big ( e^{-|\xi| y_d} -  e^{-\omega_\lambda(\xi)y_d} \big )
\,e^{-\omega_\lambda(\xi)z_d} \frac{\xi\otimes\xi}{|\xi|}d\xi .
\end{align}

\noindent \emph{\underline{Step 1.}} Case $y_d\geq 1$. In this case by virtue of the factor $e^{-|\xi|y_d}$ and $e^{-\omega_\lambda (\xi)y_d}$ the kernel $s_\lambda$ becomes smooth. Moreover, the factor $e^{-\omega_\lambda (\xi)z_d}$ gives the exponential decay like $e^{-c z_d}$ since $\lambda\in S_{\pi-\ep}$ and $|\lambda|=1$. Thus the main issue is the decay in $y'$ and $y_d$. By the change of the variables $\eta = \xi y_d$ we see
\begin{align*}
s_\lambda (y',y_d,z_d) & = y_d^{-d} \int_{\R^{d-1}} e^{i\tilde y'\cdot \eta} 
 \big ( e^{-|\eta|} -  e^{-\sqrt{\lambda y_d^2 + |\eta|^2}} \big )
\,e^{-\omega_\lambda (\frac{\eta}{y_d}) z_d} \frac{\eta\otimes\eta}{|\eta|}d\eta\\
& =: y_d^{-d} \, \tilde s_\lambda (\tilde y',y_d,z_d), 
\end{align*}
where $\tilde y'  = \frac{y'}{y_d}$.
We will show that 
\begin{align}\label{est.tilde_s.y}
|\tilde s_\lambda (\tilde y',y_d,z_d)| \leq \frac{Ce^{-c z_d}}{(1+|\tilde y'|)^d},
\end{align}
from which we can derive the desired bound of $s_\lambda$ for $y_d\geq 1$, since 
\begin{align}\label{est.s_lambda.y.1}
|s_\lambda(y',y_d,z_d)| & \leq \frac{C e^{-c z_d}}{(y_d + |y'|)^d} \nonumber \\
& \leq  \frac{C y_d \, e^{-c z_d}}{(1+y_d + z_d + |y'|)^d (1+y_d+z_d)}
\end{align}
by changing the constants $C$ and $c$ suitably.
To show \eqref{est.tilde_s.y} we first observe that 
\begin{align*}
|\tilde s_\lambda (\tilde y,y_d,z_d) | \leq  \int_{\R^{d-1}} \big ( e^{-|\eta|} + e^{-c|\eta|} \big ) e^{-c z_d} |\eta| d \eta \leq C e^{-c z_d},
\end{align*}
which gives the estimate \eqref{est.tilde_s.y} for the case $|\tilde y'|\leq 1$.
Next we consider the case $|\tilde y'|\geq 1$. In this case, we notice that for $y_d\geq 1$ and $\alpha\in\N$, $0\leq \alpha\leq d+1$, for all $\eta\in\R^{d-1}\setminus\{0\}$,
\begin{equation*}
\left|\nabla_\eta^{\alpha} \bigg \{ \left ( e^{-|\eta|} -  e^{-\sqrt{\lambda y_d^2 + |\eta|^2}} \right )
\,e^{-\omega_\lambda (\frac{\eta}{y_d}) z_d} \frac{\eta\otimes\eta}{|\eta|} \bigg \} \right|\leq C e^{-c z_d}e^{-c_0|\eta|} |\eta|^{-\alpha+1}.
\end{equation*}
Therefore, we apply Lemma \ref{lem.optdecay} with 
\begin{equation*}
m(\eta):=\left( e^{-|\eta|} -  e^{-\sqrt{\lambda y_d^2 + |\eta|^2}} \right )
\,e^{-\omega_\lambda (\frac{\eta}{y_d}) z_d} \frac{\eta\otimes\eta}{|\eta|}
\end{equation*}
for all $\eta\in\R^{d-1}\setminus\{0\}$ and $K(\tilde y'):=\tilde s_\lambda (\tilde y',y_d,z_d)$, where $\lambda,\ y_d$ and $z_d$ are parameters, $n=1$ and $C_0:=C e^{-c z_d}$. This gives the bound
\begin{align*}
|\tilde s_\lambda(\tilde y',y_d,z_d)| \leq C |\tilde y'|^{-d} e^{-c z_d}.
\end{align*}
Hence, estimate \eqref{est.tilde_s.y} holds also for $|\tilde y'|\geq 1$.

\noindent \emph{\underline{Step 2.}} Case $0<y_d\leq 1$. In this case we have to be careful about both the decay in $y'$ and the singularity in $y'$ near $y'=0$. Set $R_0=2$ and we decompose $s_\lambda$ by using the cut-off $\chi_{R_0}$ as
\begin{align*}
s_\lambda & = \int_{\R^{d-1}} \chi_{R_0}(\xi) \cdots d\xi + \int_{\R^{d-1}} (1-\chi_{R_0}(\xi)) \cdots d\xi\\
& = : s_{\lambda,low} + s_{\lambda,high}.
\end{align*}
As for the term $s_{\lambda,low}$, we have from $|e^{-|\xi|y_d} - e^{-\omega_\lambda(\xi)y_d}|\leq C y_d$ for $|\xi|\leq 3$,
\begin{equation*}
|s_{\lambda,low}(y',y_d,z_d)|\leq C \int_{|\xi|\leq 3} y_d e^{-c z_d} |\xi| d \xi \leq C y_d e^{-c z_d} \leq \frac{C y_d\, e^{-c z_d}}{(1+y_d + z_d)^{d+1}}.  
\end{equation*}
Here the condition $0<y_d\leq 1$ is used. This estimate gives the desired bound of $s_{\lambda,low}$ for the case $|y'|\leq 1$. Next we consider the case $|y'|\geq 1$. A direct computation implies that for $0<y_d\leq 1$, $\alpha\in N$, $0\leq \alpha\leq d+1$, for all $\xi\in\R^{d-1}\setminus\{0\}$,
\begin{multline*}
\left|\nabla_\xi^{\alpha} \bigg \{ \chi_{R_0}(\xi) \left ( e^{-|\xi| y_d} -  e^{-\omega_\lambda(\xi)y_d} \right )
\,e^{-\omega_\lambda(\xi)z_d} \frac{\xi\otimes\xi}{|\xi|}\bigg \}\right|\\
\leq Cy_d e^{-c_0 z_d} |\xi|^{-\alpha+1}\chi_{R_0}(\xi)\leq Cy_d e^{-c z_d}e^{-c_0|\xi|}|\xi|^{-\alpha+1}.
\end{multline*}
Hence, we can apply Lemma \ref{lem.optdecay} with 
\begin{equation*}
m(\xi):=\chi_{R_0}(\xi) \left ( e^{-|\xi| y_d} -  e^{-\omega_\lambda(\xi)y_d} \right )
\,e^{-\omega_\lambda(\xi)z_d} \frac{\xi\otimes\xi}{|\xi|}
\end{equation*}
for all $\eta\in\R^{d-1}\setminus\{0\}$ and $K(y'):=s_{\lambda,low} (y',y_d,z_d)$, where $\lambda,\ y_d$ and $z_d$ are parameters, and $n=1$. This yields
\begin{equation*}
|s_{\lambda,low}|\leq C y_d |y'|^{-d} e^{-c z_d},
\end{equation*}
for $|y'|\geq 1$. Combining this with the estimate in the case $|y'|\leq 1$, we have 
\begin{align}
|s_{\lambda,low}(y',y_d,z_d)| \leq \frac{C y_d \, e^{-c z_d}}{(1+y_d +z_d + |y'|)^d (1+ y_d + z_d)},
\end{align}
for $0<y_d\leq 1$, $z_d\geq 0$, and $y'\in \R^{d-1}$.

Finally, let us estimate $s_{\lambda,high}$. Since the associated symbol is smooth, the singularity around $y'=0$ is the main issue.
The key point is to use the smoothing effect from the symbol 
\begin{equation*}
e^{-|\xi| y_d} - e^{-\omega_\lambda (\xi) y_d}= \big (1- e^{(|\xi|-\omega_\lambda (\xi)) y_d} \big ) e^{-|\xi| y_d}.
\end{equation*}
Indeed,  
\begin{equation*}
|\xi| - \omega_\lambda (\xi) = | \xi | \left( 1-\sqrt{1+\frac{\lambda}{|\xi|^2}} \right) = - \frac{\lambda}{2|\xi|}\int_0^1\frac{1}{\sqrt{1+\frac{\lambda}{|\xi|^2}t}}dt=-\frac{\lambda}{2|\xi|}+\mathcal{O}\left(\frac{\lambda^2}{|\xi|^3}\right)
\end{equation*}
for $|\xi|\geq 2$ and $|\lambda|=1$. Hence, we have for $|\xi|\geq 2$, $0<y_d\leq 1$, and for all $\alpha\in\N$, $0\leq \alpha\leq d+1$,
\begin{equation}\label{est.s_lambda.y.2}
 \left| \nabla_\xi^\alpha \bigg ( \big ( e^{-|\xi| y_d} -  e^{-\omega_\lambda(\xi)y_d} \big ) 
\,e^{-\omega_\lambda(\xi)z_d} \frac{\xi\otimes\xi}{|\xi|} \bigg ) \right | \leq C y_d e^{-c z_d} e^{-c |\xi| (y_d + z_d)} |\xi|^{-\alpha}.
\end{equation}
If $|y'|\geq \frac14$ then \eqref{est.s_lambda.y.2} implies for $j=1,\ldots d-1$,
\begin{align*}
&| (-i y_j)^{d+1} s_{\lambda,high}(y',y_d,z_d)|\\
&=\left|\int_{\R^{d-1}} e^{iy'\cdot \xi} \partial_{\xi_j}^{d+1} \bigg ( (1-\chi_{R_0} ) \big ( e^{-|\xi| y_d} -  e^{-\omega_\lambda(\xi)y_d} \big ) \,e^{-\omega_\lambda(\xi)z_d} \frac{\xi\otimes\xi}{|\xi|} \bigg )  d\xi \right|\\
& \leq C y_d e^{-c z_d} \int_{|\xi|\geq R_0} |\xi |^{-d-1} d\xi \leq C y_d e^{-c z_d}, 
\end{align*}
which gives 
\begin{align*}
|s_{\lambda,high}(y',y_d,z_d)| \leq C y_d e^{-c z_d} |y'|^{-d-1} \leq \frac{C y_d \, e^{-c z_d}}{(1+y_d + z_d + |y'|)^{d+1}},
\end{align*}
since $0<y_d\leq 1$ and $|y'|\geq \frac14$.
It remains to consider the case $|y'|\leq \frac14$.
If $|y'|\leq y_d+z_d$ and $|y'|\leq \frac14$, then estimate \eqref{est.s_lambda.y.2} with $\alpha=0$ yields 
\begin{align*}
|s_{\lambda,high}(y',y_d,z_d)| &\leq C \int_{|\xi|\geq R_0} y_d e^{-c z_d} e^{-c |\xi| (y_d + z_d)} d \xi  \\
& \leq C y_d\, e^{-c z_d}\, e^{-c' (y_d+z_d)} \\
&\leq \frac{C y_d \, e^{-c z_d}}{(y_d + z_d + |y'|)^{d-1} (1+y_d + z_d +|y'|)^2}.
\end{align*} 
On the other hand, if $0<y_d + z_d \leq |y'|\leq \frac14$ then we take $R\geq 4$ and the cut-off $\chi_R$,
and decompose $s_{\lambda,high}$ into 
\begin{align*}
s_{\lambda,high} = \int_{\R^{d-1}} \chi_R(\xi) (1-\chi_{R_0}(\xi)) \cdots d\xi  + \int_{\R^{d-1}} (1-\chi_R(\xi))  \cdots d\xi =: I_R + II_R.
\end{align*}
The term $I_R$, on the one hand, is estimated from \eqref{est.s_lambda.y.2} with $\alpha=0$ as 
\begin{align*}
|I_R | \leq C \int_{R_0\leq |\xi|\leq 2 R} y_d e^{-c z_d} d \xi \leq C y_d e^{-c z_d} R^{d-1},
\end{align*}
and the term $II_R$, on the other hand, is estimated by integration by parts,
\begin{align*}
|(-i y_j)^d II_R| & = \left|\int_{\R^{d-1}} e^{iy'\cdot \xi} \partial_{\xi_j}^d \bigg ( (1-\chi_R) \cdots \bigg ) d \xi \right|\\
& \leq C \int_{|\xi|\geq R} y_d e^{-c z_d} |\xi|^{-d} d\xi  \leq C y_d e^{-c z_d} R^{-1}
\end{align*}
for each $j=1,\cdots\, d-1$. Therefore, by taking $R=|y'|^{-1}$ we have 
\begin{align*}
|s_{\lambda,high}(y',y_d,z_d)|\leq \frac{C y_d \, e^{-c z_d}}{|y'|^{d-1}} \leq \frac{C y_d \, e^{-c z_d}}{(y_d + z_d + |y'|)^{d-1} (1+ y_d + z_d+|y'|)^2}
\end{align*}
for $0<y_d + z_d \leq |y'|\leq \frac14$. Thus, we have arrived at the following estimate for $s_\lambda$ when $0<y_d\leq 1$:
\begin{equation}
\begin{aligned}\label{est.s_lambda.y.3}
|s_\lambda(y',y_d,z_d)|&\leq |s_{\lambda,low}(y',y_d,z_d)| + |s_{\lambda,high}(y',y_d,z_d)| \\
&\leq \frac{C y_d \, e^{-c z_d}}{(y_d + z_d +|y'|)^{d-1} (1+y_d + z_d+|y'|) (1+y_d+ z_d)}.
\end{aligned}
\end{equation}
From \eqref{est.s_lambda.y.1} for $y_d\geq 1$ and \eqref{est.s_lambda.y.3} for $0<y_d\leq 1$ we conclude that \eqref{est.s_lambda.y.3} holds for all $y_d>0$. The same bound is also valid for $r'_\lambda$ by the identity \eqref{est.r'_lambda.y} and $|\lambda|=1$. By scaling back to general $\lambda$, we complete the proof of \eqref{r.bound.1'}.

\noindent \emph{\underline{Step 3.}}  Next we consider the estimates for derivatives of the kernel. Again we assume that $\lambda\in S_{\pi-\ep}$ and $|\lambda|=1$, and it suffices to focus on the estimate of $s_\lambda$ in view of \eqref{est.r'_lambda.y} and \eqref{est.s_lambda.y.0}. The estimate for the derivative in $y'$ is obtained from the same argument as above for $s_\lambda$ itself, for the symbol of $\partial_{y'}^\alpha s_\lambda$ is just the multiplication by $(i\xi)^\alpha$ of the symbol of $s_\lambda$. Hence, the argument for the proof of \eqref{r.bound.1'} gives the bound 
\begin{align}\label{est.s_lambda.y.4}
|\nabla_{y'}^\alpha s_\lambda (y',y_d,z_d) | \leq \frac{C y_d \, e^{-c z_d}}{(y_d + z_d + |y'|)^{d-1+\alpha} (1+y_d + z_d + |y'|) (1+y_d + z_d )}
\end{align}
for $|\lambda|=1$ and $\alpha=1,\, 2$. Thus, estimate \eqref{r.bound.2'} holds.

As for the derivative in $y_d$, we observe the identity 
\begin{align*}
\partial_{y_d} s_\lambda (y',y_d,z_d) & = - \int_{\R^{d-1}} e^{-i y'\cdot \xi} |\xi| \big ( e^{-|\xi| y_d} - e^{-\omega_\lambda (\xi) y_d} \big ) e^{-\omega_\lambda(\xi) z_d} \frac{\xi\otimes \xi}{|\xi|} d\xi\\
& \quad +  \int_{\R^{d-1}} e^{-i y'\cdot \xi}  \big ( \omega_\lambda (\xi) - |\xi| \big ) e^{-\omega_\lambda (\xi) (y_d +z_d)}  \frac{\xi\otimes \xi}{|\xi|} d\xi.
\end{align*}
Then the first term of this right-hand side satisfies the estimate \eqref{est.s_lambda.y.4} with $\alpha=1$.
As for the second term,  we see that the symbol $(\omega_\lambda (\xi) - |\xi| ) \frac{\xi\otimes\xi}{|\xi|}$ behaves like 
\begin{align}
(\omega_\lambda (\xi) - |\xi| ) \frac{\xi\otimes\xi}{|\xi|} \sim
\left\{\begin{array}{ll}
\mathcal{O}(|\xi|),&{\rm for}~|\xi|\ll 1, \\
\mathcal{O}(1),&{\rm for}~|\xi|\gg 1.
\end{array}\right.
\end{align}
Thus, we decompose the integral into the low frequency part $|\xi|\ll 1$
and the high frequency part $|\xi|\gg 1$ using the cut-off $\chi_{R_0}$ as in the proof for $s_\lambda$. We can show that the contribution from the low frequency part is bounded by
\begin{align*} 
\frac{C e^{-c (y_d+z_d)}}{(1+y_d + z_d + |y'|)^d} \leq \frac{Ce^{-c(y_d+z_d)}}{(y_d+z_d +|y'|)^{d-1} (1+y_d+z_d+|y'|) (1+y_d+z_d)}
\end{align*}
while the contribution from the high frequency part is bounded by
\begin{align*}
&\frac{C e^{-c (y_d + z_d)}}{(y_d + z_d + |y'|)^{d-1} (1+y_d + z_d + |y'|)}\\
\leq\ & \frac{Ce^{-c(y_d+z_d)}}{(y_d+z_d +|y'|)^{d-1} (1+y_d+z_d+|y'|) (1+y_d+z_d)}.
\end{align*}
Here are have replaced the constant $c>0$ suitably.
Collecting these bounds, we conclude that 
\begin{align}
|\partial_{y_d} s_\lambda(y',y_d,z_d)| \leq \frac{C e^{-c z_d}}{(y_d+z_d+|y'|)^{d-1} (1+ y_d + z_d + |y'|) (1+y_d + z_d)}
\end{align}
for $|\lambda|=1$, which implies \eqref{r.bound.3'} with $\beta=0$.
The similar observation yields the estimate \eqref{r.bound.3'} with $\beta=1$ and also \eqref{r.bound.4'}. The details are omitted here.
Finally we consider the estimate for the derivative in $z_d$.
Again it suffices to consider the estimate of $s_\lambda$ with $|\lambda|=1$.
We observe from \eqref{est.s_lambda.y.0} that 
\begin{align*}
\partial_{z_d} s_\lambda(y',y_d,z_d) & = - \int_{\R^{d-1}} e^{iy'\cdot \xi} 
 \big ( e^{-|\xi| y_d} -  e^{-\omega_\lambda(\xi)y_d} \big )
\,e^{-\omega_\lambda(\xi)z_d} \omega_\lambda (\xi) \frac{\xi\otimes\xi}{|\xi|}d\xi \\
& = - \int_{\R^{d-1}} e^{iy'\cdot \xi} 
 \big ( e^{-|\xi| y_d} -  e^{-\omega_\lambda(\xi)y_d} \big )
\,e^{-\omega_\lambda(\xi)z_d} |\xi|  \frac{\xi\otimes\xi}{|\xi|}d\xi \\
& \quad + \int_{\R^{d-1}} e^{iy'\cdot \xi} 
 \big ( e^{-|\xi| y_d} -  e^{-\omega_\lambda(\xi)y_d} \big )
\,e^{-\omega_\lambda(\xi)z_d} \big ( |\xi|-\omega_\lambda (\xi)\big ) \frac{\xi\otimes\xi}{|\xi|}d\xi .
\end{align*}
Then the first term of the right-hand side has the similar pointwise estimate as $\nabla_{y'} s_\lambda$ which is already obtained, while the symbol of the second term has the similar behavior as the one of $s_\lambda$ for $|\xi|\ll 1$ and also decays faster for $|\xi|\gg 1$. Hence the second term satisfies at least the same estimate as $s_\lambda$. From these observations we conclude that 
\begin{align*}
|\partial_{z_d} s_\lambda(y',y_d,z_d)| \leq \frac{C y_d \, e^{-c z_d}}{(y_d + z_d + |y'|)^d (1+y_d + z_d)},\qquad |\lambda|=1.
\end{align*}
This proves \eqref{r.bound.5'} with $\beta=0$. Estimate \eqref{r.bound.5'} with $\beta=1$ and estimate \eqref{r.bound.6'} are proved in the same manner, and the details are omitted here.
The proof of estimates \eqref{r.bound.1'} to \eqref{r.bound.6'} is complete.
\end{proof}

\subsection{Kernel bounds for the pressure}

The goal of this section is to prove the following bounds on the pressure kernel $q_\lambda$ defined by \eqref{e.defqlambda}. These bounds are crucial to the estimate of the pressure in \cite[Section 2-5]{MMP17b}.

\begin{proposition}\label{prop.kernpressure}
Let $\lambda\in S_{\pi-\ep}$. There exist positive constants $c(d,\ep)$, $C(d,\ep)<\infty$  such that for all $y'\in\R^{d-1}$, $y_d,\, z_d>0$,
\begin{equation}\label{q.bound.1}
\left|q_\lambda(y',y_d,z_d)\right|\\
\leq \frac{Ce^{-c |\lambda|^{\frac 12}z_d}}{(y_d + z_d +|y'|)^{d-1}}.
\end{equation}
Moreover, for $\alpha=1,\ldots\, 3$,
\begin{align}\label{q.bound.2}
\begin{split}
\left| \nabla_{y'}^\alpha  q_\lambda(y',y_d,z_d)\right|+\left|\partial_{y_d}^\alpha q_\lambda(y',y_d,z_d)\right| \leq  \frac{Ce^{-c|\lambda|^{\frac 12}z_d}}{(y_d + z_d +|y'|)^{d-1+\alpha}},
\end{split}
\end{align}
\begin{align}\label{q.bound.3}
\begin{split}
\left|\nabla_{y'}\partial_{y_d}q_\lambda(y',y_d,z_d)\right| \leq \frac{Ce^{-c|\lambda|^{\frac 12}z_d}}{(y_d + z_d +|y'|)^{d+1}},
\end{split}
\end{align}
and for $\beta=0,1,2$, 
\begin{align}\label{q.bound.2'}
\begin{split}
\left| \nabla _y^\beta \partial_{z_d} q_\lambda(y',y_d,z_d)\right| \leq  \frac{Ce^{-c|\lambda|^{\frac12}z_d}}{(y_d+z_d+|y'|)^{d-1+\beta}} \Big ( |\lambda|^\frac12 +  \frac{1}{y_d+z_d+|y'|} \Big ).
\end{split}
\end{align}
\end{proposition}

The general scheme of the proof is the same as for the kernels corresponding to the nonlocal part (see Section \ref{sec.kernelnonlocal}). Again, using the scaling \eqref{eq.scaling}, we assume without loss of generality that $|\lambda|=1$.

\begin{proof}
\noindent \emph{\underline{Step 1.}} We assume $y_d\geq 1$. By the change of variable $\eta=\xi y_d$, we get
\begin{align*}
q_\lambda(y',y_d,z_d)=\ &\frac{i}{y_d^{d-1}}\int_{\R^{d-1}}e^{i\tilde{y}'\cdot\eta}e^{-|\eta|}e^{-\omega_{\lambda}(\frac{\eta}{y_d})z_d}\left(\frac{\eta}{|\eta|}+\frac{\eta}{\omega_{\lambda y_d^2}(\eta)}\right) d\eta\\
=\ &\frac{1}{y_d^{d-1}}\tilde{q}_\lambda(\tilde{y}',y_d,z_d),
\end{align*}
with $\tilde{y}'=\frac{y'}{y_d}$. We aim at proving 
\begin{align}\label{est.tilde_q.y}
|\tilde{q}_\lambda(\tilde{y}',y_d,z_d)| \leq \frac{Ce^{-c z_d}}{(1+|\tilde y'|)^{d-1}},
\end{align}
from which we can derive the desired bound of $q_\lambda$ for $y_d\geq 1$: 
\begin{align}\label{est.q_lambda.y.1}
|q_\lambda(y',y_d,z_d)| \leq \frac{C e^{-c z_d}}{(y_d + |y'|)^{d-1}} \leq  \frac{Ce^{-c z_d}}{(y_d + z_d + |y'|)^{d-1}}
\end{align}
by changing the constants $C$ and $c$ suitably. For $|\tilde{y}|\leq 1$, we simply bound the integrand by its modulus, and get
\begin{equation*}
\left|\tilde{q}_\lambda(\tilde{y}',y_d,z_d)\right|\leq e^{-cz_d}\int_{\R^{d-1}}e^{-|\eta|}d\eta\leq e^{-cz_d}, 
\end{equation*}
hence \eqref{est.tilde_q.y}. For $|\tilde{y}|\geq 1$, we rely on Lemma \ref{lem.optdecay}. It follows from the bound
\begin{equation*}
\left|\nabla^\alpha_\eta\left\{e^{-|\eta|}e^{-\omega_{\lambda}\left(\frac{\eta}{y_d}\right)z_d}\left(\frac{\eta}{|\eta|}+\frac{\eta}{\omega_{\lambda y_d^2}(\eta)}\right)\right\}\right|\leq Ce^{-cz_d}e^{-c_0|\eta|}|\eta|^{-\alpha}
\end{equation*}
valid for all $\eta\in\R^{d-1}\setminus\{0\}$ and the lemma, that there exists $C>0$ such that for all $\tilde y'\in\R^{d-1}$, $y_d,\ z_d>0$,
\begin{equation*}
|\tilde{q}_\lambda(\tilde{y}',y_d,z_d)|\leq \frac{Ce^{-c z_d}}{|\tilde y'|^{d-1}}.
\end{equation*}
This implies \eqref{est.tilde_q.y}.

\noindent \emph{\underline{Step 2.}} We now deal with the case $y_d\leq 1$. We split the kernel between low and high frequencies:
\begin{align*}
q_\lambda & = \int_{\R^{d-1}} \chi_{R_0}(\xi) \cdots d\xi + \int_{\R^{d-1}} (1-\chi_{R_0}(\xi)) \cdots d\xi\\
& = : q_{\lambda,low} + q_{\lambda,high}.
\end{align*}
We first deal with $q_{\lambda,low}$. Our goal is to show that
\begin{equation}\label{e.goallow}
|q_{\lambda,low}(y',y_d,z_d)|\leq \frac{Ce^{-cz_d}}{(y_d+z_d+|y'|)^{d-1}}.
\end{equation}
If $|y'|\leq 1$, we bound straightforwardly and get
\begin{equation*}
|q_{\lambda,low}(y',y_d,z_d)|\leq Ce^{-cz_d}\leq \frac{Ce^{-cz_d}}{(y_d+|y'|)^{d-1}},
\end{equation*}
from which \eqref{e.goallow} follows up to changing the constants $c$ and $C$. If $|y'|\geq 1$, we apply Lemma \ref{lem.optdecay} and get
\begin{equation*}
|q_{\lambda,low}(y',y_d,z_d)|\leq \frac{Ce^{-cz_d}}{|y'|^{d-1}}\leq \frac{Ce^{-cz_d}}{(y_d+|y'|)^{d-1}},
\end{equation*} 
from which \eqref{e.goallow} follows up to changing the constants $c$ and $C$. We now handle $q_{\lambda,high}$. We aim at proving that
\begin{equation}\label{e.goalhigh}
|q_{\lambda,high}(y',y_d,z_d)|\leq \frac{Ce^{-cz_d}}{(y_d+z_d+|y'|)^{d-1}}.
\end{equation}
If $|y'|\geq \frac{1}{4}$, we integrate by parts $d$ times and obtain for $j=1,\ldots d-1$,
\begin{align*}
&| (-i y_j)^{d} q_{\lambda,high}(y',y_d,z_d)|\\
&=\left|\int_{\R^{d-1}} e^{iy'\cdot \xi} \partial_{\xi_j}^{d} \left( ( (1-\chi_{R_0} )e^{-|\xi| y_d}e^{-\omega_\lambda(\xi)z_d}\left(\frac{\xi}{|\xi|}+\frac{\xi}{\omega(\xi)}\right)\right )  d\xi \right|\\
& \leq Ce^{-c z_d} \int_{|\xi|\geq R_0} |\xi |^{-d} d\xi \leq C e^{-c z_d}, 
\end{align*}
which gives 
\begin{align*}
|q_{\lambda,high}(y',y_d,z_d)| \leq \frac{Ce^{-c z_d}}{|y'|^{d}}\leq \frac{Ce^{-c z_d}}{|y'|^{d-1}},
\end{align*}
from which \eqref{e.goalhigh} follows. If $|y'|\leq \frac{1}{4}$, we directly bound the kernel as follows,
\begin{equation*}
|q_{\lambda,high}(y',y_d,z_d)|\leq Ce^{-c(y_d+z_d)},
\end{equation*}
which implies \eqref{e.goalhigh} in the case when $y_d+z_d\geq |y'|$. If $|y'|\leq \frac{1}{4}$ and $y_d+z_d\leq |y'|$, then we take $R\geq 4$ and the cut-off $\chi_R$,
and decompose $q_{\lambda,high}$ into 
\begin{align*}
q_{\lambda,high} = \int_{\R^{d-1}} \chi_R(\xi) (1-\chi_{R_0}(\xi)) \cdots d\xi  + \int_{\R^{d-1}} (1-\chi_R(\xi))  \cdots d\xi =: I_R + II_R.
\end{align*}
The term $I_R$, on the one hand, is estimated directly
\begin{align*}
|I_R | \leq C \int_{R_0\leq |\xi|\leq 2 R}e^{-c z_d} d \xi \leq Ce^{-c z_d} R^{d-1},
\end{align*}
and the term $II_R$, on the other hand, is estimated by integration by parts,
\begin{equation*}
|(-i y_j)^{d} II_R| = \leq Ce^{-c z_d}\int_{|\xi|\geq R}|\xi|^{-d} d\xi  \leq Ce^{-c z_d} R^{-1}
\end{equation*}
for each $j=1,\cdots\, d-1$. Therefore, by taking $R=|y'|^{-1}$ we have 
\begin{align*}
|q_{\lambda,high}(y',y_d,z_d)|\leq \frac{Ce^{-c z_d}}{|y'|^{d-1}}
\end{align*}
for $0<y_d + z_d \leq |y'|\leq \frac14$, which yields \eqref{e.goalhigh}. Consequently, we have proved \eqref{q.bound.1}.

The bounds for the derivatives \eqref{q.bound.2}-\eqref{q.bound.2'} are obtained in a rigorously similar way. Therefore, we do not repeat the argument.
\end{proof}

\section{Resolvent estimates}
\label{sec.resolest}

This section is devoted to the proof of Theorem \ref{prop.fresolvent}. In particular, the resolvent estimates 
\eqref{est.prop.fresolvent.1}-\eqref{est.prop.fresolvent.4}
for the Dirichlet-Lapalace part and the nonlocal part are shown in
 subsections \ref{sec.estDL} and \ref{sec.estnonlocal}, respectively. 
Note that since we work on the space including the nondecaying functions, 
 an assumption on the behavior of the pressure $p$ itself, rather than $\nabla p$ is needed
to ensure the uniqueness; see Theorem \ref{thm.unique}. Indeed, if one allows the linear growth for $p$ the uniqueness 
is proved only ``modulo shear flows'' in general. 
The proof of Theorem \ref{prop.fresolvent} including the uniqueness part  
is completed in the end of this section.
 
The general principles to estimate the integral formulas \eqref{def:vw} are: (i) to localize the integrals on small cubes and (ii) to use convolution estimates in the tangential direction. Integrals in the vertical direction on $z_d\in(0,1)$ may require relying on singular integral estimates. Further insights are given in the proceedings \cite{P18Xedp}.

\subsection{Estimates for the Dirichlet-Laplace part}\label{sec.dirichlet-laplace}
\label{sec.estDL}

In this subsection, we prove the $L^p_{uloc}$-$L^q_{uloc}$ estimate for the resolvent problem
for the Laplacian.
The following lemma plays a crucial role for our purpose. 
\begin{lemma}
\label{prop:v}
Assume that 
\begin{align}
1\leq q\leq p \leq \infty,\qquad 0\leq \frac1q-\frac1p<\frac1d.
\label{est:pq}
\end{align} 
Define the funtions $K=K_{\lambda}(y',y_d)$ and $K'=K'_{\lambda}(y',y_d)$ by 
 \begin{align} 
K_{\lambda}(y',y_d) &= 
\begin{cases}
C e^{-c|\lambda|^{\frac{1}{2}}|y_d|} \min \! \left\{ 
\displaystyle{\log \left(e+\frac{1}{|\lambda|^{\frac12}(|y'|+|y_d|)}\right), \,
 \frac{1}{|\lambda| (|y'| +|y_d|)^{2}}}  \right\},
\\
\qquad \qquad \qquad  \qquad \qquad \qquad \qquad \qquad \qquad \qquad \qquad \qquad d=2,
\\
\displaystyle{\frac{Ce^{-c|\lambda|^{\frac{1}{2}}|y_d|} }{
(|y'|+|y_d|)^{d-2}(1+|\lambda|^{\frac{1}{2}}(|y'| +|y_d|))^{2}}}, \qquad d\ge 3,
\end{cases}
\label{def:k}
\end{align} 
\begin{align} 
K'_{\lambda}(y',y_d) &= 
\frac{Ce^{-c|\lambda|^{\frac{1}{2}}|y_d|}}{
(|y'|+|y_d|)^{d-1}(1+|\lambda|^{\frac{1}{2}}(|y'| +|y_d|))^{2}}
\label{def:k'}
\end{align}
for $\lambda \in S_{\pi-\varepsilon}$.
Then there exists a constant $C=C(d,\ep,q,p)>0$ (independent of $\lambda$) such that
\begin{align}\label{est:k}
\| K_{\lambda}*_{y}f \|_{L^p_{uloc}} \leq \frac{C}{|\lambda|} \left ( 1+ |\lambda|^{\frac{d}{2}(\frac1q-\frac1p)} \right ) \| f \|_{L^q_{uloc}},
\\
\label{est:k'}
\|K'_{\lambda}*_{y}f\|_{L^p_{uloc}} \leq \frac{C}{|\lambda|^{\frac12}}
\left( 1+ |\lambda|^{\frac{d}{2}(\frac1q-\frac1p)} \right ) \| f \|_{L^q_{uloc}},
\end{align} 
where $*_y$ denotes the convolution in $\R^d$.
\end{lemma}
\begin{proof}
Since both estimates can be proved in the same way, we will only deal with \eqref{est:k}.
For $\eta=(\eta',\eta_d)\in \Z^{d-1}\times \Z_{\geq 0}$ 
 we estimate the $L^p$ norm of $K_{\lambda}*_y f$
in the cube of the form $B_\eta= B'_{\eta'} \times [\eta_d, \eta_d+1]$,
where $B'_{\eta'}=\eta' +[0,1]^{d-1}$. We first consider the case when $d \ge 3$.
Without loss of generality we may assume that $\eta=0$.
Let $\chi_{\alpha}$ be the characteristic function on the cube $B_{\alpha}$, for $\alpha\in\Z^{d}$.
Then we have
\begin{align*}
(K_{\lambda})  *_{y} f 
 =&
 \big (\sum_{\beta\in \Z^d} \chi_{\beta} K_{\lambda}  \big )*_y  \big ( \sum_{\alpha\in \Z^{d}} \chi_{\alpha} f  \big ) 
 \\
= &\sum_{\alpha,\beta\in \Z^{d},~\max |\alpha_i+\beta_i|\leq 2} 
\big (\chi_{\beta} K_{\lambda} \big ) *_y   \big ( \chi_{\alpha} f  \big ), 
\end{align*}
due to the support of $\chi_{\beta}$ and $\chi_{\alpha}$. 
Thus, the Young inequality for convolution yields for $\frac1p=\frac1s+\frac1q-1$,
\begin{align*}
\| K_{\lambda}*_yf \|_{L^p (B_{0})} 
 \leq &\sum_{\alpha,\beta\in \Z^{d},~\max |\alpha_i+\beta_i| \leq 2}
\| \chi_{\beta} K_{\lambda} \|_{L^s(\R^{d})} \| \chi_{\alpha} f  \|_{L^q(\R^{d})}  
\\
 =&
 \sum_{\max |\beta_i| \le 2,~\max |\alpha_i+\beta_i| \leq 2}
\| \chi_{\beta} K_{\lambda}  \|_{L^s(\R^{d})} \| \chi_{\alpha} f  \|_{L^q(\R^{d})}  
\\
&+
\sum_{\max |\beta_i| \ge 3,~\max |\alpha_i+\beta_i| \leq 2}
\| \chi_{\beta} K_{\lambda}  \|_{L^s(\R^{d})} \| \chi_{\alpha} f  \|_{L^q(\R^{d})}  
 \\ 
 =: &I_1 + I_2.
\end{align*}
For the estimate of $I_1$ we have 
\begin{align*}
\| K_{\lambda}  \|_{L^s(\R^{d})}^s 
\le& 
C\int_{\R^d} |y|^{-(d-2)s} (1+|\lambda|^{\frac12}|y|)^{-2s} dy
\\
=&
C\int_{\R^d} |\lambda|^{\frac{(d-2)s}{2}} |z|^{-(d-2)s} (1+|z|)^{-2s} dz |\lambda|^{-\frac{d}{2}}
\\
\le &C|\lambda|^{\frac{(d-2)s}{2}-\frac{d}{2}}\left(
\int_{|z| \le 1}  |z|^{-(d-2)s} dz
+
\int_{|z| \ge 1}  |z|^{-ds} dz
\right)
\\
\le& C|\lambda|^{\frac{(d-2)s}{2}-\frac{d}{2}},
\end{align*}
where we have used the assumption \eqref{est:pq} in the last line.
Therefore
\begin{align*}
I_1
\le& 
C \sum_{\max |\beta_i| \le 2}\| \chi_{\beta} K_{\lambda}  \|_{L^s(\R^{d})} 
\| f  \|_{L^q_{uloc}(\R^{d})}
\\
\le&
C |\lambda|^{\frac{(d-2)}{2}-\frac{d}{2s}}\| f  \|_{L^q_{uloc}(\R^{d})}
\\
\le &
C |\lambda|^{-1+\frac{d}{2}(\frac{1}{q}-\frac{1}{p}) } \| f  \|_{L^q_{uloc}(\R^{d})}.
\end{align*}
In order to estimate $I_2$ 
we further decompose the sum in $\beta$ as  
\begin{align*}
I_2 \le& 
\sum_{\max |\beta'_i| \ge  3,  ~\beta_d \in \Z}
\| \chi_{\beta} K_{\lambda}  \|_{L^s(\R^{d})} \| f  \|_{L^q_{uloc}(\R^{d})} 
\\
&+  \sum_{\max |\beta'_i| \le  3, ~|\beta_d| \ge 3}
\| \chi_{\beta} K_{\lambda}  \|_{L^s(\R^{d})} \| f  \|_{L^q_{uloc}(\R^{d})} 
\end{align*}
Using \eqref{def:k}, we have 
\begin{align*}
&\sum_{\max |\beta'_i| \ge  3,  ~\beta_d \in \Z}
\| \chi_{\beta} K_{\lambda}  \|_{L^s(\R^{d})}\\
\le &
C\sum_{\beta_d \in \Z,\, \max |\beta'_i| \ge  3} 
\left(\int^{\beta_d+1}_{\beta_d}e^{-cs|\lambda|^{\frac12}|y_d|} 
\int_{B'_{\beta'}} |\lambda|^{-\frac{3s}4}(|y'|+|y_d|)^{-(d-2)s-\frac{3s}{2}}dy' dy_d \right)^{\frac{1}{s}}
\\
\le &
C|\lambda|^{-\frac34} \sum_{\beta_d \in \Z}e^{-c|\lambda|^{\frac12} |\beta_d|} 
\sum_{\max |\beta'_i| \ge  3} (|\beta'|+|\beta_d|)^{-(d-\frac12)}
\\
\le &
C|\lambda|^{-\frac34} \sum_{\beta_d \in \Z}e^{-c|\lambda|^{\frac12} |\beta_d|} 
 (3+|\beta_d|)^{-\frac12}
\\
\le & 
C|\lambda|^{-\frac34} \int_{\R} e^{-c|\lambda|^{\frac12} t}t^{-\frac12}dt
\leq C|\lambda|^{-1}.
\end{align*}
On the other hand, from \eqref{def:k} 
we also have
\begin{align*}
\sum_{|\beta_d| \ge 3, ~\max |\beta'_i| \le  3}
\| \chi_{\beta} K_{\lambda}  \|_{L^s(\R^{d})}
\le &
\sum_{|\beta_d| \ge 3} \sum_{\max |\beta'_i| \le  3}
 \left(\int^{\beta_d+1}_{\beta_d} \int_{B'_{\beta'}} |y_d|^{-ds}
|\lambda|^{-s} dy' dy_d \right)^{\frac{1}{s}}
\\
\le & 
C|\lambda|^{-1} \sum_{|\beta_d| \ge 3} \beta_d^{-d}\le C|\lambda|^{-1}.
\end{align*}
Therefore we obtain 
$$
I_2 \le C|\lambda|^{-1} \|f\|_{L^q_{uloc}}.
$$
Thus we obtain \eqref{est:k} for $d\ge 3$.

For the case when $d=2$, from \eqref{def:k} we easily see that 
$$
|K_{\lambda}(y',y_d)| \le |\lambda|^{\frac14} (|y'|+|y_d|)^{-\frac12}(1+|\lambda|^{\frac12} (|y'|+|y_d|))^{-1}.
$$
By using this bound, the same argument as for the case $d\ge 3$ applies to prove \eqref{est:k} for $d=2$. So we omit the details.
\end{proof}

\begin{proposition}\label{cor.dirichlet-laplace} Let $\lambda\in S_{\pi-\epsilon}$ and let $m_\alpha (D')$ be any Fourier multiplier (in the tangential variables) homogeneous order $\alpha>0$. Assume that $p,q\in [1,\infty]$ fulfill the condition \eqref{est:pq}.
Then for the function $v$ defined in \eqref{def:vw}, i.e.,
\begin{multline*}
v(y',y_d)=\int_{\R^{d-1}}\int^{\infty}_0 k_{1,\lambda} (y'-z', y_d-z_d)f(z',z_d)dz_ddz'\\+\int_{\R^{d-1}}\int^{\infty}_0 k_{2,\lambda} (y'-z', y_d, z_d)f(z',z_d)dz_d dz',
\end{multline*}
with the kernels $k_{1,\lambda}$ and $k_{2,\lambda}$
given in \eqref{def:k_1} and \eqref{def:k_2} respectively, 
the following estimates hold: there exist positive constants $C(d,\epsilon,q,p)<\infty$ and $C_\alpha =C(\alpha,m_\alpha, d,\epsilon,q)<\infty$ (independent of $\lambda$) such that
\begin{align}
\| v \|_{L^p_{uloc}} &\leq \frac{C}{|\lambda|} \big ( 1+ |\lambda|^{\frac{d}{2}(\frac1q-\frac1p)} \big ) \| f \|_{L^q_{uloc}},
\label{est:v1}
\\
\|\nabla v  \|_{L^p_{uloc}}
& \leq \frac{C}{|\lambda|^{\frac12}}
\big ( 1+ |\lambda|^{\frac{d}{2}(\frac1q-\frac1p)} \big ) \| f \|_{L^q_{uloc}}, 
\label{est:v2}\\
\begin{split}\label{est:v2'}
\| m_\alpha (D') v \|_{L^q_{uloc}} & \leq \frac{C_\alpha}{|\lambda|^{\frac{2-\alpha}{2}}} \| f \|_{L^q_{uloc}}, \qquad \alpha\in (0,2),\\
\| m_\alpha (D') \nabla v \|_{L^q_{uloc}} & \leq \frac{C_\alpha}{|\lambda|^{\frac{1-\alpha}{2}}} \| f \|_{L^q_{uloc}}, \qquad \alpha\in (0,1).
\end{split}
\end{align}
Moreover we have for $1<q<\infty$: 
\begin{align}
\| \nabla^2 v \|_{L^q_{uloc}} 
&\leq C\left(1+e^{-c|\lambda|^\frac12} \log |\lambda| \right)\| f \|_{L^q_{uloc}}.
\label{est:v3}
\end{align} 
\end{proposition}
\begin{proof}
We extend $f$ by zero in $\R^d$ and still denote the extension by $f$. 
By Proposition \ref{l.kernel} we have $|k_{1,\lambda}(y',y_d)| \le CK_{\lambda}(y',y_d)$ for $y'\in \R^{d-1}$ and $y_d \in R$, 
and $|k_{2,\lambda}(y',y_d, z_d)|\le CK_{\lambda}(y',y_d-z_d)$ for $y'\in \R^{d-1}$ and $y_d$, $z_d \ge 0$, 
where $K \ge 0$ is the function defined in \eqref{def:k}. 
This shows that 
\begin{align*}
\| v \|_{L^p_{uloc}(\R^d_+)} 
&=
C\|K* |f| \|_{L^p_{uloc}(\R^d_+)}
\\
&\le 
C\|K* |f| \|_{L^p_{uloc}(\R^d)}
\\
&\le
 \frac{C}{|\lambda|} \big ( 1+ |\lambda|^{\frac{d}{2}(\frac1q-\frac1p)} \big ) \|f\|_{L^p_{uloc}(\R^d)}
\\
&=\frac{C}{|\lambda|} \big ( 1+ |\lambda|^{\frac{d}{2}(\frac1q-\frac1p)} \big ) \|f \|_{L^p_{uloc}(\R^d_+)},
\end{align*}
which yields the desired estimate \eqref{est:v1}.
Since the estimates \eqref{est:v2} and \eqref{est:v2'} 
can be proved in the same way (for \eqref{est:v2'} with $\alpha\in (0,2)$ we use the pointwise bound \eqref{k.bound.frac} and then apply the calculation as in Lemma \ref{prop:v}), the details will be omitted.

In order to prove \eqref{est:v3}, 
we focus on the estimate for  $(\nabla ^2 k_1)*_y f$, 
since the term associated with the kernel $k_2$ is easier to handle.
As in the proof of Lemma \ref{prop:v}, it 
suffice to consider the $L^p$ norm in $B_0$.
We decompose
\begin{align*}
\|(\nabla ^2k_{1,\lambda})  *_{y} f \|_{L^q(B_0)}
 =&
\Big\| \big (\sum_{\beta\in \Z^d} \chi_{\beta} \nabla^2 k_{1, \lambda}  \big )*_y  \big ( \sum_{\alpha\in \Z^{d}} \chi_{\alpha} f  \big ) 
\Big\|_{L^q(\R^d)}
 \\
 \le&  \sum_{\alpha,\beta\in \Z^{d},~\max |\alpha_i+\beta_i|\leq 2} 
\| \big (\chi_{\beta} \nabla ^2k_{1,\lambda}\big ) *_y   \big ( \chi_{\alpha} f \big )
\|_{L^q(\R^d)}
\\
 =&
 \sum_{\max |\beta_i| \le 2,~\max |\alpha_i+\beta_i| \leq 2}
\|\big (\chi_{\beta} \nabla ^2k_{1,\lambda} \big ) *_y   \big ( \chi_{\alpha} f \big ) 
\|_{L^q(\R^d)}
\\
&+
\sum_{\max |\beta_i| \ge 3,~\max |\alpha_i+\beta_i| \leq 2}
\|
\big (\chi_{\beta} \nabla ^2k_{1,\lambda} \big ) *_y   \big ( \chi_{\alpha} f  \big ) 
\|_{L^q(\R^d)}
 \\ 
 =:& I_1 + I_2.
\end{align*}
By \eqref{est:dk_12}, the H\"{o}rmander-Mihlin theorem applies for  $\nabla^2 k_{1, \lambda}$ and therefore: 
\begin{align*}
I_1 \le C \sum_{\max |\alpha_i| \le 4} \| \chi_{\alpha} f\|_{L^q(\R^d)} 
\le C\|f \|_{L^q_{uloc}(\R^d_+)}.
\end{align*}
We further decompose the sum in $\beta$ as  
\begin{align*}
I_2 \le& 
\sum_{\max |\beta'_i| \ge  3,  ~\beta_d \in \Z}
\| \chi_{\beta} \nabla^2 k_{1,\lambda}  \|_{L^1(\R^{d})} 
\| f  \|_{L^q_{uloc}(\R^{d})} 
\\
&+  \sum_{\max |\beta'_i| \le  3, ~|\beta_d| \ge 3}
\| \chi_{\beta} \nabla^2 k_{1,\lambda}  \|_{L^1(\R^{d})} \| f  \|_{L^q_{uloc}(\R^{d})}
\\
=:&I_{2,1} + I_{2,2}. 
\end{align*}
Using \eqref{est:dk_12}, we have 
\begin{align*}
&\sum_{\max |\beta'_i| \ge  3,  ~\beta_d \in \Z}
\| \chi_{\beta} \nabla^2 k_{1, \lambda}  \|_{L^1(\R^{d})}
\\
\le &
C\sum_{\beta_d \in \Z} \sum_{\max |\beta'_i| \ge  3} 
\int^{\beta_d+1}_{\beta_d}e^{-c|\lambda|^{\frac12}|y_d|} \int_{B'_{\beta'}} (|y'|+y_d)^{-d}dy' dy_d
\\
\le &
C\sum_{\beta_d \in \Z}e^{-c|\lambda|^{\frac12} |\beta_d|} 
\sum_{\max |\beta'_i| \ge  3} (|\beta'|+|\beta_d|)^{-d}
\\
\le &
C\sum_{\beta_d \in \Z}e^{-c|\lambda|^{\frac12} |\beta_d|} 
 (1+|\beta_d|)^{-1}
\\
\le & 
C\int_{\R} e^{-c|\lambda|^{\frac12} t}(1+t)^{-1}dt
\\
\leq& C\left(1+e^{-c|\lambda|^\frac12} \log |\lambda| \right).
\end{align*}
On the other hand, from \eqref{def:k} 
we also have
\begin{align*}
\sum_{|\beta_d| \ge 3, ~\max |\beta'_i| \le  3}
\| \chi_{\beta} \nabla^2 k_{1, \lambda}  \|_{L^1(\R^{d})}
\le &
\sum_{|\beta_d| \ge 3} \sum_{\max |\beta'_i| \le  3}
\int^{\beta_d+1}_{\beta_d} \int_{B'_{\beta'}} y_d^{-d}
 dy' dy_d 
\\
\le & 
C \sum_{|\beta_d| \ge 3} \beta_d^{-d}
\le C.
\end{align*}
Therefore we obtain 
$$
I_2 \le C\left(1+e^{-c|\lambda|^\frac12} \log |\lambda| \right)\|f\|_{L^q_{uloc}}.
$$
This completes the proof.
\end{proof}

\subsection{Estimates for the nonlocal part}
\label{sec.estnonlocal}

In this subsection we give the $L^p_{uloc}\mathchar`- L^q_{uloc}$ estimates of 
\begin{align}\label{est.nonlocal.1'}
\begin{split}
w' (y',y_d)  & = I'[f'] (y',y_d)  = \int_{\R^{d-1}} \int_0^\infty r'_\lambda(y'-z',y_d,z_d)f'(z',z_d)dz_d dz',\\
w_d (y',y_d) & = I_d[f'] (y',y_d)  = \int_{\R^{d-1}} \int_0^\infty r_{d,\lambda} (y'-z',y_d,z_d) \cdot f'(z',z_d)dz_d dz',
\end{split}
\end{align}
where the kernels are defined by \eqref{e.defgreenkernels} and \eqref{e.defgreenkernels'}.

\begin{proposition}\label{prop.estnonlocal} Let $\lambda\in S_{\pi-\ep}$ and let $m_\alpha (D')$ be any Fourier multiplier of homogeneous order $\alpha>0$. 
Assume that 
\begin{align}\label{est.nonlocal.3'}
1< q = p \leq \infty \quad {\rm or} \quad 1\leq q<p\leq \infty \quad {\rm with}\quad  0\leq \frac1q-\frac1p<\frac1d.
\end{align} 
Then for the function $w$ defined in \eqref{est.nonlocal.1'} 
the following estimates hold: there exist positive constants $C(d,\ep,q,p)$ and $C_\alpha=C (\alpha,m_\alpha,d,\ep,q)$ (independent of $\lambda$) such that
\begin{align}
\| w \|_{L^p_{uloc}} &\leq \frac{C}{|\lambda|} \big ( 1+ |\lambda|^{\frac{d}{2}(\frac1q-\frac1p)} \big ) \| f \|_{L^q_{uloc}},
\label{est:w1}
\\
\|\nabla w  \|_{L^p_{uloc}} & \leq \frac{C}{|\lambda|^{\frac12}}
\big ( 1+ |\lambda|^{\frac{d}{2}(\frac1q-\frac1p)} \big ) \| f \|_{L^q_{uloc}}, \label{est:w2} \\
\begin{split}\label{est:w2'}
\| m_\alpha (D') w\|_{L^q_{uloc}} & \leq \frac{C_\alpha}{|\lambda|^{\frac{2-\alpha}{2}}} \| f \|_{L^q_{uloc}},\qquad \alpha\in (0,2),\\
\| m_\alpha (D') \nabla w\|_{L^q_{uloc}} & \leq \frac{C_\alpha}{|\lambda|^{\frac{1-\alpha}{2}}} \| f \|_{L^q_{uloc}},\qquad \alpha\in (0,1).
\end{split}
\end{align}
Moreover we have for $1<q<\infty$: 
\begin{align}
\| \nabla^2 w \|_{L^q_{uloc}} 
&\leq C\left(1+e^{-c|\lambda|^\frac12} \log |\lambda| \right)\| f \|_{L^q_{uloc}}.\label{est:w3}
\end{align} 
\end{proposition}

\begin{remark}\label{rem.prop.estnonlocal}{\rm Estimate \eqref{est:w2} holds even for the case $p=q=1$. Similarly, if $\alpha\in (0,2)$ then \eqref{est:w2'} holds also for the case $p=q=1$. 
It is not difficult to check these facts from the proof below, and we do not give the details here.
}
\end{remark}
\begin{proof}[Proof of Proposition \ref{prop.estnonlocal}] We focus on the estimate of $w'=I'[f']$, for the estimate of $w_d = I_d[f]$ is obtained in the same manner. 

\noindent\emph{\underline{Step 1.}} We first focus on the estimate of $I'[f']$ itself. The next steps will be devoted to derivative estimates. Our estimate is based on the pointwise estimate \eqref{r.bound.1'} of the kernel $r_\lambda$.
In particular, we often use the estimate 
\begin{align}\label{est.nonlocal.2'}
|r'_\lambda (y',y_d,z_d)| \leq \frac{Cy_d\, e^{-c |\lambda|^\frac12 z_d}}{|\lambda|^\frac12 (y_d + z_d + |y'|)^d (1+ |\lambda|^\frac12 (y_d + z_d ))}
\end{align}
which easily follows from \eqref{r.bound.1'}. Notice that the variables $y_d$ and $z_d$ are not interchangeable with each other. In particular, we do not have exponential decay in $y_d$. Hence the trick used in the previous section, which transforms the action of the kernel into a convolution, does not work here.

Let $\eta=(\eta',\eta_d)\in \Z^{d-1}\times \Z_{\geq 0}$. Let us estimate the $L^p$ norm of $I'$ in the cube of the form $B_\eta= B'_{\eta'} \times [\eta_d, \eta_d+1]$,
where $B'_{\eta'}=\eta' +[0,1]^{d-1}$. Without loss of generality we may assume that $\eta'=0$.
Let $\chi_{\eta'}$ be the characteristic function on the cube $B'_{\eta'}$.
Then we have for $y=(y',y_d)\in B'_{0'} \times [\eta_d,\eta_d+1]$,
\begin{align*}
r'_\lambda (\cdot, y_d,z_d) *_{y'} f' (\cdot, z_d) & = \left(\sum_{\alpha'\in \Z^{d-1}} \chi_{\alpha'} r'_\lambda (\cdot, y_d,z_d) \right)*_{y'}  \left ( \sum_{\beta'\in \Z^{d-1}} \chi_{\beta'} f' (\cdot, z_d) \right ) \\
& = \sum_{\alpha',\beta'\in \Z^{d-1},~\max |\alpha_i^{'}+ \beta_i^{'}|\leq 2} \big (\chi_{\alpha'} r'_\lambda (\cdot, y_d,z_d) \big ) *_{y'}   \big ( \chi_{\beta'} f' (\cdot, z_d) \big ), 
\end{align*}
due to the support of $\chi_{\alpha'}$ and $\chi_{\beta'}$. Thus, the Young inequality for convolution yields for $\frac1p=\frac1s+\frac1q-1$,
\begin{align*}
& \| I'[f'] (\cdot ,y_d)\|_{L^p (B'_{0'})} \\
& \leq \sum_{\alpha',\beta'\in \Z^{d-1},~\max |\alpha_i^{'}+ \beta_i^{'}|\leq 2}\int_0^\infty \| \chi_{\alpha'} r'_\lambda (\cdot, y_d,z_d) \|_{L^s(\R^{d-1})} \| \chi_{\beta'} f' (\cdot, z_d) \|_{L^q(\R^{d-1})} d z_d \\
& \leq \sum_{\max |\alpha'_i|\leq 2,~\max |\alpha_i^{'}+ \beta_i^{'}|\leq 2}\int_0^\infty \| \chi_{\alpha'} r'_\lambda (\cdot, y_d,z_d) \|_{L^s(\R^{d-1})} \| \chi_{\beta'} f' (\cdot, z_d) \|_{L^q(\R^{d-1})} d z_d \\
& + \sum_{\max |\alpha'_i|\geq 3,~\max |\alpha_i^{'}+ \beta_i^{'}|\leq 2}\int_0^\infty \| \chi_{\alpha'} r'_\lambda (\cdot, y_d,z_d) \|_{L^s(\R^{d-1})} \| \chi_{\beta'} f' (\cdot, z_d) \|_{L^q(\R^{d-1})} d z_d \\ 
& =: I_1 + I_2.
\end{align*}
For the term $I_1$ the data is localized and $\max |\beta'|\leq 4$ holds, and therefore,
\begin{align*}
I _1 & \leq C \sum_{n=0}^\infty \int_n^{n+1} \| r'_\lambda (\cdot, y_d, z_d) \|_{L^s(\R^{d-1})} \| f' (\cdot, z_d)\|_{L^q(\{|z'|\leq 8\})}  d z_d\\
& \leq C \int_0^{1} \| r'_\lambda (\cdot, y_d, z_d) \|_{L^s(\R^{d-1})} \| f' (\cdot, z_d)\|_{L^q(\{|z'|\leq 8\})}  d z_d\\
&  \quad + C \sum_{n=1}^\infty \left ( \int_n^{n+1} \| r'_\lambda (\cdot, y_d, z_d) \|_{L^s(\R^{d-1})}^{q'} d z_d \right )^\frac{1}{q'} \| f' \|_{L^q_{uloc}}\\
& =: I_{1,1} + I_{1,2}. 
\end{align*}
From $\frac1p = \frac1s+\frac1q-1$ the pointwise estimate \eqref{est.nonlocal.2'}  implies 
\begin{align}\label{est.nonlocal.4'}
\begin{split}
\| r'_\lambda (\cdot,y_d,z_d) \|_{L^s(\R^{d-1})}& \leq \frac{C y_d\, e^{-c |\lambda|^\frac12 z_d}}{|\lambda|^\frac12 \big (1+|\lambda|^\frac12 (y_d+z_d)\big ) (y_d+z_d)^{1+(d-1)(\frac1q-\frac1p)}}\\
& \leq  \frac{C \, e^{-c |\lambda|^\frac12 z_d}}{|\lambda|^\frac12 \big (1+|\lambda|^\frac12 (y_d+z_d)\big ) (y_d+z_d)^{(d-1)(\frac1q-\frac1p)}}.
\end{split}
\end{align} 
To estimate $I_{1,1}$ for the case $p=q$ we introduce the operator $T_{1,1}$ given by 
\begin{align*}
(T_{1,1} h )(y_d)  = \int_0^1 \frac{e^{-c |\lambda|^\frac12 z_d}}{|\lambda|^\frac12(1+|\lambda|^\frac12 (y_d+z_d))} h (z_d) d z_d.
\end{align*}
It is straightforward to see 
\begin{align*}
\|T_{1,1} h \|_{L^\infty_{y_d}} \leq \frac{C}{|\lambda|} \| h \|_{L^\infty_{z_d}(0,1)}.
\end{align*}
Moreover, we have 
\begin{align*}
\left|T_{1,1} h (y_d)\right|  \leq \frac{1}{|\lambda| y_d} \| h\|_{L^1_{z_d}(0,1)},
\end{align*}
which implies 
\begin{align*}
\| T_{1,1} h \|_{L^{1,\infty}_{y_d}}\leq \frac{C}{|\lambda|} \| h \|_{L^1_{z_d}(0,1)}.
\end{align*}
Thus, $T_{1,1}$ is bounded from $L^1(0,1)$ to $L^{1,\infty}(\R_+)$, where $L^{1,\infty}(\R_+)$ is the weak $L^1$ space on $\R_+$. By the Marcinkiewicz interpolation theorem, $T_{1,1}$ is bounded from $L^q(0,1)$ to $L^q(\R_+)$ for any $1<q<\infty$, and we have 
\begin{align}\label{est.nonlocal.5'}
\| I_{1,1} \|_{L^q(\R_+)} \leq \frac{C}{|\lambda|} \| f' \|_{L^q_{uloc}},\qquad 1<q\leq \infty.
\end{align}
Next we estimate $I_{1,1}$ for the case $q<p$. Note that \eqref{est.nonlocal.4'} implies for $y_d,z_d>0$,
\begin{align*}
\| r'_\lambda (\cdot,y_d,z_d) \|_{L^s(\R^{d-1})}
& \leq  \frac{C}{|\lambda|^\frac12 \big (1+|\lambda|^\frac12 |y_d-z_d|)\big ) |y_d-z_d|^{(d-1)(\frac1q-\frac1p)}}.
\end{align*} 
Then, recall that $0<\frac1q-\frac1p<\frac1d$, which implies $0< s(d-1)(\frac1q-\frac1p)<1$ for $\frac1p=\frac1s+ \frac1q-1$. By the Young inequality for convolution, the term $I_{1,1}$ is estimated as 
\begin{align}\label{est.nonlocal.6'}
\| I_{1,1}\|_{L^p_{y_d}} & \leq \frac{C}{|\lambda|^\frac12} \bigg ( \int_\R  \frac{1}{(1+|\lambda|^\frac12 |y_d|)^s |y_d|^{s(d-1)(\frac1q-\frac1p)}} d y_d \bigg )^\frac1s  \| f' \|_{L^q_{uloc}} \nonumber \\
& \leq \frac{C}{|\lambda|^{1-\frac{d}{2}(\frac1q-\frac1p)}}\| f' \|_{L^q_{uloc}} .
\end{align}
Here we have used the fact $s>1$ since $q<p$. It is also not difficult to see $I_{1,1}\in L^1_{loc}(\R_+; L^q_{uloc}(\R^{d-1}))$ when $f'\in L^q_{uloc}(\R^d_+)$ for $q\in [1,\infty]$ (e.g. it is shown from the expression of $T_{1,1}$), and the details are omitted here. To estimate $I_{1,2}$ we observe from \eqref{est.nonlocal.2'},
\begin{align*}
& \sum_{n=1}^\infty \left ( \int_n^{n+1} \| r'_\lambda (\cdot, y_d, z_d) \|_{L^s(\R^{d-1})}^{q'} d z_d \right )^\frac{1}{q'}\\
& \leq \frac{C}{|\lambda|^\frac12} \sum_{n=1}^\infty \frac{e^{-c |\lambda|^\frac12 n}}{\big (1+|\lambda|^\frac12 (y_d+n)\big ) (y_d+n)^{(d-1)(\frac1q-\frac1p)}}\\
& \leq \frac{C}{|\lambda|^\frac12} \int_1^\infty \frac{e^{-c |\lambda|^\frac12 z_d}}{ \big (1+|\lambda|^\frac12 (y_d+z_d)\big ) (y_d+z_d)^{(d-1)(\frac1q-\frac1p)}} d z_d\\
& \leq \frac{C}{|\lambda|},
\end{align*}
which shows $\|I_{1,2}\|_{L^\infty_{y_d}}\leq \frac{C}{|\lambda|} \| f' \|_{L^q_{uloc}}$. 
Hence we have 
\begin{align}\label{est.nonlocal.7'}
\| I_1 \|_{L^p_{y_d}} \leq \frac{C}{|\lambda|} \big ( 1 + |\lambda|^{\frac{d}{2}(\frac1q-\frac1p)} \big ) \| f'\|_{L^q_{uloc}}, \qquad p,q~{\rm satisfy}~\eqref{est.nonlocal.3'}.
\end{align}
Next we estimate $I_2$. First we observe that, when $\max |\alpha_i'|\geq 3$, 
\begin{align*}
& \| \chi_{\alpha'} r'_\lambda (\cdot, y_d, z_d) \|_{L^s(\R^{d-1})} \\
& \leq \frac{Cy_d \, e^{-c |\lambda|^\frac12 z_d} }{(1+y_d + z_d + |\alpha'|)^{d-1} (1+|\lambda|^\frac12 (1+y_d + z_d + |\alpha'|)) (1+|\lambda|^\frac12 (y_d+z_d))}\\
& \leq  \frac{Cy_d \, e^{-c |\lambda|^\frac12 z_d} }{|\lambda|^\frac12 (1+y_d + z_d + |\alpha'|)^{d} (1+|\lambda|^\frac12 (y_d + z_d))}.
\end{align*}
Thus we have
\begin{align*}
I_2&\leq \sum_{\alpha'\in\mathbb Z^{d-1}}\sum_{n=0}^\infty\left(\int_n^{n+1}\frac{Cy_d^{q'} \, e^{-c q'|\lambda|^\frac12 z_d} }{|\lambda|^\frac{q'}2 (1+y_d + z_d + |\alpha'|)^{dq'} (1+|\lambda|^\frac12 (y_d + z_d))^{q'}}dz_d\right)^{\frac{1}{q'}}\| f' \|_{L^q_{uloc}}\\
&\leq \sum_{\alpha'\in\mathbb Z^{d-1}}\sum_{n=0}^\infty\frac{Cy_d \, e^{-c |\lambda|^\frac12 n} }{|\lambda|^\frac{1}2 (1+y_d + n + |\alpha'|)^{d} (1+|\lambda|^\frac12 (y_d + n))}\| f' \|_{L^q_{uloc}}\\
&\leq \sum_{\alpha'\in\mathbb Z^{d-1}}\int_0^\infty\frac{Cy_d\, e^{-c |\lambda|^\frac12 z_d} }{|\lambda|^\frac{1}2 (1+y_d + z_d + |\alpha'|)^{d} (1+|\lambda|^\frac12 (y_d + z_d))}dz_d\| f' \|_{L^q_{uloc}}.
\end{align*}
Then
\begin{align}\label{est.nonlocal.8'}
I_2  & \leq \sum_{\alpha'\in \Z^{d-1}} \frac{Cy_d}{|\lambda|^\frac12} \int_0^\infty \frac{e^{-c |\lambda|^\frac12 z_d}}{(1+y_d + z_d + |\alpha'|)^{d} (1+|\lambda|^\frac12 (y_d + z_d))} d z_d \| f' \|_{L^q_{uloc}}\nonumber \\
& \leq \frac{C y_d}{|\lambda|^\frac12} \int_0^\infty \frac{e^{-c|\lambda|^\frac12 z_d}}{(1+y_d+z_d) (1+|\lambda|^\frac12 (y_d+z_d))} d z_d \| f' \|_{L^q_{uloc}} \nonumber \\
& \leq \frac{C}{|\lambda|^\frac12} \int_0^\infty  \frac{e^{-c|\lambda|^\frac12 z_d}}{(1+|\lambda|^\frac12 z_d)} d z_d \| f' \|_{L^q_{uloc}},
\end{align}
which implies
\begin{align}\label{est.nonlocal.9'}
\| I_2 \|_{L^\infty_{y_d}}\leq \frac{C}{|\lambda|} \| f'\|_{L^q_{uloc}}.
\end{align}
Combining \eqref{est.nonlocal.7'} with \eqref{est.nonlocal.9'}, we obtain for $p,q$ satisfying \eqref{est.nonlocal.3'},
\begin{align*}
\| I'[f'] \|_{L^p_{uloc}} \leq \frac{C}{|\lambda|} \big ( 1+ |\lambda|^{\frac{d}{2}(\frac1q-\frac1p)} \big ) \| f' \|_{L^q_{uloc}}.
\end{align*} 
Note that the above proof also shows that $I'[f']\in L^1_{loc} (\R_+; L^q_{uloc}(\R^{d-1}))$ if $f'\in L^q_{uloc}(\R^d_+)$ for some $q\in [1,\infty]$.

\noindent\emph{\underline{Step 2.}} Next we consider the estimate for the derivatives.
We will use
\begin{align}\label{est.nonlocal.11'}
| \nabla^{1+\alpha} r'_\lambda (y',y_d,z_d) | &\leq \frac{C e^{-c |\lambda|^{\frac12}z_d}}{(y_d + z_d + |y'|)^{d-1+\alpha} (1+ \delta_{0\alpha}|\lambda|^\frac12 (y_d + z_d + |y'|))},
\end{align}
for $\alpha=0,\, 1$, which follows from \eqref{r.bound.2'}, \eqref{r.bound.3'}, and \eqref{r.bound.4'}. Here $\delta_{0\alpha}$ is the Kronecker delta. From \eqref{est.nonlocal.11'} we observe that for $\delta\in (0,1)$,
\begin{align}\label{est.nonlocal.delta}
| \nabla r'_\lambda (y',y_d,z_d) | &\leq \frac{C e^{-c |\lambda|^{\frac12}z_d}}{|\lambda|^\frac{\delta}{2} (y_d + z_d + |y'|)^{d-1+\delta} (1+|\lambda|^\frac12 (y_d+z_d))^{1-\delta}}.
\end{align}
By arguing as above, we see
\begin{align*}
& \| \nabla I'[f'] (\cdot,y_d) \|_{L^p (B'_{0'})}  \nonumber \\
& \leq C \int_0^{1} \| \nabla r'_\lambda (\cdot, y_d, z_d) \|_{L^s(\R^{d-1})} \| f' (\cdot, z_d)\|_{L^q(\{|z'|\leq 8\})}  d z_d \nonumber \\
&  \quad + C \sum_{n=1}^\infty \left ( \int_n^{n+1} \| \nabla r'_\lambda (\cdot, y_d, z_d) \|_{L^s(\R^{d-1})}^{q'} d z_d \right )^\frac{1}{q'} \| f' \|_{L^q_{uloc}} \nonumber \\
& \quad + \sum_{\max |\alpha'_i|\geq 3,~\max |\alpha_i^{'}+ \beta_i^{'}|\leq 2}\int_0^\infty \| \chi_{\alpha'} \nabla r'_\lambda (\cdot, y_d,z_d) \|_{L^s(\R^{d-1})} \| \chi_{\beta'} f' (\cdot, z_d) \|_{L^q(\R^{d-1})} d z_d \nonumber \\ 
& =: II_{1,1} + II_{1,2} + II_2.
\end{align*}
The last term $II_2$ is computed as in the derivation of \eqref{est.nonlocal.8'} and \eqref{est.nonlocal.9'}, and one can show
\begin{align}\label{est.nonlocal.13'}
II_2 & \leq \sum_{\alpha'\in \Z^{d-1}} \frac{C}{|\lambda|^\frac14} \int_0^\infty \frac{e^{-c |\lambda|^\frac12 z_d}}{(1+y_d+z_d+|y'|)^{d-\frac12} (1+|\lambda|^\frac12 (y_d + z_d))^\frac12} d z_d \| f' \|_{L^q_{uloc}} \nonumber \\
& \leq \frac{C}{|\lambda|^\frac14} \int_0^\infty \frac{e^{-c |\lambda|^\frac12 z_d}}{(1+ z_d)^\frac12} d z_d \| f'\|_{L^q_{uloc}} \nonumber \\
& \leq \frac{C}{|\lambda|^\frac12} \| f' \|_{L^q_{uloc}}.
\end{align}
As for $II_{1,1}$ and $II_{1,2}$, it follows from estimate \eqref{est.nonlocal.delta} that
\begin{align}\label{est.nonlocal.14'}
\| \nabla r'_\lambda (\cdot,y_d,z_d) \|_{L^s (\R^{d-1})} \leq \frac{C e^{-c|\lambda|^\frac12 z_d}}{|\lambda|^\frac{\delta}{2} (1+|\lambda|^\frac12 (y_d+z_d))^{1-\delta} (y_d+z_d)^{\delta+(d-1)(\frac1q-\frac1p)} } .
\end{align}
Take $\delta\in (0,1)$ small so that $s\left(\delta+(d-1)(\frac1q-\frac1p)\right)<1$. 
Then, the Young inequality as in the derivation of \eqref{est.nonlocal.6'} implies 
\begin{align}\label{est.nonlocal.15'}
\| II_{1,1} \|_{L^p_{y_d}} & \leq \frac{C}{|\lambda|^\frac{\delta}{2}} \left ( \int_\R \frac{1}{(1+|\lambda|^\frac12 |y_d|)^{s(1-\delta)} |y_d|^{s\left(\delta+(d-1)(\frac1q-\frac1p)\right)}} d y_d \right)^\frac1s \| f' \|_{L^q_{uloc}} \nonumber \\
& \leq \frac{C}{|\lambda|^{\frac12-\frac{d}{2}(\frac1q-\frac1p)}} \| f' \|_{L^q_{uloc}}.
\end{align}
On the other hand, the term $II_{1,2}$ is estimated as in the proof for $I_{1,2}$ by using \eqref{est.nonlocal.14'}, and we have 
\begin{align}\label{est.nonlocal.16'}
\| II_{1,2} \|_{L^\infty_{y_d}} & \leq \sup_{y_d} \frac{C}{|\lambda|^\frac{\delta}{2}} \int_1^\infty \frac{e^{-c |\lambda|^\frac12 z_d}}{(1+|\lambda|^\frac12 (y_d + z_d))^{1-\delta}(y_d + z_d)^{\delta+ (d-1)(\frac1q-\frac1p)}} d z_d \| f' \|_{L^q_{uloc}} \nonumber \\
&\leq \sup_{y_d} \frac{C}{|\lambda|^\frac{\delta}{2}} \int_1^\infty \frac{e^{-c |\lambda|^\frac12 z_d}}{(1+|\lambda|^\frac12 (y_d + z_d))^{1-\delta}(y_d + z_d)^{\delta} }d z_d \| f' \|_{L^q_{uloc}} \nonumber \\
& \leq \frac{C}{|\lambda|^\frac12} \| f' \|_{L^q_{uloc}}.
\end{align}
Thus, we have from \eqref{est.nonlocal.13'}, \eqref{est.nonlocal.15'}, and \eqref{est.nonlocal.16'},
\begin{align}\label{est.nonlocal.17'}
\| \nabla I'[f'] \|_{L^p_{uloc}}\leq \frac{C}{|\lambda|^\frac12} \left(1 + |\lambda|^{\frac{d}{2} (\frac1q-\frac1p)} \right) \| f' \|_{L^q_{uloc}}, \qquad 0\leq \frac1q-\frac1p<\frac1d.
\end{align}
Note that the case $p=q=1$ is allowed in \eqref{est.nonlocal.17'}. The proof of \eqref{est:w2'} is the same as above (it suffices to use the bound \eqref{r.bound.2''}), and we omit the details.

\noindent{\emph{\underline{Step 3.}}} Finally we give the estimate for $\nabla^2 I'[f']$. Our aim is to show 
\begin{align}\label{est.nonlocal.18'}
\| \nabla^2 I'[f'] \|_{L^q_{uloc}}\leq C\left( 1+ e^{-c|\lambda|^\frac12} \log |\lambda|\right) \| f' \|_{L^q_{uloc}}, \qquad 1 < q<\infty.
\end{align}
The key pointwise estimate reads 
\begin{align*}
|\nabla^2 r'_\lambda (y',y_d,z_d)|\leq \frac{C e^{-c |\lambda|^\frac12 z_d}}{(y_d + z_d + |y'|)^d},
\end{align*}
which follows from \eqref{est.nonlocal.11'}. This bound implies
\begin{equation}\label{e.estnabla2r'L1}
\| \nabla^2 r'_\lambda (\cdot,y_d,z_d) \|_{L^1(\R^{d-1})}\leq \frac{C e^{-c |\lambda|^\frac12 z_d}}{y_d + z_d}.
\end{equation}
As in the proof for $I'[f']$ and $\nabla I'[f']$ above, we start from 
\begin{align*}
& \| \nabla^2 I'[f'] (\cdot,y_d) \|_{L^q (B'_{0'})}  \nonumber \\
& \leq C \int_0^{1} \| \nabla^2 r'_\lambda (\cdot, y_d, z_d) \|_{L^1(\R^{d-1})} \| f' (\cdot, z_d)\|_{L^q(\{|z'|\leq 8\})}  d z_d \nonumber \\
&  \quad + C \sum_{n=1}^\infty \left ( \int_n^{n+1} \| \nabla^2 r'_\lambda (\cdot, y_d, z_d) \|_{L^1(\R^{d-1})}^{q'} d z_d \right )^\frac{1}{q'} \| f' \|_{L^q_{uloc}} \nonumber \\
& \quad + \sum_{\max |\alpha'_i|\geq 3,~\max |\alpha_i^{'}+ \beta_i^{'}|\leq 2}\int_0^\infty \| \chi_{\alpha'} \nabla^2 r'_\lambda (\cdot, y_d,z_d) \|_{L^1(\R^{d-1})} \| \chi_{\beta'} f' (\cdot, z_d) \|_{L^q(\R^{d-1})} d z_d \nonumber \\ 
& =: III_{1,1} + III_{1,2} + III_2.
\end{align*}
To estimate $III_{1,1}$ we introduce the operator $T$ given by 
\begin{align*}
(T h )(y_d)  = \int_0^1 \frac{e^{-c |\lambda|^\frac12 z_d}}{y_d+z_d} h (z_d) d z_d.
\end{align*}
It is straightforward to see 
\begin{align*}
|(T h)(y_d) |\leq \frac{C}{y_d} \| h \|_{L^1_{z_d}},\qquad | (Th) (y_d)| \leq \frac{C}{y_d^{\frac1q}} \| h \|_{L^q_{y_d}}
\end{align*}
for any $1<q<\infty$. Thus, $T$ is bounded from $L^q(\R_+)$ to $L^{q,\infty}(\R_+)$ for any $1\leq q<\infty$, where $L^{q,\infty}(\R_+)$ is the weak $L^q$ space on $\R_+$.
By the Marcinkiewicz interpolation theorem, $T$ is bounded from $L^q(\R_+)$ to $L^q(\R_+)$ for any $1<q<\infty$.
This implies 
\begin{align*}
\| III_{1,1} \|_{L^q(\R_+)} \leq C \| f' \|_{L^q_{uloc}},\qquad 1<q<\infty.
\end{align*}
The terms $III_{1,2}$ and $III_2$ are estimated in the similar manner as for $I_{1,2}$ and $I_2$ above and we see 
\begin{align*}
\|III_{1,2}\|_{L^\infty_{y_d}} + \|III_2\|_{L^\infty_{y_d}}  & \leq C \int_1^\infty \frac{e^{-c |\lambda|^\frac12 z_d}}{ z_d} d z_d \| f' \|_{L^q_{uloc}}\\
& \quad + C \int_0^\infty \frac{e^{-c |\lambda|^\frac12 z_d}}{1+z_d} d z_d \| f' \|_{L^q_{uloc}}\\
& \leq C\left(1+ e^{-c|\lambda|^\frac12} \log |\lambda| \right) \| f' \|_{L^q_{uloc}}.
\end{align*}
The proof is complete.
\end{proof}
To conclude, let us notice that it is easy to get the uniform estimate
\begin{equation}\label{e.unif2y'}
\| \nabla^2_{y'} I'[f'] \|_{L^q_{uloc}}\leq C\| f' \|_{L^q_{uloc}}, \ {\rm and ~thus} \  \| \nabla^2_{y'} w \|_{L^q_{uloc}} \leq C \| f' \|_{L^q_{uloc}}, \ 1 < q<\infty.
\end{equation}
Indeed, for these tangential derivatives, we can rely on the kernel bound \eqref{r.bound.2'}, which yields
\begin{equation*}
\|\nabla_{y'}^2r'_\lambda(\cdot,y_d,z_d)\|_{L^1(\R^{d-1})}\leq \frac{y_d}{(y_d+z_d)^2},
\end{equation*}
instead of \eqref{e.estnabla2r'L1}. This enables to get estimates uniform in $\lambda$ following the same strategy as above. 
We do not know whether the difficulty we encounter here to show a similar uniform estimate on $\partial_{y_d}^2 I'[f']$ or $\partial_{z_d}^2I'[f']$ is a technical one or reveals an essential obstruction.

\begin{remark}[Estimates for the pressure]\label{rem.pressure.estimate}
{\rm
We are only concerned with gradient estimates on the pressure. From the pointwise estimate \eqref{q.bound.2}, it is clear that $\nabla p$ is estimated in $L^q_{uloc}(\R^d_+)$ in the exact same way as we estimated $\nabla^2 I'[f']$ stated in \eqref{est.nonlocal.18'}, i.e., 
\begin{align}\label{est.rem.pressure.estimate.1}
\| \nabla p\|_{L^q_{uloc}} \leq C \bigg ( 1+ e^{-c|\lambda|^\frac12} \log |\lambda| \bigg ) \| f' \|_{L^q_{uloc}}\,, \qquad 1<q<\infty.
\end{align}
On the other hand, recalling the identity $\omega_\lambda (\xi)^2 = \lambda+ |\xi|^2$, we also have from the formula \eqref{pressure.velocity.4'},
\begin{align}\label{est.rem.pressure.estimate.2}
\| \nabla p \|_{L^q_{uloc}} \leq C\bigg ( 1+ e^{-c|\lambda|^\frac12} \log |\lambda| \bigg ) \| (\lambda - \Delta ) u' \|_{L^q_{uloc}},\qquad 1<q<\infty.
\end{align}
The estimate \eqref{est.rem.pressure.estimate.2} is crucial in obtaining the characterization of the domain of the Stokes operator in the $L^q_{uloc}$ spaces with $1<q<\infty$.}
\end{remark}

\noindent
\begin{proof}[Proof of Theorem \ref{prop.fresolvent}.]\ The estimates 
\eqref{est.prop.fresolvent.1}, \eqref{est.prop.fresolvent.2},  \eqref{est.prop.fresolvent.3} and \eqref{est.prop.fresolvent.4} 
are proved in Proposition \ref{cor.dirichlet-laplace}, Proposition 
\ref{prop.estnonlocal}
and \eqref{est.rem.pressure.estimate.1}.
Hence taking account of Theorem \ref{thm.unique} in the Appendix, 
 it suffices to show the pressure gradient given by the formula \eqref{pressure.velocity.2} satisfies \eqref{e.condpressure}. 
Consider only the tangential gradient $\nabla'p$, for the normal gradient
can be estimated in the same manner.
By the formula \eqref{pressure.velocity.3'}, the tangential gradient is
written as follows:
$$
\nabla' p(y_d)
= 
R' \nabla' P(y_d) \cdot \partial_{y_d} u'(0),
$$
where $R'=(R_1,R_2,\cdots, R_{d-1})$ is 
the vector valued Riesz transform in $\R^{d-1}$ and $\partial_{y_d} u' (0)$ makes sense in $L^q_{uloc} (\R^{d-1})$ by the trace theorem and the regularity $\nabla^\alpha  u'\in L^q_{uloc}(\R^d_+)$ for $\alpha=0,1,2$ and $1<q<\infty$.
By the property of the Poisson kernel $P_{y_d}(y')$, 
it is easy to see that $\nabla' P_{y_d}$ belongs to the Hardy space 
$\mathcal{H}^1(\R^{d-1})$ for $y_d>0$. 
Therefore from the boundedness of the Riesz transform from 
$\mathcal{H}^1(\R^{d-1})$ to $L^1(\R^{d-1})$   
we have 
\begin{align*} 
\|\nabla' p(y_d)\|_{L^1_{uloc}(\R^{d-1})} 
& = 
\|R' \nabla' P(y_d)*\partial_{y_d} u'(0)\|_{L^1_{uloc}(\R^{d-1})}
\\
&\le 
\|R' \nabla' P_{y_d}\|_{L^1(\R^{d-1})}
 \|\partial_{y_d} u'(0)\|_{L^1_{uloc}(\R^{d-1})}
\\
&\le
\|\nabla' P_{y_d}\|_{\mathcal{H}^1(\R^{d-1})}
\|\partial_{y_d} u'(0)\|_{L^1_{uloc}(\R^{d-1})}
\\
&\le Cy_d^{-1} \|\partial_{y_d} u'(0)\|_{L^1_{uloc}(\R^{d-1})},
\end{align*}
which proves the desired bound \eqref{e.condpressure}.
\end{proof}

\section{The Stokes semigroup in $L^p_{uloc}$ spaces} \label{sec.semigroup}

In this section we construct the Stokes operator in the $L^q_{uloc}$ spaces and the associated semigroup. Usually the Stokes operator ${\bf A}$ is written as ${\bf A} = - \mathbb{P} \Delta_D$, where $\mathbb{P}$ is the Helmholtz-Leray projection and $\Delta_D$ is the realization of the Laplace operator under the Dirichlet boundary condition. However, the action of $\mathbb{P}$ does not make sense in general for nondecaying data, so we need to define the Stokes operator in a different way. In principle, we follow the argument of \cite{DHP01} to define the Stokes operator but with a slight change of some technical details. 

Notice that several works \cite{U87,CPS00,DZ14} provide representations formulas for the solution of the unsteady Stokes problem. However, these formulas involve singular integral operators, which are unbounded on spaces of nonintegrable functions. Our approach which relies on the Dunford formula and the Stokes resolvent problem takes advantage of the fundamental insight of Desch, Hieber and Pr\"uss, which circumvents the unboundedness of the Helmholtz-Leray transform.

Let $\lambda\in S_{\pi-\ep}$ with $\ep \in (0,\pi)$. 
Let $1<q\leq \infty$ and $f\in L^q_{uloc,\sigma}(\R^d_+)$. Then there exists a unique solution $(u,\nabla p)$ to \eqref{e.resol} in the class stated in Theorem \ref{prop.fresolvent}.  
We denote this linear map from $L^q_{uloc,\sigma}(\R^d_+)$ to $L^q_{uloc,\sigma}(\R^d_+)$ as $R(\lambda)$.
For convenience we also write the associated pressure $\nabla p$ as $\nabla p_u$ to emphasize that $\nabla p_u$ is determined from $u$ by the formula \eqref{pressure.velocity.3}.
Note that $\gamma \partial_{y_d} u' $ makes sense in $L^q_{uloc} (\R^{d-1})$ 
as is stated in the proof of  Theorem \ref{prop.fresolvent}. 
What we need to show is that (i) the null space of $R(\lambda)$ is trivial, and (ii) the resolvent identity
$R(\lambda) - R(\mu) = - (\lambda-\mu) R(\lambda) R(\mu)$ holds for any $\lambda,\mu\in S_{\pi-\ep}$. Remark that (ii) implies in particular that $R(\lambda)$ commutes with $R(\mu)$.
To prove (i), we assume that $u:=R(\lambda)f=0$ for some $f\in L^q_{uloc,\sigma}(\R^d_+)$. 
Then the associated pressure $\nabla p_u$ is zero by the formula \eqref{pressure.velocity.3}.
Hence, we must have $f=0$ since $(u,\nabla p_u)$ solves \eqref{e.resol}. Thus, $R(\lambda)$ is injective.
Next we prove the resolvent identity. Fix any $f\in L^q_{uloc,\sigma}(\R^d_+)$ and set $u=R(\lambda) f$ and $v=R(\mu)f$. Then $(u-v, \nabla p_u - \nabla p_v)$ solves \eqref{e.resol} with $f=-(\lambda -\mu) v$.
By Theorem \ref{prop.fresolvent} there exists a solution $(w, \nabla p_w)$ to \eqref{e.resol}  with $f=-(\lambda-\mu) v$, which is unique in the class stated in Theorem  \ref{prop.fresolvent}, and $w=R(\lambda) (\mu-\lambda) v=-(\lambda-\mu) R(\lambda) v$ by the definition of  $R(\lambda)$. 
Since $(w, \nabla p_w)$ and $(u-v, \nabla p_u - \nabla p_v)$ belong to the same class as stated in Theorem \ref{prop.fresolvent} (in particular, both satisfy the decay condition on the derivative of the pressure as $y_d\rightarrow \infty$), by the uniqueness result of Theorem \ref{prop.fresolvent}, we have $w=u-v$.
This implies $R(\lambda) f - R(\mu) f = - (\lambda-\mu) R(\lambda) R(\mu) f$ for any $f\in L^q_{uloc,\sigma}(\R^d_+)$, and hence, the resolvent identity is proved.

From (i) and (ii) we conclude that there exists a closed linear operator ${\bf A}: D({\bf A})\subset L^q_{uloc,\sigma}(\R^d_+)\rightarrow L^q_{uloc,\sigma} (\R^d_+)$, such that the domain $D({\bf A})$ of ${\bf A}$ is the range of $R(\lambda)$ which is independent of $\lambda$, and the resolvent set of $-{\bf A}$ includes $S_{\pi-\ep}$ for any $\ep\in (0,\pi)$, and $(\lambda + {\bf A})^{-1}=R(\lambda)$ for any $\lambda\in S_{\pi-\ep}$. We say that ${\bf A}$ is the Stokes operator realized in $L^q_{uloc,\sigma}(\R^d_+)$.

\begin{proposition}\label{prop.domain} 
Let $1<q<\infty$ and let ${\bf A}$ be the Stokes operator realized in $L^q_{uloc,\sigma}(\R^d_+)$. Then 
\begin{align}\label{eq.prop.domain.1}
D ({\bf A}) = \{ u \in L^q_{uloc,\sigma} (\R^d_+)~|~\nabla^\alpha u \in L^q_{uloc}(\R^d_+), ~\alpha=0,1,2, ~~~u =0 ~ {\rm on} ~ \partial\R^d_+\}
\end{align}
\end{proposition}

\begin{proof} Theorem \ref{prop.fresolvent} implies that for any $f\in L^q_{uloc,\sigma} (\R^d_+)$, the function $R(\lambda) f$ belongs to $L^q_{uloc,\sigma}(\R^d_+)$, $\nabla^\alpha R(\lambda) f\in L^q_{uloc}(\R^d_+)$ for $\alpha=0,1,2$, and $R(\lambda)f =0$ on $\partial\R^d_+$. Thus, the domain $D({\bf A})$, which is the range of $R(\lambda)$, belongs to the set defined in the right-hand side of \eqref{eq.prop.domain.1}.
Conversely, let $u$ be any function belonging to the right-hand side of \eqref{eq.prop.domain.1}.
Then the couple $(u, \nabla p_u)$ with $\nabla p_u$ defined by \eqref{pressure.velocity.3} solves \eqref{e.resol} with $f=\lambda u - \Delta u + \nabla p_u$, and $f$ belongs to $L^q_{uloc,\sigma} (\R^d_+)$ by the definition of $\nabla p_u$ and the estimate \eqref{est.rem.pressure.estimate.2}. This implies that $u$ belongs to the range of $R(\lambda)$, and thus to $D({\bf A})$.
The proof is complete.
\end{proof}

Note that we do not have the characterization of the domain of ${\bf A}$ in the space $L^\infty_\sigma (\R^d_+)$. Theorem \ref{prop.fresolvent} and the definition of $R(\lambda)$ immediately yield the following
\begin{proposition}\label{prop.analyticity}
Let  $1<q \leq \infty$ and let ${\bf A}$ be the Stokes operator realized in $L^q_{uloc,\sigma}(\R^d_+)$. Then for any $\ep\in (0,\pi)$ the sector $S_{\pi-\ep}$ belongs to the resolvent of  $-{\bf A}$ and 
\begin{align*}
|\lambda| \| (\lambda+{\bf A})^{-1} f\|_{L^q_{uloc}}\leq C_\ep \| f\|_{L^q_{uloc}},\qquad \lambda\in S_{\pi-\ep},\quad f\in L^q_{uloc,\sigma}(\R^d_+).
\end{align*}
Therefore, $-{\bf A}$ generates a bounded analytic semigroup in $L^q_{uloc,\sigma}(\R^d_+)$.
\end{proposition}
 
Notice that $\mathbf A$ is not known to be strongly continuous, because $D(\mathbf A)$ is not dense in $L^q_{uloc,\sigma}$ (this is seen easily, see for instance \cite[Lemma 3.1 (d)]{MS95}). Applying Theorem \ref{prop.fresolvent}, we also have the $L^p_{uloc}-L^q_{uloc}$ estimates for $e^{-t{\bf A}}$ as follows.
\begin{proposition}\label{prop.L^p_u-L^q_u.semigroup}  
Let  $1<q \leq \infty$ and let ${\bf A}$ be the Stokes operator realized in $L^q_{uloc,\sigma}(\R^d_+)$. Then there exists a constant $C(d,q)<\infty$, for $\alpha=0,1$, 
\begin{align}\label{est.prop.L^p_u-L^q_u.semigroup.1} 
t^\frac{\alpha}{2} \| \nabla^\alpha  e^{-t {\bf A}} f \|_{L^q_{uloc}}  + t\left\| \frac{d}{dt} e^{-t {\bf A}} f \right\|_{L^q_{uloc}} \leq C\| f\|_{L^q_{uloc}}, \qquad t>0,\quad f\in L^q_{uloc,\sigma} (\R^d_+),
\end{align}
and when $1<q<\infty$,
\begin{align}\label{est.prop.L^p_u-L^q_u.semigroup.1'} 
\frac{t}{\log (e+t)} \| \nabla^2  e^{-t {\bf A}} f \|_{L^q_{uloc}} \leq C\| f\|_{L^q_{uloc}}, \qquad t>0,\quad f\in L^q_{uloc,\sigma} (\R^d_+).
\end{align}
Moreover, for $1<q \leq p \leq \infty$ or $1\leq q<p\leq\infty$, there exists a constant $C(d,p,q)<\infty$ such that
\begin{align}
\| e^{-t{\bf A}} f \|_{L^p_{uloc}} & \leq C \left( t^{-\frac{d}{2} (\frac1q-\frac1p)} + 1\right) \| f\|_{L^q_{uloc}},\qquad t>0, \quad f\in L^q_{uloc,\sigma}(\R^d_+), \label{est.prop.L^p_u-L^q_u.semigroup.2} \\
\left\| \nabla e^{-t{\bf A}} f \right\|_{L^p_{uloc}} & \leq C t^{-\frac12} \left( t^{-\frac{d}{2} (\frac1q-\frac1p)} + 1\right) \| f\|_{L^q_{uloc}},\qquad t>0, \quad f\in L^q_{uloc,\sigma}(\R^d_+).\label{est.prop.L^p_u-L^q_u.semigroup.3} 
\end{align}
\end{proposition}

\begin{remark}
{\rm In \eqref{est.prop.L^p_u-L^q_u.semigroup.2} and \eqref{est.prop.L^p_u-L^q_u.semigroup.3} the estimates are stated, in particular, for the exponents $1=q<p\leq \infty$, while the generation of the analytic semigroup in $L^1_{uloc,\sigma}(\R^d_+)$ seems to fail, by a reason similar to the case of $L^1$ observed in \cite{DHP01}. 
The estimate for the case $p=q=\infty$ is also well known. In \eqref{est.prop.L^p_u-L^q_u.semigroup.1'} the logarithmic growth factor appears due to the logarithmic factor in the resolvent estimate \eqref{est.prop.fresolvent.2}.
This additional growth does not seem to be optimal at least for the semigroup bound, and it is possible to remove it if one obtains the resolvent estimate such as 
\begin{align}\label{rem.est.prop.L^p_u-L^q_u.semigroup.1} 
\| \nabla^2 (\lambda + {\bf A})^{-1} f \|_{L^q_{uloc}} \leq C \big ( \| f \|_{L^q_{uloc}} + |\lambda|^{-\frac12} \| \nabla f \|_{L^q_{uloc}}\big ),
\end{align}
for one can then use the identity $e^{-t{\bf A}} f = \frac{1}{2\pi i} \int_\Gamma e^{t\lambda} (\lambda + {\bf A})^{-1} f d\lambda = \frac{1}{2\pi i t}   \int_\Gamma e^{t\lambda} (\lambda + {\bf A})^{-2} f d\lambda$ in estimating $\nabla^2 e^{-t{\bf A}} f$, where the integration by parts is used.
Estimate \eqref{rem.est.prop.L^p_u-L^q_u.semigroup.1} seems to be valid, though we do not give the detailed proof in this paper.
We also note that the estimates for the higher order derivatives can be shown by our method, but we do not go into the details here.
}
\end{remark}

\begin{proof} The estimate $\| e^{-t{\bf A}} f\|_{L^q_{uloc}}\leq C \| f\|_{L^q_{uloc}}$ for $t>0$ is already shown in Proposition \ref{prop.analyticity},
and we focus on the other estimates.
Let us recall the standard representation formula of $e^{-t{\bf A}}$ in terms of the Dunford integral 
\begin{align}\label{proof.prop.L^p_u-L^q_u.semigroup.1}
e^{-t{\bf A}} f = \frac{1}{2\pi i} \int_\Gamma e^{t\lambda} (\lambda + {\bf A})^{-1} f d\lambda.
\end{align}
Here $\Gamma=\Gamma_\kappa$ with $\kappa\in (0,1)$ is the curve 
$\{\lambda\in \mathbb{C}~|~|{\rm arg}\, \lambda|=\eta, ~|\lambda|\geq \kappa\} \cup \{\lambda\in \mathbb{C}~|~|{\rm arg}\, \lambda|\leq \eta, ~|\lambda|=\kappa\}$ for some $\eta \in (\frac{\pi}{2},\pi)$. Then, estimate \eqref{est.prop.fresolvent.1} with $\alpha=1$ yields
\begin{align*}
\| \nabla e^{-t{\bf A}} f \|_{L^q_{uloc}} & \leq  C\int_\Gamma e^{t\Re (\lambda)} |\lambda|^{-\frac12} |d\lambda| \| f\|_{L^q_{uloc}}.
\end{align*}
Since $\kappa\in (0,1)$ is arbitrary, we may take the limit $\kappa\rightarrow 0$ and obtain 
\begin{align*}
\| \nabla e^{-t{\bf A}} f \|_{L^q_{uloc}} & \leq C \int_0^\infty e^{-t r \cos \eta}  r^{-\frac12}  d r \| f\|_{L^q_{uloc}} \leq C t^{-\frac12} \| f\|_{L^q_{uloc}}, \qquad t>0.
\end{align*}
The estimate of $\frac{d}{dt} e^{-t{\bf A}} f$ and $\nabla^2 e^{-t{\bf A}} f$ are obtained in the same manner.
Note that, as for the estimate of $\nabla^2 e^{-t {\bf A}} f$, we have for $t>0$
\begin{align*}
\| \nabla^2 e^{-t{\bf A}} f \|_{L^q_{uloc}} & \leq C \int_0^\infty  e^{-t r \cos \eta} e^{-r^\frac12} \log r d r \| f\|_{L^q_{uloc}} \leq C \log (e + t)  \| f\|_{L^q_{uloc}}.
\end{align*}
Let $1<q<p\leq \infty$. To prove \eqref{est.prop.L^p_u-L^q_u.semigroup.2} we first observe that the following the formula holds for each $m\in \mathbb{N}$ in virtue of the integration by parts in \eqref{proof.prop.L^p_u-L^q_u.semigroup.1}.
\begin{align}\label{proof.prop.L^p_u-L^q_u.semigroup.2}
e^{-t{\bf A}} f = \frac{m!}{2\pi i t^m} \int_\Gamma e^{t\lambda} (\lambda + {\bf A})^{-m-1} f d\lambda.
\end{align}
By taking $m$ large enough, we can choose $\{q_j\}_{j=0}^m$ such that $q_0=q$, $q_j<q_{j+1}$, $q_m=p$, and $\frac{1}{q_j}-\frac{1}{q_{j+1}}< \frac1d$. Then, estimate \eqref{est.prop.fresolvent.3} is applied for each couple $(q_j, q_{j+1})$,
and we obtain 
\begin{align}\label{proof.prop.L^p_u-L^q_u.semigroup.3}
& \| e^{-t{\bf A}} f \|_{L^p_{uloc}}  \nonumber \\
& \leq Ct^{-m} \int_\Gamma e^{-t \Re (\lambda)} \| (\lambda+{\bf A})^{-m-1} f \|_{L^p_{uloc}} |d\lambda| \nonumber \\
& \leq Ct^{-m} \int_\Gamma e^{-t \Re (\lambda)} |\lambda|^{-m-1} \big ( 1 + |\lambda|^{\frac{d}{2} (\frac{1}{q_{m}}-\frac{1}{q_{m-1}})} \big ) \cdots \big ( 1 + |\lambda|^{\frac{d}{2} (\frac{1}{q_{1}}-\frac{1}{q_{0}})} \big ) \|  f \|_{L^p_{uloc}} |d\lambda| \nonumber \\
& \leq C t^{-m} \int_\Gamma  e^{-t \Re (\lambda)} |\lambda|^{-m-1} \big ( 1 + |\lambda|^{\frac{d}{2} (\frac{1}{p}-\frac{1}{q})} \big ) |d\lambda| \| f\|_{L^q_{uloc}}.
\end{align}
Thus, again by taking the limit $\kappa \rightarrow 0$, we have 
\begin{align}\label{proof.prop.L^p_u-L^q_u.semigroup.4}
\| e^{-t{\bf A}} f \|_{L^p_{uloc}}  
& \leq C t^{-m} \int_0^\infty  e^{-t r \cos \eta } r^{-m-1} \big ( 1 + r^{\frac{d}{2} (\frac{1}{p}-\frac{1}{q})} \big ) d r \| f\|_{L^q_{uloc}} \nonumber \\
& \leq C \big (1+ t^{-\frac{d}{2}(\frac1q-\frac1p)} \big ) \| f\|_{L^q_{uloc}}.
\end{align}
This proves \eqref{est.prop.L^p_u-L^q_u.semigroup.2}.
Estimate \eqref{est.prop.L^p_u-L^q_u.semigroup.3} is shown in the same manner by using the formula \eqref{proof.prop.L^p_u-L^q_u.semigroup.2} and the resolvent estimate \eqref{est.prop.fresolvent.4}, we omit the details. The proof is complete.
\end{proof}

\section{Bilinear estimates for the Navier-Stokes equations}
\label{sec.bilinse}

\subsection{The symbol of the Helmholtz-Leray projector}

The formulas derived for the resolvent problem in Section \ref{sec.intrepr} are valid for a right hand side $f$ in the class $L^p_{uloc,\sigma}$, i.e. soleno\"idal vector fields such that $f_d$ vanishes on $\partial\mathbb R^d_+$. When dealing with the Navier-Stokes system, the nonlinear term $u\cdot\nabla u=\nabla\cdot (u\otimes u)$ is such that for any $z'\in\mathbb R^{d-1}$
\begin{equation*}
(u\cdot\nabla u)_d(z',0)=u(z',0)\cdot\nabla u_d(z',0)=0
\end{equation*}
by the noslip boundary condition, but it is not divergence-free. Hence, for $f\in C^\infty(\mathbb R^d_+)$ and $f_d=0$ on $\partial\mathbb R^d_+$, we have to compute the symbol of the Helmholtz-Leray projector $\mathbb P$ on the divergence-free fields.

In order to compute the Helmholtz-Leray projection we look for a formal decomposition of $f $ into
\begin{equation*}
f=\mathbb Pf+\nabla g,
\end{equation*}
with
\begin{equation*}
\nabla\cdot\mathbb Pf=0,\qquad (\mathbb Pf)_d(z',0)=0\ \mbox{for any}\ z'\in\R^{d-1}.
\end{equation*}
For the moment $f$ is assumed to be smooth and decay fast enough at spatial infinity.
We have to solve the following problem elliptic problem for $g$ with Neumann boundary condition
\begin{subequations}
\label{e.pressureg}
\begin{equation}
\left\{ 
\begin{aligned}
& \Delta g = \nabla\cdot f  & \mbox{in} &\ \R^d_+, \\
& \nabla g\cdot e_d  = f_d  & \mbox{on} & \ \partial\R^d_+,
\end{aligned}
\right.
\end{equation}
and such that 
\begin{equation}
\nabla g(z',z_d)\longrightarrow 0\qquad\mbox{when}\ z_d\rightarrow\infty.
\end{equation}
\end{subequations}
The solution $g$ to \eqref{e.pressureg} is expressed in Fourier space by for all $\xi\in\R^{d-1}\setminus\{0\}$, for all $z_d>0$,
\begin{align*}
& \quad \widehat{g}(\xi,z_d) \\
& =-\frac{e^{-z_d|\xi|}}{|\xi|} f_d(\xi,0)  -\int_0^{\infty}\frac{i\xi\cdot\widehat{f'}(\xi,s)+\partial_d\widehat{f_d}(\xi,s)}{2|\xi|}\left[e^{-|z_d-s||\xi|}+e^{-(z_d+s)|\xi|}\right]ds\\
& = -\int_0^{\infty}\frac{i\xi\cdot\widehat{f'}(\xi,s)}{2|\xi|}\left[e^{-|z_d-s||\xi|}+e^{-(z_d+s)|\xi|}\right]ds\\
& \quad  + \frac{1}{2} \int_0^{z_d} \widehat{f_d}(\xi,s) e^{-(z_d-s)|\xi|} d \xi - \frac12\int_{z_d}^\infty \widehat{f_d}(\xi,s) e^{-(s-z_d)|\xi|} d \xi  \\
& \quad - \frac12 \int_0^\infty \widehat{f_d}(\xi,s) e^{-(z_d+s)|\xi|} d \xi. 
\end{align*}
Here we have used the integration by parts.
As a consequence, we obtain the formulas for the Helmholtz-Leray projection: for all $\xi\in\R^{d-1}\setminus\{0\}$, for all $z_d>0$,
\begin{subequations}\label{e.deflerayproj}
\begin{align}
\begin{split}
\widehat{(\mathbb Pf)'}(\xi,z_d)
& =\widehat{f'}(\xi,z_d)+\frac{i\xi}{2|\xi|}\int_0^{\infty} i\xi\cdot\widehat{f'}(\xi,s)\left[e^{-|z_d-s||\xi|}+e^{-(z_d+s)|\xi|}\right]ds\\
& \quad -\frac{i\xi}{2} \int_0^{z_d} \widehat{f_d}(\xi,s) e^{-(z_d-s)|\xi|} d \xi +\frac{i\xi}{2}\int_{z_d}^\infty \widehat{f_d}(\xi,s) e^{-(s-z_d)|\xi|} d \xi  \\
& \quad + \frac{i\xi}{2} \int_0^\infty \widehat{f_d}(\xi,s) e^{-(z_d+s)|\xi|} d \xi
\end{split}
\end{align}
and
\begin{align}
\begin{split}
\widehat{(\mathbb Pf)_d}(\xi,z_d)
& =-\frac{1}{2}\int_0^{z_d} i\xi\cdot\widehat{f'}(\xi,s) \left[e^{-(z_d-s)|\xi|}+e^{-(z_d+s)|\xi|}\right] ds \\
& \quad +\frac{1}{2}\int_{z_d}^\infty i\xi\cdot\widehat{f'}(\xi,s) \left[e^{(z_d-s)|\xi|}-e^{-(z_d+s)|\xi|}\right] ds \\
& \quad + \frac12 \int_0^\infty \widehat{f_d}(\xi,s) \left[ e^{-|z_d-s||\xi|} -e^{-(z_d+s)|\xi|} \right] d \xi .
\end{split}
\end{align}
\end{subequations}

\subsection{The Helmholtz-Leray projector and the divergence}
\label{sec.leraydiv}

In view of the application to the Navier-Stokes system, we need to analyze the  operator 
\begin{equation*}
F \in L^q_{uloc}\mapsto \mathbb P\nabla\cdot F =\left(\mathbb P_{\beta\gamma}\left(\partial_\alpha F_{\alpha\gamma}\right)\right)_{\beta=1,\ldots\, d},
\end{equation*}
rather than the Helmholtz-Leray projector $\mathbb P$ itself. Here we develop an approach similar to the one of Lemari\'e-Rieusset \cite[Chapter 11]{lemariebook}. An analogous method has also been used in other works concerned with non localized solutions of fluid equations, such as for instance \cite{TTY10} and \cite{AKLN15} reminiscent of Serfati's work \cite{Serf95}. In the setting of the whole space $\R^d$, let $\chi\in C^\infty_c(\R^d)$ be a cut-off in physical space which is supported in $B(0,2)$ and equal to $1$ on $B(0,1)$. The operator $\mathbb P\nabla\cdot$ is equal to $\nabla\cdot +\frac{D\otimes D\nabla\cdot }{-\Delta}$. The kernel $T_{\alpha\beta\gamma}$ of the operator $\frac{D_\alpha D_\beta D_\gamma}{-\Delta}$ is decomposed into
\begin{equation*}
T_{\alpha\beta\gamma}=\partial_\alpha\partial_\beta\left((1-\chi)T_\gamma\right)+\partial_\alpha\partial_\beta\left(\chi T_\gamma\right)=:A_{\alpha\beta\gamma}+\partial_\alpha\partial_\beta B_\gamma,
\end{equation*}
where $T_\gamma$ is the kernel associated with the operator $\frac{D_\gamma}{-\Delta}$. We have $A_{\alpha\beta\gamma}\in WL^\infty(\R^d)$ i.e. 
\begin{equation*}
\sum_{\eta\in\Z^{d-1}\times\mathbb Z_{\geq 0}}\sup_{\eta+(0,1)^d}|A_{\alpha\beta\gamma}|<\infty
\end{equation*}
and $B_\gamma\in L^1_{c}(\R^d)$. Hence, for any $1\leq p<\infty$, for all $f\in L^p_{uloc}(\R^d)$, $A_{\alpha\beta\gamma}*f\in L^p_{uloc}(\R^d)$ and $B_{\gamma}*f\in L^p_{uloc}(\R^d)$.

We now return to the case of $\R^d_+$. We first compute the action of $\mathbb P\nabla\cdot$ on $F\in C^\infty_c(\overline{\R^d_+})^{d^2}$. Assume that for all $\alpha, \gamma\in\{1,\ldots\, d\}$, for all $z'\in\R^{d-1}$,
\begin{equation}\label{e.cancelf}
F_{\alpha d}(z',0)=F_{d\gamma}(z',0)=\partial_d F_{dd}(z',0)=0.
\end{equation}
Notice that we will later take $F$ in product form, i.e. $F=u\otimes u$, but for now we stick to the general $F$ just described. For the tangential component, we have for all $\xi\in\R^{d-1}$, for all $z_d>0$, for all $\beta\in\{1,\ldots\,d-1\}$,
\begin{align*}
& \left\{\reallywidehat{\mathbb P\left(\partial_\alpha(F_{\alpha\cdot})\right)}\right\}_\beta(\xi,z_d)  =\reallywidehat{\mathbb P_{\beta\gamma}\left(\partial_\alpha(F_{\alpha\gamma})\right)}(\xi,z_d)\\
& =i\xi_\alpha\widehat{F_{\alpha\beta}}(\xi,z_d)+\partial_d(\widehat{F_{d\beta}})(\xi,z_d)\\
& \quad +\frac{i\xi_\beta}{2|\xi|}\int_0^\infty\big(-\xi_\gamma \xi_\alpha \widehat{F_{\alpha\gamma}}  + i \xi_\gamma\partial_d \widehat{F_{d \gamma}}  \big ) (\xi,s)  \left[e^{-|z_d-s||\xi|}+e^{-(z_d+s)|\xi|}\right] ds \\
&  \quad -\frac{i\xi_\beta}{2} \int_0^{z_d} \big(i\xi_\alpha \widehat{F_{\alpha d}}+\partial_d \widehat{F_{dd}}\big ) (\xi,s) e^{-(z_d-s)|\xi|} d \xi \\
& \quad +\frac{i\xi_\beta}{2}\int_{z_d}^\infty \big(i\xi_\alpha \widehat{F_{\alpha d}}+\partial_d \widehat{F_{dd}}\big ) (\xi,s) e^{-(s-z_d)|\xi|} d \xi  \\
& \quad + \frac{i\xi_\beta}{2} \int_0^\infty \big(i\xi_\alpha \widehat{F_{\alpha d}}+\partial_d \widehat{F_{dd}}\big ) (\xi,s)  e^{-(z_d+s)|\xi|} d \xi
\end{align*}
Hence, integrating by parts we get
\begin{align}\label{e.tancomp}
\left\{\reallywidehat{\mathbb P\left(\partial_\alpha(F_{\alpha\cdot})\right)}\right\}_\beta(\xi,z_d)
& =i\xi_\alpha\widehat{F_{\alpha\beta}}(\xi,z_d)+\partial_d\widehat{F_{d\beta}}(\xi,z_d)-i\xi_\beta\widehat{F_{dd}}(\xi,z_d)\\
& \quad -\frac{i\xi_\alpha\xi_\beta\xi_\gamma}{2|\xi|}\int_0^\infty\widehat{F_{\alpha\gamma}}(\xi,s)\left[e^{-|z_d-s||\xi|}+e^{-(z_d+s)|\xi|}\right]ds\nonumber\\
& \quad +\frac{i\xi_\beta|\xi|}{2}\int_0^\infty\widehat{F_{dd}}(\xi,s)\left[e^{-|z_d-s||\xi|}+e^{-(z_d+s)|\xi|}\right]ds\nonumber\\
&\quad -\frac{i\xi_\beta}{2}\int_0^{z_d} \big(i\xi_\gamma\widehat{F_{d\gamma}}+i\xi_\alpha\widehat{F_{\alpha d}}\big)(\xi,s)e^{-(z_d-s)|\xi|}ds\nonumber\\
& \quad +\frac{i\xi_\beta}{2}\int_{z_d}^\infty \big(i\xi_\gamma\widehat{F_{d\gamma}}+i\xi_\alpha\widehat{F_{\alpha d}}\big)(\xi,s)e^{-(s-z_d)|\xi|}ds\nonumber\\
&\quad +\frac{i\xi_\beta}{2}\int_{0}^\infty \big(i\xi_\gamma\widehat{F_{d\gamma}}+i\xi_\alpha\widehat{F_{\alpha d}}\big)(\xi,s)e^{-(z_d+s)|\xi|}ds.\nonumber
\end{align}
As for the vertical component, we have for all $\xi\in\R^{d-1}$, for all $z_d>0$,
\begin{align*}
&\left\{\reallywidehat{\mathbb P\left(\partial_\alpha(F_{\alpha\cdot})\right)}\right\}_d(\xi,z_d) =\reallywidehat{\mathbb P_{d\gamma}\left(\partial_\alpha(F_{\alpha\gamma})\right)}(\xi,z_d)\\
&= -\frac{1}{2}\int_0^{z_d} \big (-\xi_\gamma \xi_\alpha \widehat{F_{\alpha\gamma}}(\xi,s)+i\xi_\gamma\partial_d\widehat{F_{d\gamma}} \big ) (\xi,s) \left[e^{-(z_d-s)|\xi|}+e^{-(z_d+s)|\xi|}\right] ds \\
&\quad+\frac12 \int_{z_d}^\infty \big (-\xi_\gamma \xi_\alpha \widehat{F_{\alpha\gamma}}(\xi,s)+i\xi_\gamma\partial_d\widehat{F_{d\gamma}} \big ) (\xi,s) \left[e^{-(z_d-s)|\xi|}- e^{-(z_d+s)|\xi|}\right] ds\\
& \quad + \frac12 \int_0^\infty \big(i\xi_\alpha\widehat{F_{\alpha d}}+ \partial_d \widehat{F_{d d}}\big)(\xi,s) \left[e^{-(z_d-s)|\xi|}- e^{-(z_d+s)|\xi|}\right]  ds.
\end{align*}
Thus, again integrating by parts we have
\begin{align}\label{e.vertcomp}
&\quad \left\{\reallywidehat{\mathbb P\left(\partial_\alpha(F_{\alpha\cdot})\right)}\right\}_d(\xi,z_d)\\
& =-i\xi_\gamma\widehat{F_{d\gamma}}(\xi,z_d)  +\frac{1}{2}\int_0^{z_d}\Big(\xi_\alpha\xi_\gamma\widehat{ F_{\alpha\gamma}}-|\xi|^2\widehat{F_{dd}}\Big)(\xi,s)\left[e^{-(z_d-s)|\xi|}+e^{-(z_d+s)|\xi|}\right]ds\nonumber\\
& \quad -\frac{1}{2}\int_{z_d}^\infty\Big(\xi_\alpha\xi_\gamma\widehat{ F_{\alpha\gamma}}-|\xi|^2\widehat{F_{dd}}\Big)(\xi,s)\left[e^{-(s-z_d)|\xi|}-e^{-(z_d+s)|\xi|}\right]ds\nonumber\\
& \quad +\frac{|\xi|}{2}\int_0^{z_d}\Big(i\xi_\gamma\widehat{F_{d\gamma}}+i\xi_\alpha\widehat{F_{\alpha d}}\Big)(\xi,s)\left[e^{-(z_d-s)|\xi|}-e^{-(z_d+s)|\xi|}\right]ds\nonumber\\
& \quad +\frac{|\xi|}{2}\int_{z_d}^\infty\Big(i\xi_\gamma\widehat{F_{d\gamma}}+i\xi_\alpha\widehat{F_{\alpha d}}\Big)(\xi,s)\left[e^{-(s-z_d)|\xi|}-e^{-(z_d+s)|\xi|}\right]ds.\nonumber
\end{align}
Notice that the integrations by parts carried out above are in the same vein as the decomposition of the multiplier $R(\lambda)$ of the resolvent problem \eqref{e.resol} into a local part associated with the Dirichlet-Laplace operator and a nonlocal part coming from the pressure. This technique was introduced in \cite{DHP01}. In both situations, the goal is to get around the direct use of the Helmholtz-Leray projector $\mathbb P$.

We have to deal with several type of multipliers: for $\alpha,\, \beta,\, \gamma,\, \delta,\, \iota\in\{1,\ldots\,~d-1\}$, for $\xi\in\R^{d-1}$,
\begin{align}
\tag{type A}\label{e.termA}
\frac{\xi_\alpha\xi_\beta\xi_\gamma}{|\xi|} e^{-(z_d-s)|\xi|}\widehat{F_{\delta\iota}},&\\
\xi_\alpha\xi_\beta e^{-(z_d-s)|\xi|}\widehat{F_{\gamma\delta}}\quad \mbox{and}\quad \xi_\alpha|\xi|e^{-(z_d-s)|\xi|}\widehat{F_{\beta\gamma}},&\qquad\forall\xi\in\R^{d-1},\ z_d>s,\nonumber\\
\tag{type B}\label{e.termB}
\frac{\xi_\alpha\xi_\beta\xi_\gamma}{|\xi|} e^{-(s-z_d)|\xi|}\widehat{F_{\delta\iota}},&\\
\xi_\alpha\xi_\beta e^{-(s-z_d)|\xi|}\widehat{F_{\gamma\delta}}\quad \mbox{and}\quad \xi_\alpha|\xi|e^{-(s-z_d)|\xi|}\widehat{F_{\beta\gamma}},&\qquad\forall\xi\in\R^{d-1},\ s>z_d,\nonumber\\
\tag{type C}\label{e.termC}
\frac{\xi_\alpha\xi_\beta\xi_\gamma}{|\xi|} e^{-(s+z_d)|\xi|}\widehat{F_{\delta\iota}},&\\
\xi_\alpha\xi_\beta e^{-(s+z_d)|\xi|}\widehat{F_{\gamma\delta}}\quad \mbox{and}\quad \xi_\alpha|\xi|e^{-(s+z_d)|\xi|}\widehat{F_{\beta\gamma}},&\qquad\forall\xi\in\R^{d-1},\ s,\, z_d>0,\nonumber
\end{align}  
All the terms associated with the multipliers \eqref{e.termA}, \eqref{e.termB} and \eqref{e.termC} can be handled via Lemma \ref{lem.kernelsecdivleray'} below, which will allow us in Section \ref{sec.bilinear} to combine the operator $\mathbb P\nabla\cdot$ with the operator $(\lambda+{\bf A})^{-1}$.

We develop an idea similar to the one of Lemari\'e-Rieusset explained above, except that rather than cutting-off in physical space, we cut-off on the Fourier side. This appears to be more convenient for us, since we will have a decomposition based on the nonlocal operator $(-\Delta')^\frac{2-\theta}{2}$ instead of the local derivatives $\partial_\alpha\partial_\beta$. 

\begin{lemma}\label{lem.kernelsecdivleray'} Let $\chi\in C_0^\infty (\R^{d-1})$ be such that $\chi=1$ for $|\xi|\leq 2$ and $\chi=0$ for $|\xi|\geq 3$.  Let $m_2 \in C^\infty(\R^{d-1}\setminus\{0\})$ a multiplier in Fourier space homogeneous of order $2$ i.e. such that for all $t>0$, $\eta\in\R^{d-1}$, $m_2 (t\eta) =t^2 m_2 (\eta)$, and such that for all $\xi\in\R^{d-1}$, for all $\alpha\in\mathbb N^{d-1}$,
\begin{equation*}
\left|\partial^l m_2(\xi)\right|\leq |\xi|^{2-|\alpha|}.
\end{equation*}
Let $\theta\in [0,2]$ and let $K_2 \in C^\infty(\R^d\setminus\{0\})$ (resp. $K_\theta\in C^\infty(\R^d\setminus\{0\})$) be the kernel associated with $m_2(\xi)e^{-t|\xi|}$ (resp. $\frac{m_2(\xi)}{|\xi|^{2-\theta}}e^{-t|\xi|}$), i.e. for all $y'\in\R^{d-1}$ and $t>0$,
\begin{equation*}
K_2 (y',t):=\int_{\R^{d-1}}e^{iy'\cdot\xi}m_2 (\xi)e^{-t|\xi|}d\xi\quad\mbox{and}\quad K_\theta (y',t):=\int_{\R^{d-1}}e^{iy'\cdot\xi}\frac{m_2 (\xi)}{|\xi|^{2-\theta}}e^{-t|\xi|}d\xi.
\end{equation*}
Then, for each $\lambda\in \mathbb{C}\setminus \{0\}$ we can decompose $K_2$ in the following way:
\begin{equation}\label{e.identityK_2'}
K_2=(-\Delta')^\frac{2-\theta}{2} K_{\theta,\geq |\lambda|^{\frac12}}+K_{2,\leq |\lambda|^{\frac12}},
\end{equation}
where the symbol of $K_{\theta,\geq |\lambda|^{\frac12}}$ is $\big (1-\chi (|\lambda|^{-\frac12} \xi) \big ) \frac{m_2(\xi)}{|\xi|^{2-\theta}}e^{-t|\xi|}$, while the symbol of $K_{2,\leq |\lambda|^{\frac12}}$ is $\chi (|\lambda|^{-\frac12} \xi) m_2 (\xi) e^{-t|\xi|}$.
Moreover, there exists $C=C(d)>0$ such that, for all $y'\in\R^{d-1}$ and $t>0$,
\begin{align}
\left|K_{\theta,\geq |\lambda|^{\frac12}}(y',t)\right|&\leq \frac{Ce^{-t |\lambda|^\frac12}}{(|y'|+t)^{d-1+\theta}},\label{e.decayK_theta'}\\
\left|K_{2,\leq |\lambda|^{\frac12}}(y',t)\right|&\leq \frac{C}{(|\lambda|^{-\frac12} +|y'|+t)^{d+1}}\label{e.decayK_2'}.
\end{align}
\end{lemma}

\begin{proof} The decomposition \eqref{e.identityK_2'} simply follows from $1= \big (1-\chi (|\lambda|^{-\frac12} \xi) \big ) +\chi (|\lambda|^{-\frac12} \xi)$ and $m_2 (\xi) = |\xi|^{2-\theta} \frac{m_2 (\xi)}{|\xi|^{2-\theta}}$ for the ``high'' frequency part, and thus, we focus on the proof of \eqref{e.decayK_theta'} and \eqref{e.decayK_2'}.
Set 
\begin{align*}
m_{\theta,\geq |\lambda|^{\frac12}} (\xi,t)  & := \big (1-\chi (|\lambda|^{-\frac12} \xi) \big ) \frac{m_2(\xi)}{|\xi|^{2-\theta}}e^{-t|\xi|},\\
m_{2,\leq |\lambda|^\frac12} (\xi,t) & := \chi (|\lambda|^{-\frac12} \xi) m_2 (\xi) e^{-t|\xi|}.
\end{align*}
Then it is straightforward to show 
\begin{align*}
|\partial_\xi^\alpha m_{\theta,\geq |\lambda|^{\frac12}} (\xi,t)| & \leq C |\xi|^{\theta-\alpha} e^{-\frac{t}{2}|\xi|}  e^{-t|\lambda|^\frac12},\\
|\partial_\xi^\alpha m_{2,\leq |\lambda|^\frac12} (\xi,t) | & \leq C |\xi|^{2-\alpha} e^{-\frac{t}{2}|\xi|}.
\end{align*}
Hence, as for \eqref{e.decayK_theta'}, if $|y'|\geq t$, then the argument in Lemma \ref{lem.optdecay}  (i.e., introduce the cut-off in the Fourier side with the radius $R$ and optimize $R$ later as $R=|y'|^{-1}$) gives the bound 
\begin{align*}
\left|K_{\theta,\geq |\lambda|^{\frac12}}(y',t)\right|&\leq \frac{Ce^{-t |\lambda|^\frac12}}{|y'|^{d-1+\theta}},
\end{align*}
While if $t> |y'|$, we simply estimate as follows
\begin{align}\label{proof.lem.kernelsecdivleray'.1}
\left|K_{\theta,\geq |\lambda|^{\frac12}}(y',t)\right| \leq C \int_{|\xi|\geq 2|\lambda|^\frac12} |\xi|^\theta e^{-t|\xi|} d \xi \leq C t^{-d+1-\theta} e^{-t|\lambda|^\frac12}
\end{align}
by changing the variables $t\xi = \eta$. Thus, estimate \eqref{e.decayK_theta'} holds.
As for \eqref{e.decayK_2'}, let us first consider the case $|y'|+t\geq |\lambda|^{-\frac12}$. 
If $t>|y'|$ in addition, then the simple calculation as in \eqref{proof.lem.kernelsecdivleray'.1} gives the bound 
\begin{align*}
\left|K_{2,\leq |\lambda|^{\frac12}}(y',t)\right| & \leq C t^{-d-1}.
\end{align*}
While if $|y'|>t$, then the argument as in Lemma \ref{lem.optdecay} yields
\begin{align*}
\left|K_{2,\leq |\lambda|^{\frac12}}(y',t)\right|&\leq \frac{C}{|y'|^{d+1}}.
\end{align*}
Finally, if $|\lambda|^{-\frac12} >|y'|+t$, then we have 
\begin{align*}
\left|K_{2,\leq |\lambda|^{\frac12}}(y',t)\right| \leq C \int_{|\xi|\leq 3|\lambda|^\frac12} |\xi|^2 d\xi \leq C |\lambda|^{\frac{d+1}{2}}.
\end{align*}
Collecting these, we obtain \eqref{e.decayK_2'}. The proof is complete.
\end{proof}

We now estimate the action of $K_{\theta,\geq |\lambda|^\frac12}$ and $K_{2,\leq |\lambda|^\frac12}$ on functions in $L^q_{uloc}(\R^d_+)$.

\begin{lemma}\label{lem.termsABCscaled}
Let $p,q\in [1,\infty]$ satisfy
\begin{equation*}
1\leq q\leq p \leq \infty,\qquad 0\leq \frac1q-\frac1p<\frac1d.
\end{equation*}
Then there exist $\theta\in (0,1)$ and $C=C(d,p,q,\theta)>0$ such that for all $\lambda\in\mathbb C\setminus\{0\}$, $f\in L^q_{uloc}(\R^d_+)$, and $y_d>0$, we have
\begin{align*}
\left\|\int_0^{y_d}\int_{\R^{d-1}}K_{\theta,\geq |\lambda|^{\frac12}}(y'-z',y_d \pm z_d)f(z',z_d)dz'dz_d\right\|_{L^p_{uloc}}&\\
\leq C|\lambda|^{-\frac{1-\theta}{2}}\left(1+|\lambda|^{\frac d2(\frac1q-\frac1p)}\right)\|f\|_{L^q_{uloc}},&\\
\left\|\int_0^{y_d}\int_{\R^{d-1}}K_{2,\leq |\lambda|^{\frac12}} (y'-z',y_d\pm z_d)f(z',z_d)dz'dz_d\right\|_{L^q_{uloc}}&\leq C|\lambda|^{\frac12}\|f\|_{L^q_{uloc}},
\end{align*}
and 
\begin{align*}
\left\|\int_{y_d}^\infty\int_{\R^{d-1}}K_{\theta,\geq |\lambda|^{\frac12}}(y'-z',y_d\pm z_d)f(z',z_d)dz'dz_d\right\|_{L^p_{uloc}}&\\
\leq C|\lambda|^{-\frac{1-\theta}{2}}\left(1+|\lambda|^{\frac d2(\frac1q-\frac1p)}\right)\|f\|_{L^q_{uloc}},&\\
\left\|\int_{y_d}^\infty\int_{\R^{d-1}}K_{2,\leq |\lambda|^{\frac12}} (y'-z',y_d\pm z_d)f(z',z_d)dz'dz_d\right\|_{L^q_{uloc}} & \leq C|\lambda|^{\frac12}\|f\|_{L^q_{uloc}}.
\end{align*}
\end{lemma}

\begin{proof} Let $s\in [1,\infty]$ such that $\frac1p=\frac1q+\frac1s-1$.
By the condition $\frac1q-\frac1p<\frac1d$ we can take $\theta\in (0,1)$ so that $(d-1 + \theta) s<d$.
We fix such a $\theta\in (0,1)$ below.
To show the estimates stated in the lemma it suffices to consider the case with the variable $y_d-z_d$.
Then, in virtue of the bounds \eqref{e.decayK_theta'} and \eqref{e.decayK_2'} all terms are reduced to the estimate for the convolution $K_\lambda * f$ in $\R^d$ with $K_\lambda$ either 
\begin{align}\label{proof.lem.termsABCscaled.1}
K_\lambda (y',y_d) = \frac{C e^{-|y_d||\lambda|^\frac12}}{(|y'| +|y_d|)^{d-1+\theta}} \quad ({\rm for~the ~terms~involving}~~K_{\theta,\geq |\lambda|^{\frac12}}) 
\end{align}
or 
\begin{align}\label{proof.lem.termsABCscaled.2}
K_\lambda (y',y_d) = \frac{C}{(|\lambda|^{-\frac12} + |y'| + |y_d|)^{d+1}} \quad ({\rm for~the ~terms~involving}~~K_{2,\leq |\lambda|^{\frac12}}).
\end{align}
Note that $f$ is extended by zero to $\R^d$.
Then, the proof is parallel to that of Lemma \ref{prop:v}.
Without loss of generality it suffices to estimate the $L^p$ norm on the cube $B_0=(0,1)^d$.
The case when $K_\lambda$ is given by \eqref{proof.lem.termsABCscaled.2} is easily estimated. Indeed,
\begin{align*}
&\left(\int_{0}^{1}\int_{(0,1)^{d-1}}\left|\int_0^{z_d}\int_{\R^{d-1}}K_\lambda (y'-z',y_d-z_d)f(z',z_d)dz'dz_d\right|^qdy'dy_d\right)^{\frac 1q}\\
\leq&C\sum_{\eta\in\mathbb Z^d}\left(\int_{0}^{1}\int_{(0,1)^{d-1}}\left|\frac{1}{\big(|\lambda|^{-\frac12}+|\cdot|\big)^{d+1}}*_y\chi_\eta f\right|^qdy'dy_d\right)^{\frac 1q}\\
\leq&C\sum_{\eta\in\mathbb Z^d}\frac{1}{(|\lambda|^{-\frac12}+|\eta'|+|\eta_d|)^{d+1}}\|f\|_{L^q_{uloc}}\\
\leq& C\int_0^\infty\frac{r^{d-1}}{\big(|\lambda|^{-\frac 12}+r\big)^{d+1}}dr\leq C|\lambda|^{\frac12}\|f\|_{L^q_{uloc}}.
\end{align*}
Next we consider the case when $K_\lambda$ is given by \eqref{proof.lem.termsABCscaled.1}.
Let $\frac1p=\frac1s+\frac1q-1$.
Then, arguing as in the proof of Lemma \ref{prop:v}, we have from the Young inequality for convolution,
\begin{align*}
& \| K_\lambda *_y f \|_{L^p (B_0)} \\
& \leq  C \bigg ( \| K_\lambda \|_{L^s}  +  \sum_{\max |\beta_i'|\geq 3, \beta_d\in \Z} \| \chi_\beta K_\lambda\|_{L^s} + \sum_{\max |\beta_i'|\leq 3, |\beta_d|\geq 3} \| \chi_\beta K_\lambda\|_{L^s} \bigg )  \| f\|_{L^q_{uloc}}.
\end{align*}
Here $\chi_\beta$ is the characteristic function on the cube $B_\beta$ (see the proof of Lemma \ref{prop:v}). In virtue of the choice of $\theta\in (0,1)$ above, we see that $K_\lambda\in L^s(\R^d)$ and 
\begin{align*}
\| K_\lambda\|_{L^s}\leq C |\lambda|^{-\frac{1-\theta}{2} + \frac{d}{2} (\frac1q-\frac1p)}.
\end{align*}
Similarly, the direct computation yields 
\begin{align*}
\sum_{\max |\beta_i'|\geq 3, \beta_d\in \Z} \| \chi_\beta K_\lambda\|_{L^s}  \leq C \int_0^\infty e^{-c  t |\lambda|^\frac12} (1+t)^{-\theta} d t \leq C |\lambda|^{-\frac{1-\theta}{2}}
\end{align*}
and 
\begin{align*}
\sum_{\max |\beta_i'|\leq 3, |\beta_d|\geq 3} \| \chi_\beta K_\lambda\|_{L^s}  \leq C e^{-c |\lambda|^\frac12} \leq C |\lambda|^{-\frac{1-\theta}{2}}.
\end{align*}
The proof is complete.
\end{proof}

\subsection{The bilinear operator}
\label{sec.bilinear}

The goal of this section is to study the bilinear operator 
\begin{equation}\label{e.bilopuv}
(u,v)\mapsto(\lambda+{\bf A})^{-1}\mathbb P\nabla\cdot(u\otimes v).
\end{equation}
The idea is to combine the results on the operator $(\lambda+{\bf A})^{-1}$ obtained in Section \ref{sec.resolest} with the decomposition and estimates for $\mathbb P\nabla\cdot$ obtained in Section \ref{sec.leraydiv}.

Let $F:=u\otimes v$ with $u,\, v\in L^p_{uloc,\sigma} (\R^d_+)$. The outcome of the computations \eqref{e.tancomp} and \eqref{e.vertcomp} as well as Lemmas \ref{lem.kernelsecdivleray'}, \ref{lem.termsABCscaled} is the following proposition.

\begin{proposition}\label{prop.summarizebilin}
Let $\lambda\in S_{\pi-\epsilon}$ and let $u,\, v\in L^p_{uloc,\sigma}(\R^d_+)$.
Assume that $p,q\in [1,\infty]$ satisfy
\begin{equation*}
1< q\leq p \leq \infty\quad\mbox{or}\quad1\leq q<p\leq\infty,\qquad 0\leq \frac1q-\frac1p<\frac1d.
\end{equation*}
Then there exist $\theta\in (0,1)$ and $G_{\theta,\geq|\lambda|^{\frac12}} (u\otimes v),~ G_{\leq|\lambda|^{\frac12}} (u\otimes v)\in L^1_{loc}(\R^d_+;\R^d)$ such that, for all $\beta\in\{1,\ldots\, d-1\}$, 
\begin{subequations}\label{e.formulasPdiv}
\begin{align}
\begin{split}
\left\{\mathbb P\partial_\alpha(u_\alpha v)\right\}_\beta&=\partial_{\gamma} (u_{\gamma} v_\beta)+\partial_d(u_d v_\beta) - \partial_\beta (u_d v_d) \\
& \quad + (-\Delta')^\frac{2-\theta}{2} G^\beta_{\theta,\geq|\lambda|^{\frac12}} (u\otimes v) +G^\beta_{\leq|\lambda|^{\frac12}} (u\otimes v),
\end{split}\\
\begin{split}
\left\{\mathbb P\partial_\alpha(u_\alpha v)\right\}_d &=-\partial_\gamma (u_d v_\gamma)+ (-\Delta')^\frac{2-\theta}{2} G^d_{\theta, \geq|\lambda|^{\frac12}} (u\otimes v) +G^d_{\leq|\lambda|^{\frac12}} (u\otimes v),
\end{split}
\end{align}
\end{subequations}
where Einstein's convention is used (the sums run over $\alpha\in \{1,\ldots,d\}$ and $\gamma \in\{1,\ldots,\, d-1\}$), and that
\begin{align}
&\left\|G_{\theta,\geq|\lambda|^{\frac12}} (u\otimes v) \right\|_{L^p_{uloc}}\leq C|\lambda|^{-\frac{1-\theta}{2}}\left(1+|\lambda|^{\frac d2(\frac1q-\frac1p)}\right)\|u\otimes v\|_{L^q_{uloc}},\label{est.prop.summarizebilin.1}\\
&\left\|G_{\leq|\lambda|^{\frac12}} (u\otimes v) \right\|_{L^q_{uloc}}\leq C|\lambda|^{\frac12}\|u\otimes v\|_{L^q_{uloc}}. \label{est.prop.summarizebilin.2}
\end{align}
Here $C=C(d,\epsilon, p,q)>0$ is independent of $\lambda\in S_{\pi-\epsilon}$.
\end{proposition}

The only thing which remains to be done so as to estimate the bilinear operator \ref{e.bilopuv} is to combine the result of Proposition \ref{prop.summarizebilin} with the kernel bounds of Section \ref{sec.resolhp} and the estimates of Section \ref{sec.resolest}. Doing so we obtain the important result stated below.

\begin{proposition}\label{thm.estbilinop}
Let  $\lambda\in S_{\pi-\ep}$. For all $p,q\in [1,\infty]$ satisfying
\begin{equation*}
1< q\leq p \leq \infty\quad\mbox{or}\quad1\leq q<p\leq\infty,\qquad 0\leq \frac1q-\frac1p<\frac1d,
\end{equation*}
there exists $C=C(d,\ep,p,q)>0$ (independent of $\lambda$) such that for all $u,\, v\in L^p_{uloc,\sigma}(\R^d_+;\R^d)$,
\begin{align}
& \left\| (\lambda + {\bf A})^{-1} \mathbb P\nabla\cdot(u\otimes v)\right\|_{L^p_{uloc}} \leq C|\lambda|^{-\frac12}\left(1+|\lambda|^{\frac d2(\frac1q-\frac1p)}\right) \|u\otimes v\|_{L^q_{uloc}}.\label{e.estbilinearlosslambda.1}
\end{align}
Moreover, if $\nabla u, \nabla v \in L^q_{uloc}(\R^d)$ in addition, then 
\begin{align}
& \left\| \nabla (\lambda + {\bf A})^{-1} \mathbb P\nabla\cdot(u\otimes v)\right\|_{L^q_{uloc}} 
\leq  C|\lambda|^{-\frac12} \big ( \|u\nabla v\|_{L^q_{uloc}} + \| v\nabla u \|_{L^q_{uloc}}\big ).\label{e.estbilinearlosslambda.2}
\end{align}
\end{proposition}

\begin{remark}{\rm As for \eqref{e.estbilinearlosslambda.2}, we can also show 
\begin{align}
\begin{split}
& \left\| \nabla (\lambda + {\bf A})^{-1} \mathbb P\nabla\cdot(u\otimes v)\right\|_{L^q_{uloc}} \\
\leq &  C|\lambda|^{-\frac12} \big ( 1 + |\lambda|^{\frac{d}{2}(\frac1q-\frac1p)} \big )\big ( \|u\nabla v\|_{L^q_{uloc}} + \| v\nabla u \|_{L^q_{uloc}}\big ) \label{e.estbilinearlosslambda.2'}
\end{split}
\end{align}
for $1< q\leq p \leq \infty$ or $1\leq q<p\leq\infty$ satisfying in addition $0\leq \frac1q-\frac1p<\frac1d$. 
The proof is the same as in the case $p=q$, but we state the proof only for the simplest case \eqref{e.estbilinearlosslambda.2} in this paper.
}
\end{remark}

\begin{proof}[Proof of Proposition \ref{thm.estbilinop}] (i) Proof of \eqref{e.estbilinearlosslambda.1}. All the ingredients are already proved, we just have to indicate how they fit together. The key point are the formulas \eqref{e.formulasPdiv}. The idea is that we integrate by parts in the formulas \eqref{def:vw} and then we use the estimates of Proposition \ref{l.kernel}, Proposition \ref{prop.estkernel}, Proposition \ref{cor.dirichlet-laplace}, Proposition \ref{prop.estnonlocal} and Proposition \ref{prop.summarizebilin}.
The estimates of the non localized Lebesgue norms follow exactly from the bounds in Section \ref{sec.resolest}.
In particular, we recall the the resolvent $(\lambda + {\bf A})^{-1}$ consists of the Dirichlet-Laplace part (Subsection \ref{sec.dirichlet-laplace}),  and the nonlocal part (Subsection \ref{sec.estnonlocal}), that is,
\begin{align*}
(\lambda+{\bf A})^{-1} = {\bf R}_{D.L.}(\lambda) + {\bf R}_{n.l.} (\lambda).
\end{align*}
The operators ${\bf R}_{D.L.}(\lambda)$ and ${\bf R}_{n.l.}(\lambda)$ respectively satisfy the estimates in 
Proposition \ref{cor.dirichlet-laplace} and Proposition \ref{prop.estnonlocal}, that is, for $\alpha=0,1$,
\begin{align*}
\| \nabla^\alpha {\bf R}_{D.L.}(\lambda) f\|_{L^p_{uloc}} + \| \nabla^\alpha {\bf R}_{n.l.} (\lambda) f\|_{L^p_{uloc}} \leq C |\lambda|^{-\frac{2-\alpha}{2}} \big ( 1 + |\lambda|^{\frac{d}{2}(\frac1q-\frac1p)} \big ) \| f\|_{L^q_{uloc}},
\end{align*}
and 
\begin{align*}
\| m_\alpha (D') {\bf R}_{D.L.}(\lambda) f\|_{L^q_{uloc}} + \| m_\alpha (D') {\bf R}_{n.l.} (\lambda) f\|_{L^q_{uloc}}  & \leq C |\lambda|^{-\frac{2-\alpha}{2}} \| f\|_{L^q_{uloc}},\\
& \qquad \alpha\in (0,2),\\
\| m_\alpha (D') \nabla {\bf R}_{D.L.}(\lambda) f\|_{L^q_{uloc}} + \| m_\alpha (D') \nabla {\bf R}_{n.l.} (\lambda) f\|_{L^q_{uloc}}  & \leq C |\lambda|^{-\frac{1-\alpha}{2}} \| f\|_{L^q_{uloc}},\\
& \qquad \alpha\in (0,1).
\end{align*}
Here $m_\alpha (D')$ is any Fourier multiplier of homogeneous order $\alpha$. 
Moreover, we also have 
\begin{align}\label{proof.thm.estbilinop.1}
\| {\bf R}_{D.L.}(\lambda) \partial_d f \|_{L^p_{uloc}} + \| {\bf R}_{n.l.}(\lambda)\partial_d f\|_{L^p_{uloc}} & \leq C |\lambda|^{-\frac12} \big ( 1 + |\lambda|^{\frac{d}{2}(\frac1q-\frac1p)} \big ) \| f\|_{L^q_{uloc}},
\end{align}
and 
\begin{align*}
\| m_\alpha (D') {\bf R}_{D.L.}(\lambda) \partial_d f \|_{L^q_{uloc}} + \| m_\alpha (D') {\bf R}_{n.l.}(\lambda)\partial_d f\|_{L^q_{uloc}} & \leq C |\lambda|^{-\frac{1-\alpha}{2}}  \| f\|_{L^q_{uloc}}, \\
&\qquad  \alpha \in (0,1),
\end{align*}
if $f\in C_0^\infty (\R^d_+;\R^d)$, by the integration by parts in \eqref{def:vw} and applying the derivative estimates of the associated kernels. 
Thus,  ${\bf R}_{D.L.}(\lambda) \partial_d$ and ${\bf R}_{n.l.}(\lambda)\partial_d$ are extended to bounded operators from $L^q_{uloc}(\R^d_+;\R^d)$ to $L^p_{uloc}(\R^d_+; \R^d)$ with $p,q$ satisfying $0\leq \frac1q-\frac1p<\frac1d$ together with the bound \eqref{proof.thm.estbilinop.1}. Indeed any function in $L^q_{uloc}(\R^d_+)$ is approximated by a sequence of functions in $C_0^\infty (\R^d_+)$ in the topology of $L^q_{loc}(\R^d_+)$ with a uniform bound in the norm of $L^q_{uloc}(\R^d_+)$. 
This extension with the estimate \eqref{proof.thm.estbilinop.1} is applied to the term $\partial_d(u_d v_\beta)$ in the formula \eqref{e.formulasPdiv}. This concludes the proof of \eqref{e.estbilinearlosslambda.1}.

\noindent (ii) Proof of \eqref{e.estbilinearlosslambda.2}. We first observe that the proof of \eqref{e.estbilinearlosslambda.1} in fact ensures the existence of the number $\delta_0 \in (0,1)$ such that for any $\delta\in (0,\delta_0]$,
\begin{align}\label{proof.thm.estbilinop.2}
\| m_\delta (D') (\lambda + {\bf A})^{-1} \mathbb P\nabla\cdot(u\otimes v) \|_{L^q_{uloc}} \leq C|\lambda|^{-\frac{1-\delta}{2}}  \|u\otimes v\|_{L^q_{uloc}},
\end{align}
where $m_\delta (D')$ is any  Fourier-multiplier of homogeneous order $\delta$.
Indeed, for $\theta\in (0,1)$ in Proposition \ref{prop.summarizebilin} we can take $\delta_0$ such that $\delta_0\in (0,\theta)$. We will use \eqref{proof.thm.estbilinop.2} later.

Since the tangential derivatives commute with $(\lambda+{\bf A})^{-1}$ and $\mathbb{P}$, estimate \eqref{e.estbilinearlosslambda.1} implies 
\begin{align*}
\| \nabla' (\lambda + {\bf A})^{-1} \mathbb P\nabla\cdot(u\otimes v) \|_{L^q_{uloc}} \leq C|\lambda|^{-\frac{1}{2}}  \|\nabla' (u\otimes v )\|_{L^q_{uloc}},
\end{align*}
and thus, combining with \eqref{proof.thm.estbilinop.2}, we also have 
\begin{align}\label{proof.thm.estbilinop.3'}
\| m_\delta (D') \nabla' (\lambda + {\bf A})^{-1} \mathbb P\nabla\cdot(u\otimes v) \|_{L^q_{uloc}} \leq C|\lambda|^{-\frac{1-\delta}{2}}  \|\nabla' (u\otimes v )\|_{L^q_{uloc}}
\end{align}
for $\delta\in (0,\delta_0]$.
Next we consider the estimate of $\partial_d  (\lambda + {\bf A})^{-1} \mathbb P\nabla\cdot(u\otimes v)$. 
Set $U = (\lambda + {\bf A})^{-1} \mathbb P\nabla\cdot(u\otimes v)$. 
Then the divergence-free condition implies $\partial_d U_d = -\nabla' \cdot U'$, and hence, the estimate of $\partial_d U_d$ follows from the estimate for the tangential derivatives which are already shown.
It remains to estimate $\partial_d U'$. To this end, we note that, by regarding the associated pressure $\nabla' p$ as the source term, the vector $U'$ is also written as 
\begin{align*}
U_\beta = (\lambda + \Delta_D)^{-1} \bigg ( \partial_\beta p + \nabla \cdot (u  v_\beta) \bigg ). 
\end{align*}
for $\beta\in \{1,\ldots,d-1\}$. Here $(\lambda + \Delta_D)^{-1}$ denotes the resolvent for the Dirichlet-Laplace operator (hence, ${\bf R}_{D.L.}(\lambda)$ in the proof of \eqref{e.estbilinearlosslambda.1} above), for which we have already established the estimates in $L^p_{uloc}$ spaces in Section \ref{sec.dirichlet-laplace}.
In particular, we have 
\begin{align*}
\|\partial_d (\lambda + \Delta_D)^{-1}  \nabla \cdot (u  v_\beta) \|_{L^q_{uloc}} \leq \frac{C}{|\lambda|^\frac12} \| \nabla \cdot (u  v_\beta)\|_{L^q_{uloc}}.
\end{align*}
As for the term $\partial_d (\lambda + \Delta_D)^{-1}  \partial_\beta p$ we have from the integration by parts in the kernel representation,
\begin{align}\label{proof.thm.estbilinop.5}
\partial_d (\lambda + \Delta_D)^{-1} \partial_\beta p & = (\lambda+\Delta_N)^{-1} \partial_\beta \partial_d p \nonumber \\
& = \partial_\beta  (\lambda+\Delta_N)^{-1} \bigg ( -\lambda U_d + \Delta U_d + \nabla \cdot ( u v_d) \bigg ).
\end{align}
Here $(\lambda+\Delta_N)^{-1}$ denotes the resolvent of the Neumann-Laplace operator,
for which we have clearly the same estimates as for the resolvent of the Dirichlet-Laplace operator,
since the argument in Section \ref{sec.dirichlet-laplace} is based only on the pointwise estimates of the kernel function. Hence the first term of the right-hand side of \eqref{proof.thm.estbilinop.5} is estimated as 
\begin{align*}
\| \lambda \partial_\beta (\lambda+\Delta_N)^{-1}  U_d\|_{L^q_{uloc}} \leq C\| \partial_\beta U_d \|_{L^q_{uloc}} \leq C|\lambda|^{-\frac12} \| \partial_\beta (u \otimes  v ) \|_{L^q_{uloc}},
\end{align*} 
and the third term is estimated as 
\begin{align*}
\| \partial_\beta (\lambda+\Delta_N)^{-1}  \nabla \cdot (u v_d) \|_{L^q_{uloc}} \leq C|\lambda|^{-\frac12} \| \nabla\cdot  (u  v_d ) \|_{L^q_{uloc}}.
\end{align*} 
Finally, we note that $\Delta U_d = \Delta' U_d - \partial_d \nabla' \cdot U'$.
Then 
\begin{align*}
\| \partial_\beta  (\lambda+\Delta_N)^{-1} \Delta' U_d \|_{L^q_{uloc}} & = \| (-\Delta')^\frac{2-\delta}{2} (\lambda+\Delta_N)^{-1} (-\Delta')^\frac{\delta}{2} \partial_\beta U_d \|_{L^q_{uloc}} \\
& \leq C |\lambda|^{-\frac{\delta}{2}} \| (-\Delta')^\frac{\delta}{2} \partial_\beta U_d \|_{L^q_{uloc}}\\
& \leq C \| \lambda|^{-\frac12} \| \partial_\beta (u\otimes v)\|_{L^q_{uloc}},
\end{align*} 
and similarly,
\begin{align*}
\| \partial_\beta  (\lambda+\Delta_N)^{-1} \partial_d \nabla' \cdot U' \|_{L^q_{uloc}} & = \| (-\Delta')^\frac{1-\delta}{2} \partial_d (\lambda + \Delta_D)^{-1} \partial_\beta (-\Delta')^{-\frac{1-\delta}{2}} \nabla' \cdot U' \|_{L^q_{uloc}} \\
& \leq C |\lambda|^{-\frac{\delta}{2}} \| \partial_\beta (-\Delta')^{-\frac{1-\delta}{2}} \nabla' \cdot U' \|_{L^q_{uloc}}\\
& \leq C |\lambda|^{-\frac12} \| \nabla' (u v ) \|_{L^q_{uloc}}.
\end{align*} 
Here we have regarded $\partial_\beta (-\Delta')^{-\frac{1-\delta}{2}}$ as the Fourier multiplier of homogeneous order $\delta$ and applied the estimate \eqref{proof.thm.estbilinop.3'}. The proof of \eqref{e.estbilinearlosslambda.2} is complete. 
\end{proof}

It remains to transfer the stationary bounds of Proposition \ref{thm.estbilinop} to the non stationary operator $e^{-t {\bf A}} \mathbb{P} \nabla \cdot$. These bounds are stated in Theorem \ref{prop.L^p_u-L^q_u.semigroup.inhomo} in the Introduction. We now prove this theorem.

\begin{proof}[Proof of Theorem \ref{prop.L^p_u-L^q_u.semigroup.inhomo}] (i) Proof of \eqref{est.prop.L^p_u-L^q_u.semigroup.inhomo.1}. By the semigroup property $e^{-t {\bf A}} = e^{-\frac{t}{2} {\bf A}}e^{-\frac{t}{2} {\bf A}}$ and $\| \nabla e^{-\frac{t}{2}{\bf A}} f \|_{L^p_{uloc}}\leq C t^{-\frac12} \| f\|_{L^p_{uloc}}$, it suffices to consider the case $\alpha=0$. 
As in the estimate of $e^{-t{\bf A}}$, we use the formula
\begin{align*}
e^{-t{\bf A}} \mathbb{P}\nabla \cdot (u\otimes v ) = \frac{1}{2\pi i} \int_\Gamma e^{t\lambda} (\lambda+{\bf A})^{-1} \mathbb{P}\nabla \cdot (u\otimes v) d\lambda.
\end{align*}
Here the curve $\Gamma$ is taken as in the proof of Proposition \ref{prop.L^p_u-L^q_u.semigroup}.
Assume that $p,q\in [1,\infty]$ satisfy $0\leq \frac1q-\frac1p <\frac1d$.
Then we have from \eqref{e.estbilinearlosslambda.1},
\begin{align*}
\| e^{-t{\bf A}} \mathbb{P}\nabla \cdot (u\otimes v ) \|_{L^p_{uloc}} & \leq 
C \int_\Gamma |e^{t\lambda}| |\lambda|^{-\frac12} \big ( 1+|\lambda|^{\frac{d}{2}(\frac1q-\frac1p)} \big ) |d\lambda| \, \| u\otimes v\|_{L^q_{uloc}}\\
& \leq C t^{-\frac12} \big ( 1+ t^{-\frac{d}{2}(\frac1q-\frac1p)} \big ) \| u\otimes v\|_{L^q_{uloc}},
\end{align*}
as in the computation of \eqref{proof.prop.L^p_u-L^q_u.semigroup.4}.
For general $p,q$ we use the same trick as in \eqref{proof.prop.L^p_u-L^q_u.semigroup.4},
and then combine the estimate of $(\lambda+{\bf A})^{-1}$ and $(\lambda+{\bf A})^{-1}\mathbb{P}\nabla \cdot$. The details are omitted here. The proof of \eqref{est.prop.L^p_u-L^q_u.semigroup.inhomo.1} is complete.

\noindent (ii) Proof of \eqref{est.prop.L^p_u-L^q_u.semigroup.inhomo.2}. In this case it suffices to use \eqref{e.estbilinearlosslambda.2} instead of \eqref{e.estbilinearlosslambda.1}. Thus we omit the details. The proof is complete.
\end{proof}

\section{Mild solutions in $L^q_{uloc,\sigma}$, $q\geq d$}
\label{sec.nse}

In this section we consider the Navier-Stokes equations in $\R^d_+$
\begin{equation}\label{eq.ns.1}
  \left\{
\begin{aligned}
 \partial_t u - \Delta u + \nabla p & = -\nabla \cdot   (u \otimes u ) , \quad \nabla \cdot u = 0  \qquad \mbox{in}~ (0,T)\times \R^d_+\,, \\
u & = 0\quad \mbox{on}~ (0,T)\times \partial\R^d_+,  \qquad u|_{t=0}=u_0 \quad \mbox{in}~\partial \R^d_+. 
\end{aligned}\right.
\end{equation}
The corresponding integral equation is 
\begin{align}\label{eq.mild.ns}
u(t)  = e^{-t{\bf A}} u_0 - \int_0^t e^{-(t-s){\bf A}} \mathbb{P} \nabla \cdot (u\otimes u ) d s  \,, \qquad t>0\,,
\end{align}
and the solution to this integral equation is called the mild solution. The existence of such solutions was pioneered by Fujita and Kato \cite{FK64}. 
Our objective is to prove the short time existence of the mild solution for non decaying data. In view of the scaling of the equation, a natural class for the initial data is $L^q_{uloc,\sigma}(\R^d_+)$ with $q\geq d$.
In principle, one can prove the existence in a short time for arbitrary size of initial data if $q>d$, and locally in time for small data if $q=d$. We note that, contrary to the $L^d$ space, the global existence for small data in $L^d_{uloc}$ is not expected in this functional framework since $L^d_{uloc}$ contains nondecaying functions.
Another issue in the framework of $L^q_{uloc}$ spaces is the meaning of the initial condition, for the domain of ${\bf A}$ is not dense in $L^q_{uloc,\sigma}(\R^d_+)$. 
The convergence to the initial data as $t\rightarrow 0$ in the topology of $L^q_{uloc}(\R^d_+)$  is achieved  if and only if the initial data is taken from $\overline{D({\bf A})}^{L^q_{uloc}}$, where $D({\bf A})$ denotes the domain of the Stokes operator in $L^q_{uloc,\sigma} (\R^d_+)$. Thanks to the result of \cite[Theorem 4.3]{DHP01} on the $L^\infty$ theory of the Stokes operator in the half space, the embedding $L^\infty (\R^d_+)\subset L^q_{uloc}(\R^d_+)$ implies that 
\begin{align*}
\overline{BUC_\sigma (\R^d_+)}^{L^q_{uloc}} \subset \overline{D({\bf A})}^{L^q_{uloc}}\,, \qquad \overline{C_{0,\sigma}^\infty (\R^d_+)}^{L^q_{uloc}} \subset \overline{D({\bf A})}^{L^q_{uloc}}\,.
\end{align*}
If the initial data is taken from these spaces with $q\geq d$ (the case $q=d$ is allowed), 
then by using the density argument one can show the short time existence of the mild solution which satisfies the initial condition in the topology of $L^q_{uloc}(\R^d_+)$. These facts are now quite standard. In this paper we state the local existence results for \eqref{eq.mild.ns}  
without going into the details on the meaning of the initial condition. 

\subsection{Existence of mild solutions for initial data in $L^q_{uloc,\sigma}(\R^d_+)$, $q>d$}
\label{sec.appB}

\begin{proposition}\label{prop.mild.ns} 
For any $q>d$ and $u_0\in L^q_{uloc,\sigma}(\R^d_+)$ there exists a unique mild solution $u\in L^\infty (0,T; L^q_{uloc,\sigma} (\R^d_+)) \cap C((0,T); W^{1,q}_{uloc,0}(\R^d_+)^d \cap BUC_\sigma (\R^d_+))$ such that 
\begin{align*}
\sup_{0<t<T} \big ( \| u (t) \|_{L^q_{uloc}} + 
t^\frac{d}{2q} \| u(t) \|_{L^\infty} + t^\frac12 \| \nabla u(t) \|_{L^q_{uloc}} \big ) \leq C_* \| u_0 \|_{L^q_{uloc}}\,.
\end{align*}
Moreover there exists a constant $\gamma>0$ depend only on $d$ and $q$
such that $T$ can taken be as 
$$
T^{\frac12+\frac{d}{2q}}+T^{\frac12-\frac{d}{2q}} \ge 
\frac{\gamma}{\|u_0\|_{L^q_{uloc}}}.
$$
\end{proposition}

\begin{proof} The proof is  
based on the standard Banach fixed point theorem. 
Set $\|f\|_T$ as
\begin{align*}
\| f\|_T = \sup_{0<t<T} \big ( \| f (t) \|_{L^q_{uloc}} + t^\frac{d}{2q} \| f(t) \|_{L^\infty} + t^\frac12 \| \nabla f(t) \|_{L^q_{uloc}} \big ).
\end{align*}
Let $C_0>0$ be a constant such that 
\begin{align*}
\| e^{-\cdot {\bf A}} f \|_T \leq C_0 ( 1+ T^\frac{d}{2q}) \| f\|_{L^q_{uloc}}\,, \qquad f\in L^q_{uloc,\sigma}(\R^d_+),
\end{align*}
which is well-defined in virtue of Proposition \ref{prop.L^p_u-L^q_u.semigroup}. Then let us introduce the set 
\begin{align*}
\begin{split}
X_T & =\Big \{ f\in L^\infty (0,T; L^q_{uloc,\sigma} (\R^d_+)) \cap C((0,T); W^{1,q}_{uloc,0}(\R^d_+; \R^d) \cap BUC_\sigma (\R^d_+)) ~|~ \\
& \qquad \| f\|_T \leq  2 C_0 (1+T^\frac{d}{2q}) \| u_0 \|_{L^q_{uloc}} \Big \}.
\end{split}
\end{align*}
For each $f\in X_T$ we define the map $\Phi [f] (t) =e^{-t{\bf A}} u_0 + B [f,f] (t)$, where
\begin{align*}
B [f,g] (t) = - \int_0^t e^{-(t-s){\bf A}} \mathbb{P} \nabla \cdot (f\otimes g ) d s,\qquad t>0, ~f,g\in X_T.
\end{align*}
We will show that if $T$ is sufficiently small, then $\Phi$ defines a contraction map in $X_T$. Theorem \ref{prop.L^p_u-L^q_u.semigroup.inhomo} yields for $f,g\in X_T$,
\begin{align}
\label{proof.thm.mild.ns.5}
\| B [f,g] (t) \|_{L^q_{uloc}} & \leq C \int_0^t (t-s)^{-\frac12} \| f\otimes g \|_{L^q_{uloc}} d s \nonumber \\
& \leq C \int_0^t (t-s)^{-\frac12} s^{-\frac{d}{2q}} d s \sup_{0<s<t} s^\frac{d}{2q} \| f(s) \|_{L^\infty} \sup_{0<s<t} \| g(s) \|_{L^q_{uloc}}.
\end{align}
Similarly, we have for $f,g\in X_T$, 
\begin{align}\label{proof.thm.mild.ns.6}
\| B[f,g] (t) \|_{L^\infty}\leq\ & C\int^t_0 (t-s)^{-\frac12} 
\big ( (t-s)^{-\frac{d}{2q}} + 1\big ) \| f\otimes g (s) \|_{L^q_{uloc}} \nonumber \\
 \leq\ &C (t^{\frac12-\frac dq} + t^{\frac12-\frac d{2q}}) 
 \sup_{0<s<t} s^\frac{d}{2q} \| f(s) \|_{L^\infty}  
 \sup_{0<s<t}  \big ( \| g(s) \|_{L^q_{uloc}} + s^\frac{d}{2q} \| g(s) \|_{L^\infty} \big ),
\end{align}
and 
\begin{align}\label{proof.thm.mild.ns.7}
\| \nabla B[f,g] (t)\|_{L^q_{uloc}}
 \leq\ & C\int_0^\frac{t}{2} (t-s)^{-1}  \| f\otimes g (s) \|_{L^q_{uloc}}ds \nonumber \\
 &\quad + C \int_\frac{t}{2}^t (t-s)^{-\frac12} \big ( \| f\nabla g  (s)\|_{L^q_{uloc}} + \| g\nabla f (s)\|_{L^q_{uloc}}\big ) d s\\
\begin{split}
 \leq\ &C t^{-\frac{d}{2q}}\bigg (  \sup_{0<s<t} s^\frac{d}{2q} \| f(s) \|_{L^\infty} \sup_{0<s<t}  \| g(s) \|_{L^q_{uloc}} \\
&\quad +
\sup_{0<s<t} s^\frac{d}{2q} \| f(s) \|_{L^\infty} 
\sup_{0<s<t} s^\frac12 \| \nabla g(s) \|_{L^q_{uloc}}\\
& \quad + 
\sup_{0<s<t} s^\frac12 \| \nabla f(s)\|_{L^q_{uloc}} 
\sup_{0<s<t} s^\frac{d}{2q} \| g(s) \|_{L^\infty} \bigg ).
\end{split}
\end{align}
Thus we obtain 
\begin{align*}
\| B[f,g] \|_T\leq C_1 (T^{\frac12 -\frac{d}{2q}} + T^\frac12) \| f\|_T 
\| g\|_T, \qquad f,g \in X_T.
\end{align*}
The continuity in time for $t\in (0,T)$ also follows from that of $f,g$ in $(0,T)$, and we skip the details here.
If $T$ is small so that 
\begin{align}\label{proof.thm.mild.ns.9}
C_1 (T^{\frac12 -\frac{d}{2q}} + T^\frac12) 2 C_0 (1 + T^\frac{d}{2q})
\|u_0\|_{L^q_{uloc}} \leq \frac14,
\end{align}
then \eqref{proof.thm.mild.ns.9} and the definition of $C_0$ imply that $\Phi$ defines a contraction map from $X_T$ into $X_T$. Hence, there exists a unique fixed point $u$ of $\Phi$ in $X_T$, which is the unique mild solution to \eqref{eq.ns.1} in $X_T$.
\end{proof}

\subsection{Existence of mild solutions for initial data in $L^d_{uloc,\sigma}(\R^d_+)$}

The first result is stated in any dimension $d\geq 2$. Below we define $W^{1,d}_{uloc,0}(\R^d_+)$ as $W^{1,d}_{uloc,0}(\R^d_+) = \{f\in L^d_{uloc}(\R^d_+)~|~\nabla f\in L^d_{uloc}(\R^d_+), \quad f|_{x_d=0} =0\}$. 

\begin{proposition}\label{thm.mild.ns} For any $T>0$ there exist $\epsilon,\, C_*>0$ such that the following statement holds.
For any $u_0\in L^d_{uloc,\sigma}(\R^d_+)$ satisfying $\|u_0 \|_{L^d_{uloc}}\leq \epsilon$ there exists  a unique mild solution $u\in L^\infty (0,T; L^d_{uloc,\sigma} (\R^d_+)) \cap C((0,T); W^{1,d}_{uloc,0}(\R^d_+)^d \cap BUC_\sigma (\R^d_+))$ such that 
\begin{align*}
\sup_{0<t<T} \big ( \| u (t) \|_{L^d_{uloc}} + t^\frac12 \| u(t) \|_{L^\infty} + t^\frac12 \| \nabla u(t) \|_{L^d_{uloc}} \big ) \leq C_* \| u_0 \|_{L^d_{uloc}}\,.
\end{align*}
If $u_0 \in \overline{D({\bf A})}^{L^d_{uloc}}$ in addition, then $\lim_{t\rightarrow 0} u(t)=u_0$ in $L^d_{uloc}(\R^d_+)^d$.
\end{proposition}

\begin{remark}{As usual, by using the density argument, we do not need to assume the smallness of $\|u_0\|_{L^d_{uloc}}$ to show the short time existence of the mild solution if  $u_0$ belongs of $\overline{D({\bf A})}^{L^d_{uloc}}$. 
}
\end{remark}

\begin{proof}  The proof is based on the standard Banach fixed point theorem. We fix $T>0$.
Set $\|f\|_T$ as
\begin{align*}
\| f\|_T := \sup_{0<t<T} \big ( \| f (t) \|_{L^d_{uloc}} + t^\frac12 \| f(t) \|_{L^\infty} + t^\frac12 \| \nabla f(t) \|_{L^d_{uloc}} \big ).
\end{align*}
Let $C_0>0$ be a constant such that 
\begin{align*}
\| e^{-\cdot {\bf A}} f \|_T \leq C_0 ( 1+ T^\frac12) \| f\|_{L^d_{uloc}}\,, \qquad f\in L^d_{uloc,\sigma}(\R^d_+),
\end{align*}
which is well-defined in virtue of Proposition \ref{prop.L^p_u-L^q_u.semigroup}.
Then let us introduce the set 
\begin{align*}
\begin{split}
X_T & :=\Big \{ f\in L^\infty (0,T; L^d_{uloc,\sigma} (\R^d_+)) \cap 
C((0,T); W^{1,d}_{uloc,0}(\R^d_+; \R^d) \cap BUC_\sigma (\R^d_+)) ~|~ \\
& \qquad \| f\|_T \leq  2 C_0 (1+T^\frac12) \| u_0 \|_{L^d_{uloc}} \Big \}.
\end{split}
\end{align*}
For each $f\in X_T$ we define the map $\Phi [f] (t) =e^{-t{\bf A}} u_0 + B [f,f] (t)$, where
\begin{align*}
B [f,g] (t) = - \int_0^t e^{-(t-s){\bf A}} \mathbb{P} \nabla \cdot (f\otimes g ) d s,\qquad t>0, ~f,g\in X_T.
\end{align*}
We will show that if $\|u_0\|_{L^d_{uloc}}\leq \epsilon$ and $\epsilon>0$ is sufficiently small then $\Phi$ defines a contraction map in $X_T$. Theorem \ref{prop.L^p_u-L^q_u.semigroup.inhomo} yields for $f,g\in X_T$,
\begin{align}\label{proof.thm.mild.ns.5}
\| B [f,g] (t) \|_{L^d_{uloc}} & \leq C \int_0^t (t-s)^{-\frac12} \| f\otimes g \|_{L^d_{uloc}} d s \nonumber \\
& \leq C \int_0^t (t-s)^{-\frac12} s^{-\frac12} d s \sup_{0<s<t} s^\frac12 \| f(s) \|_{L^\infty} \sup_{0<s<t} \| g(s) \|_{L^d_{uloc}}.
\end{align}
Similarly, we have for $f,g\in X_T$, 
\begin{align}\label{proof.thm.mild.ns.6}
\| B[f,g] (t) \|_{L^\infty}\leq & C\int_0^\frac{t}{2} (t-s)^{-\frac12} \big ( (t-s)^{-\frac12} + 1\big ) \| f\otimes g (s) \|_{L^d_{uloc}} \nonumber \\
& \qquad  + C \int_\frac{t}{2}^t (t-s)^{-\frac12} \| f\otimes g  (s)\|_{L^\infty} d s\nonumber \\
 \leq &C (t^{-\frac12} + 1) \sup_{0<s<t} s^\frac12 \| f(s) \|_{L^\infty}  \sup_{0<s<t}  \big ( \| g(s) \|_{L^d_{uloc}} + s^\frac12 \| g(s) \|_{L^\infty} \big ),
\end{align}
and 
\begin{align}\label{proof.thm.mild.ns.7}
\| \nabla B[f,g] (t)\|_{L^d_{uloc}}
 \leq & C\int_0^\frac{t}{2} (t-s)^{-1}  \| f\otimes g (s) \|_{L^d_{uloc}}ds \nonumber \\
 &\quad + C \int_\frac{t}{2}^t (t-s)^{-\frac12} \big ( \| f\nabla g  (s)\|_{L^d_{uloc}} + \| g\nabla f (s)\|_{L^d_{uloc}}\big ) d s\\
\begin{split}
 \leq &C (t^{-\frac12} + 1)\bigg (  \sup_{0<s<t} s^\frac12 \| f(s) \|_{L^\infty} \sup_{0<s<t}  \| g(s) \|_{L^d_{uloc}} \\
&\quad+\sup_{0<s<t} s^\frac12 \| f(s) \|_{L^\infty} \sup_{0<s<t}s^\frac12 \| \nabla g(s) \|_{L^d_{uloc}}\\
& \quad + \sup_{0<s<t} s^\frac12 \| \nabla f(s)\|_{L^d_{uloc}} \sup_{0<s<t} s^\frac12 \| g(s) \|_{L^\infty} \bigg ).
\end{split}
\end{align}
Thus we obtain 
\begin{align*}
\| B[f,g] \|_T\leq C_1 (1+ T^\frac12) \| f\|_T \| g\|_T, \qquad f,g \in X_T.
\end{align*}
The continuity in time for $t\in (0,T)$ also follows from that of $f,g$ in $(0,T)$, and we skip the details here.
If $\epsilon$ is small so that 
\begin{align}\label{proof.thm.mild.ns.9}
C_1 ( 1+ T^\frac12) 2 C_0 (1 + T^\frac12) \epsilon \leq \frac14,
\end{align}
then \eqref{proof.thm.mild.ns.9} and the definition of $C_0$ imply that $\Phi$ defines a contraction map from $X_T$ into $X_T$. Hence, there exists a unique fixed point $u$ of $\Phi$ in $X_T$, which is the unique mild solution to \eqref{eq.mild.ns} in $X_T$. If $u_0 \in \overline{D({\bf A})}^{L^d_{uloc}}$ then we just modify the set $X_T$ as
\begin{align*}
\tilde X_T & =\big \{ f\in C ([0,T); L^d_{uloc,\sigma} (\R^d_+)) \cap C((0,T); W^{1,d}_{uloc,0}(\R^d_+; \R^d) \cap BUC_\sigma (\R^d_+)) ~|~ \\
& \qquad \| f\|_T \leq  2 C_0 (1+T^\frac12) \| u_0 \|_{L^d_{uloc}}. \quad \lim_{t\rightarrow 0} t^\frac12 \| f(t) \|_{L^\infty} =0 \big \}.
\end{align*}
Then the estimates \eqref{proof.thm.mild.ns.5}-\eqref{proof.thm.mild.ns.7} yield 
\begin{align*}
\lim_{t\rightarrow 0} t^\frac12 \| B[f,g](t) \|_{L^\infty}=\lim_{t\rightarrow 0} \| B[f,g] (t) \|_{L^d_{uloc}} = 0
\end{align*}
when $f,g\in \tilde X_T$, and we can construct the unique mild solution in $\tilde X_T$. The details are omitted here. The proof is complete. 
\end{proof}

For the next result, we specialize to $d=3$.

\begin{proposition}\label{prop.mild.ns2} 
For any $u_0\in \mathcal L^3_{uloc,\sigma}(\R^3_+)$
 there exist $T>0$ and a unique mild solution $u\in C\left([0,T); \mathcal L^3_{uloc}(\R^3_+)\right) \cap C\big((0,T); W^{1,5}_{uloc,0}(\R^3_+) \cap BUC_\sigma (\R^3_+)\big)$ such that 
\begin{align*}
\sup_{0<t<T} \big ( \| u (t) \|_{L^3_{uloc}} + 
t^\frac15 \| u(t) \|_{L^5} + t^\frac{7}{10} \| \nabla u(t) \|_{L^5_{uloc}} \big ) \leq C_* \| u_0 \|_{L^3_{uloc}}\,.
\end{align*}
\end{proposition}

\begin{proof} The proof is  based on the argument by Kato \cite{Kato84}. Set $\|f\|_T$ as
\begin{align*}
\| f\|_T := \sup_{0<t<T} \big ( t^\frac{1}{5} \| f(t) \|_{L^5_{uloc}} + t^\frac{7}{10} \| \nabla f(t) \|_{L^5_{uloc}} \big ).
\end{align*}
For any $\varepsilon>0$ there exists $\tilde{u}_0 \in C^\infty_{0,\sigma}(\R^3_+)$ such that 
$\|u_0 -\tilde{u}_0  \|_{L^3_{uloc}}<\varepsilon$.
Therefore 
\begin{align*}
t^{\frac15}\| e^{-t{\bf A}}u_0 \|_{L^5_{uloc}} 
&\le  
t^\frac15 (\|e^{-t{\bf A}}(u_0- \tilde{u}_0)\|_{L^5_{uloc}}
+\|e^{-t{\bf A}}\tilde{u}_0\|_{L^5_{uloc}}) 
\\
&\le  C(1+t^\frac15) \| u_0- \tilde{u}_0\|_{L^3_{uloc}}
+t^\frac15 \|\tilde{u}_0\|_{L^5_{uloc}}.
\end{align*}
Similarly 
\begin{align*}
t^{\frac{7}{10}}\| \nabla e^{-t{\bf A}}u_0 \|_{L^5_{uloc}} 
&\le  C(1+t^\frac15) \| u_0- \tilde{u}_0\|_{L^3_{uloc}}
+t^\frac15 \|\tilde{u}_0\|_{L^5_{uloc}}.
\end{align*}
Therefore there exist $C_0$ and  $T_0>0$ such that for $T<T_0$
\begin{align*}
\| e^{-\cdot {\bf A}} u_0 \|_T 
\leq C_0\varepsilon.
\end{align*}
Let us introduce the set 
\begin{align*}
\begin{split}
Y_T & =\Big \{ f\in L^\infty (0,T; L^3_{uloc,\sigma} (\R^3_+)) \cap C\big((0,T); W^{1,3}_{uloc,0}(\R^d_+; \R^3) \cap BUC_\sigma (\R^3_+)\big) ~|~ \\
& \qquad \| f\|_T \leq  2C_0 \varepsilon \Big \}.
\end{split}
\end{align*}
The similar argument as in the proof of Proposition \ref{prop.mild.ns} shows 
\begin{align}
\label{est:B[f,g]}
\| B[f,g] \|_T\leq C_1 (1+T^{\frac15}) \| f\|_T 
\| g\|_T, \qquad f,g \in X_T.
\end{align}
Therefore if $T$ and $\varepsilon$ are small so that 
\begin{align*}
C_1 (1 + T^\frac15) 2 C_0 \varepsilon \leq \frac12,
\end{align*}
$\Phi$ defines a contraction map from $Y_T$ into itself. 
Hence, there exists a unique fixed point $u$ of $\Phi$ in $Y_T$, which is the unique mild solution to \eqref{eq.ns.1} in $Y_T$. 
We also easily see $u \in L^\infty(0,T; L^3_{uloc})$ as follows: 
\begin{align*}
\|u(t)\|_{L^3_{uloc}} 
&\le 
\|u_0\|_{L^3_{uloc}} + \|B[u,u](t)\|_{L^3_{uloc}}
\\
&\le 
\|u_0\|_{L^3_{uloc}} + C(1+T^{\frac{1}{10}}) \|u\|^2_T.
\end{align*}
We will show that $u$ belongs to  $C\big([0,T); \mathcal L^3_{uloc,\sigma}(\R^3_+)\big)$.
It is enough to show $u$ is strongly continuous on $[0,T)$ in $L^3_{uloc}$.
Since the continuity away from $t=0$ can be shown as stated 
in the proof of Proposition \ref{prop.mild.ns}, we focus on the case when $t=0$.
We have 
\begin{align*}
\|u(t)-u_0\|_{L^3_{uloc}}
&\le \|e^{-t\mathbf A} u_0 -u_0\|_{L^3_{uloc}} +C\|B[u,u](t)\|_{L^3_{uloc}}
\\
&\le \|e^{-t\mathbf A} u_0 -u_0\|_{L^3_{uloc}} +C\|u\|^2_{T}.
\end{align*}
The standard density argument yields that the first term converges to $0$ as $t\rightarrow 0$, while
in the second term, \eqref{est:B[f,g]} 
implies $\lim_{T  \rightarrow 0} \|u \|_T=0$.  Thus 
the proof is complete. 
\end{proof}

\subsection{Concentration of $L^q$ norm, $q\geq d$, near the blow up time}
\label{sec.concentration}
This subsection is devoted to the proof of Corollary 
\ref{cor.concentration}.
We introduce the space 
$L^q_{uloc,(\rho)}(\R^d_+)$ for $\rho>0$ which is defined as follows:
\begin{equation*}
L^q_{uloc,(\rho)}(\R^d_+):=
\left\{f\in L^1_{loc}(\R^d_+)
~|~
\sup_{\eta\in\mathbb Z^{d-1}\times\mathbb Z_{\geq 0}}
\|f\|_{L^q(\rho\eta+(0,\rho)^d)}<\infty\right\}.
\end{equation*}
The following variant of proposition  \ref{prop.mild.ns} and \ref{thm.mild.ns} plays a crucial 
role in the proof.

\begin{proposition}\label{prop.mild.ns.rho} Let $q\geq d$. There exist constants $\gamma,\, C_*>0$ such that the following statement holds.
For any $u_0\in L^q_{uloc,(\rho),\sigma}(\R^d_+)$ satisfying 
$\|u_0 \|_{L^q_{uloc,(\rho)}}\leq \gamma\rho^{\frac dq-1}$ for some $\rho>0$, 
there exist  $T \ge \rho^2$ and a unique mild solution $u\in L^\infty (0,T; L^q_{uloc,(\rho),\sigma} (\R^d_+))$ such that 
\begin{align}
\label{est.thm.mild.ns.rho}
\sup_{0<t<T} \big ( \| u (t) \|_{L^q_{uloc,(\rho)}} + t^\frac d{2q} \| u(t) \|_{L^\infty} 
 \big ) \leq C_* \| u_0 \|_{L^q_{uloc,(\rho)}}\,.
\end{align}
\end{proposition}
This can be proved by a simple rescaling argument from Proposition \ref{prop.mild.ns} and \ref{thm.mild.ns},
and so we omit the details here. This results enables to control the existence time in terms of the smallness of the initial data in $L^q_{uloc,(\rho)}$.

\begin{proof}[Proof of Corollary \ref{cor.concentration}]
Let $q\geq d$ and $\gamma=\gamma(q)>0$ be given by Proposition \ref{prop.mild.ns.rho}. We define $\rho_*=\rho_*(t)$ for $t\in (0,T_*)$ by
$$
\rho_*(t)=\inf \left\{\rho>0 \ | \ \|u(t)\|_{L^q_{uloc,(\rho)}}\rho^{1-\frac dq} > \gamma \right \}   .
$$
Note that $\rho_*$ is finite since $u\not\equiv 0$, and $\rho_*>0$ since $u(t)$ is bounded for all $t\in(0,T_*)$. 
For $t \in (0,T_*)$ fixed (our new initial time) let $\rho>0$ be a constant such that $ \rho <\rho_*(t)$. It suffices to show that $\rho \le  \sqrt{T_*-t}$. Since (by the definition of $\rho$) $\|u(t)\|_{L^q_{uloc,(\rho)}}\leq\gamma\rho^{\frac dq-1}$, Proposition \ref{prop.mild.ns.rho} shows 
existence of the solution $v$  in $[t,t+T]$ such that at initial time 
$v(t)=u(t)$ and $T\ge \rho^2$. 
Assume for a moment that $u$ agrees with 
$v$ in $[t,T']$ for $T'=\min (t+T,T_*-\varepsilon)$ for 
small $\varepsilon>0$. 
Then by the definition of $T_*$ we must have $t+T<T_*-\varepsilon$.
Since $\varepsilon$ is arbitrary, this yields the desired estimate 
for $\rho$.
As for the uniqueness, we see from the continuity of $u$ on  $(0, T_*)$ 
that there exists a constant $\delta>0$ such that  
\begin{align*}
\sup_{t<s<t+\delta} \big ( \| u (s) \|_{L^q_{uloc,(\rho)}} + 
s^\frac d{2q} \| u(s) \|_{L^\infty} \big ) 
\leq C_* \| u(t) \|_{L^q_{uloc,(\rho)}}.
\end{align*}
Hence the uniqueness in $[t,t+\delta]$ follows from 
Proposition  \ref{prop.mild.ns.rho}. 
To show the uniqueness up to time $T'$, 
notice that $u$ and  $v$ are both bounded in 
$\R^d_+ \times  [t+\delta,T']$. Then the bilinear estimate shows 
that  the difference $w:=u-v$ satisfies  
\begin{align*}
\|w\|_{L^\infty(t_1,t_2; L^\infty)} 
&\le C(t_2-t_1)^{\frac12} 
\left(\|u\|_{L^\infty(t_1,t_2;L^\infty)}
+\|v\|_{L^\infty(t_1,t_2;L^\infty)}\right) 
\|w\|_{L^\infty(t_1, t_2;L^\infty)}\\
&\le C(t_2-t_1)^{\frac12} 
\left(\|u\|_{L^\infty(t+\delta,T';L^\infty)}
+\|v\|_{L^\infty(t+\delta,T';L^\infty)}\right) 
\|w\|_{L^\infty(t_1, t_2;L^\infty)}
\end{align*}
for $t+\delta \le t_1<t_2\le T'$.
Thus taking $t_2-t_1$ sufficiently small shows $w=0$ on $[t_1,t_2]$. 
Using this argument a finite number of times, we have the uniqueness in $[t,T']$. 
\end{proof}

\appendix

\section{A Liouville theorem for the resolvent problem in $L^1_{uloc}$ spaces}
\label{sec.liouville}

This appendix is devoted to the proof of the Liouville theorem, Theorem \ref{thm.unique} for the Stokes resolvent problem.

\begin{proof}[Proof of Theorem \ref{thm.unique}] 
(i) First we introduce the regularization of $(u,p)$ in $x'$ as follows:
\begin{align*}
u^\kappa (x',x_d) & = \int_{\R^{d-1}} \kappa^{-(d-1)} \eta (\frac{x'-y'}{\kappa}) u (y',x_d) \, d y',\\
p^\kappa (x',x_d) & = \int_{\R^{d-1}} \kappa^{-(d-1)} \eta (\frac{x'-y'}{\kappa}) p (y',x_d) \, d y'.
\end{align*}
Here $\kappa\in (0,1)$ and $\eta\in C_0^\infty (\R^{d-1})$ satisfies $\int_{\R^{d-1}} \eta (x') d x' =1$ and ${\rm supp}\, \eta \subset B_1 (0')$. Then, by the symmetry of $\R^d_+$, $(u^\kappa,p^\kappa )$ is also a solution to \eqref{e.resol} with $f=0$ in the sense of distributions. Moreover,  $(u^\kappa (\cdot, x_d) , \nabla p^\kappa (\cdot, x_d))$ is smooth in $x'$ satisfies the estimate 
\begin{align*}
\| \nabla'^\alpha u^\kappa \|_{L^\infty_{x'} (\R^{d-1}; L^1_{uloc} (\R_+))}  & \leq C \kappa^{-d+1-|\alpha|} \| u\|_{L^1_{uloc}},\\
\| \nabla'^\alpha \nabla p^\kappa \|_{L^\infty_{x'} (\R^{d-1}; L^1_{uloc} (\R_+))}  & \leq C \kappa^{-d+1-|\alpha|} \| \nabla p\|_{L^1_{uloc}}
\end{align*}
for any multi-index $\alpha$. We can also check that $p^\kappa  \in L^1_{loc} (\R^d_+)$ without difficulty. 
The divergence-free condition on $u^\kappa$ then implies 
\begin{align*}
\partial_{x_d} u^\kappa_d = - \nabla' \cdot (u^\kappa)' \in L^\infty (\R^{d-1}; L^1_{uloc}(\R_+))
\end{align*}
in the sense of distributions, which implies that  $u^\kappa_d$ is continuous up to the boundary.
Then, again from the divergence-free condition, $\int_{\R^d_+} u^\kappa \cdot \nabla \phi d x=0$ for all $\phi \in C_0^\infty (\overline{\R^d_+})$, we have $\displaystyle \lim_{x_d\rightarrow 0} u^\kappa_d (x',x_d) =0$ for all $x'\in \R^{d-1}$. 
Next we take arbitrary $g\in C_0^\infty (\R^d_+)^d$, and let $(v,\nabla q)$ be the smooth and decaying solution to \eqref{e.resol} with $f=\Delta' g$. By virtue of the presence of $\Delta'$ one can show that $(v, \nabla q)$ is constructed so that  
\begin{align}\label{proof.thm.unique.1} 
\nabla' v (x)\,, \nabla^\alpha \Delta' v (x) = \mathcal{O}(|x|^{-d-\frac12}), \quad \nabla^{\tilde \alpha}  \Delta' q (x) = \mathcal{O}(|x|^{-d-\frac12}), \qquad |x|\gg 1
\end{align}
for any multi-indexes $\alpha$ and $\tilde \alpha$ with $|\alpha|\leq 2$ and $|\tilde \alpha|\leq 1$; see Proposition \ref{prop.smooth.appendix} below.
Then,  by Remark \ref{thm.unique} we see
\begin{equation}\label{eq.appendix.ipp}
\begin{aligned}
\int_{\R^d_+} u^\kappa \cdot (\Delta')^2 g \, d x & = \int_{\R^d_+} u^\kappa \cdot \Delta' \big ( \lambda  v - \Delta  v + \nabla  q \big ) \, d x \\
& = - \int_{\R^d_+} \nabla p^\kappa \cdot \Delta'  v \, d x +   \int_{\R^d_+} u^\kappa \cdot \nabla \Delta'  q \, d x\\
& =  \sum_{j=1}^{d-1} \int_{\R^d_+}  \nabla \partial_j p^\kappa \cdot  \partial_j  v \, d x\\
& = - \sum_{j=1}^{d-1} \int_{\R^d_+} \partial_j p^\kappa \partial_j {\rm div}\,  v \, d x=0
\end{aligned}
\end{equation}
from the definition of the solution in the sense of distributions.
Note that the above integration by parts is justified from \eqref{proof.thm.unique.1} and from the fact that $\nabla p^\kappa$ and $\partial_j \nabla p^\kappa$ with $j=1,\ldots, d-1$ belong to $L^\infty_{x'}(\R^{d-1}; L^1_{uloc}(\R_+))$.
Since $g\in C_0^\infty (\R^d_+)^d$ is arbitrary, this identity implies that 
\begin{align*}
(\Delta')^2 u^\kappa = 0 \qquad a.e.~x\in \R^d_+.
\end{align*}
Set $U^\kappa (x',x_d) = \int_0^{x_d} u^\kappa (x',y_d) d y_d$.  
Recall that ${\nabla'}^\alpha u^\kappa (\cdot, x_d)\in L^\infty_{x'} (\R^{d-1}; L^1_{uloc} (\R_+))$,
which shows that for each fixed $x_d\geq 0$, $U^\kappa (x',x_d)$ is smooth and bounded including its derivatives in $x'$, while it is absolutely continuous in $x_d$ for each fixed $x'$. Moreover, for each $x_d\geq 0$, $U^\kappa (\cdot, x_d)$ satisfies $(\Delta')^2 U^\kappa (\cdot, x_d)=0$.
By the Liouville theorem of the bi-Laplace equation in $\R^{d-1}$, we conclude that $U^\kappa (\cdot, x_d)$ is constant in $x'$ for each $x_d\geq 0$, that is, $U^\kappa (x',x_d) = A(x_d)$. Since the left-hand side is absolutely continuous in $x_d$, so is $A$, and we have 
\begin{align}\label{proof.thm.unique.2} 
u^\kappa (x',x_d) = \partial_{x_d} U^\kappa (x',x_d) = \partial_{x_d} A(x_d) =: a^\kappa (x_d)\in L^1_{uloc} (\R_+)^d, \qquad a^\kappa_d\in C(\overline{\R_+}).
\end{align}
Then, the divergence-free condition implies that $\partial_{x_d} a^\kappa_d =0$, and thus, together with the fact $\displaystyle \lim_{x_d\rightarrow 0} u^\kappa_d (x',x_d)=0$, we have $a^\kappa_d =0$.
Next we take $\phi \in C_0^\infty (\R^d_+)$ and set $\varphi =(0,\cdots,0,\phi)^\top \in C_0^\infty (\R^d_+)^d$ in \eqref{def.weak.u1}, which yields from $u^\kappa_d=0$,
\begin{align*}
\int_{\R^d_+} \partial_{x_d} p^\kappa  \phi \, d x=0.
\end{align*}
Thus, $p^\kappa$ does not depend on $x_d$. On the other hand, by taking $\varphi = \nabla \phi$ with $\phi\in C_0^\infty (\R^d_+)$ in \eqref{def.weak.u1}, it follows that 
\begin{align*}
\int_{\R^d_+} \nabla p^\kappa \cdot \nabla \phi \, d x=0, \qquad \phi \in C_0^\infty (\R^d_+).
\end{align*}
Hence, $p^\kappa$ is harmonic in $\R^d_+$, and moreover, since $p^\kappa$ is independent of $x_d$, we have $\Delta' p^\kappa  (x')=0$ for all $x'\in \R^{d-1}$. The Liouville theorem implies that $p^\kappa$ is a harmonic polynomial. Then, going back to \eqref{def.weak.u1} and using \eqref{proof.thm.unique.2}, we have for each $j=1,\ldots,d-1$,
\begin{align}\label{proof.thm.unique.3'} 
\begin{split}
& \int_{\R_+} a_j ^\kappa (x_d) (\lambda \phi - \partial_{x_d}^2 \phi ) (x_d)  d x_d \int_{\R^{d-1}} \psi (x') \, d x' \\
&   = - \int_{\R_+} \phi (x_d) d x_d  \int_{\R^{d-1}} \partial_j p^\kappa (x') \psi (x') d x\\
&  \qquad \qquad {\rm for~all}~~\psi\in C_0^\infty (\R^{d-1}), \quad \phi \in C_0^\infty (\R_+).\end{split}
\end{align}
We first fix $\phi$ such that $\int_{\R_+}\phi \, d x_d\ne 0$. Since $\psi\in C_0^\infty (\R^{d-1})$ is arbitrary, $\partial_j p^\kappa (x')$ is constant for all $j=1,\ldots,n-1$. Hence, $p^\kappa (x')$ is polynomial at most first order about $x'$. Thus, $(u^\kappa,  p^\kappa)$ is a parasitic solution. Since $(u^\kappa, p^\kappa)$ converges to $(u,p)$ in $L^1_{loc}(\R^d_+)$, the limit $(u,p)$ must be a parasitic solution. Note that the limit $(u,p)$ with $u=(a'(x_d),0)^\top$ and $p=D \cdot x' + c$ satisfies the reduced version of \eqref{proof.thm.unique.3'}:
\begin{align}\label{proof.thm.unique.3''} 
\begin{split}
\int_{\R_+} a_j  (x_d) (\lambda \phi - \partial_{x_d}^2 \phi ) (x_d)  d x_d & = -D_j \int_\R \phi  \, d x_d \\
&  \qquad {\rm for~all}~~\phi \in C_0^\infty (\overline{\R_+}) ~~{\rm with}~\phi|_{x_d=0} =0.
\end{split}
\end{align} 
In particular, each $a_j$ is smooth and bounded and has a zero boundary trace, for $\lambda$ which belongs to the resolvent set of the Dirichlet Laplacian in $L^1 (\R_+)$. 
Moreover,  if 
\begin{equation*}
\displaystyle \lim_{R\rightarrow \infty} \| \nabla' p \|_{L^1(|x'|<1, R<x_d<R+1)}=0
\end{equation*}
then the vector $D$ must be $0$, and thus the pressure $p$ is a constant. Then \eqref{proof.thm.unique.3''} is reduced to  
\begin{align}\label{proof.thm.unique.3} 
\int_{\R_+} a_j  (x_d) (\lambda \phi - \partial_{x_d}^2 \phi ) (x_d)  d x_d =0 \,  \qquad {\rm for~all}~~\phi \in C_0^\infty (\overline{\R_+}) ~~{\rm with}~\phi|_{x_d=0} =0.
\end{align} 
The uniqueness of this very weak solution is standard and also follows from the fact that $\lambda$ belongs to the resolvent set of the Dirichlet Laplacian in $L^1(\R_+)$. Thus we have arrived at $a_j=0$ for each $j=1,\ldots,d-1$, that is $u =0$. 
On the other hand, if  $\displaystyle \lim_{|y'| \rightarrow \infty} \| u  \|_{L^1 (|x'-y'|<1, 1<x_d<2)}=0$ then $u=0$ in $1<x_d<2$ since $u=(a'(x_d),0)^\top$ is independent of $x'$, which also gives $D=0$ by \eqref{proof.thm.unique.3''}. 
Thus $p$ is a constant.  Hence $a_j$ satisfies \eqref{proof.thm.unique.3} also in this case, which gives $u=0$ for $x_d>0$. The proof of (i) is complete.

\noindent (ii) The proof is similar to the one of (i). Again it suffices to consider the mollified solution $(u^\kappa,p^\kappa)$ as in (i). Fix arbitrary $\mu\in (0,1)$ and let $(v_\mu,\nabla q_\mu)$ be the smooth and decaying solution to \eqref{e.resol} with $\lambda=\mu$ and $f=\Delta' g$, where $g\in C_0^\infty (\R^d_+)^d$ is arbitrary taken.
Then \eqref{eq.appendix.ipp} is replaced by 
\begin{align}\label{eq.appendix.ipp'}
\int_{\R^d_+} u^\kappa \cdot (\Delta')^2 g d x = \mu \int_{\R^d_+} u^\kappa \cdot \Delta' v_\mu  d x\,.
\end{align}
We observe from Proposition \ref{prop.smooth.appendix} that $|\Delta' v_\mu (x)|\leq C|\mu|^{-\frac34} (1+|x|)^{-d-\frac12}$ with $C$ independent of $\mu\in (0,1)$, and thus, we can take the limit $\mu \downarrow 0$, which leads to $(\Delta')^2 u^\kappa =0$.
Then the same argument as in (i) shows that $u^\kappa=(a'(x_d), 0)^\top$ and $p^\kappa=p^\kappa(x')$ satisfy \eqref{proof.thm.unique.3'} with $\lambda=0$, which implies that $p^\kappa$ is a first order polynomial, and we have \eqref{proof.thm.unique.3''} with $\lambda=0$. Then each $a_j (x_d)$ is smooth and satisfies $\partial_{x_d}^2 a_j = D_j$ with the Dirichlet boundary condition $a_j(0)=0$. Since $u\in L^1_{uloc}(\R^d_+)^d$ such an $a_j$ must be zero, and thus, $u=0$ and then we also see from the equation that $p$ is a constant.  The proof of (ii) is complete.
\end{proof}

\begin{proposition}\label{prop.smooth.appendix} Let $\lambda\in S_{\pi-\ep}$ with $\ep \in (0,\pi)$. Let $g\in C_0^\infty (\R^d_+)^d$. Then there exists a unique solution $(u, \nabla p)\in \big ( W^{2,2}(\R^d_+)^d \cap W^{1,2}_0 (\R^d_+)^d \cap L^2_\sigma (\R^d_+) \big ) \times L^2(\R^d_+)^d$ to \eqref{e.resol} with $f=\Delta' g$ such that $u$ and $p$ are smooth and satisfy
\begin{align}\label{est.prop.smooth.appendix} 
\begin{split}
|{\nabla'}^\alpha u(x)| + |{\nabla'}^\alpha \nabla u (x)| & \leq C_\ep (\frac{1}{|\lambda|^\frac14} + \frac{1}{|\lambda|^\frac34}) \frac{1}{(1+|x|)^{d+\frac12}},\\
|{\nabla'}^\alpha \nabla^2 u(x) | + | \nabla^{\tilde \alpha} {\nabla'}^2 p (x) | & \leq \frac{C_{\epsilon,\lambda}}{( 1+ |x|)^{d+\frac12}},  
\end{split}
\end{align}
for any multi-indexes $\alpha$ and $\tilde \alpha$ with $|\alpha|\leq 2$ and $|\tilde \alpha|\leq 1$.
Here the constant $C_\ep$ is taken uniformly in $\lambda\in S_{\pi-\ep}$, while $C_{\epsilon,\lambda}$ depends on $\ep$ and $\lambda\in S_{\pi-\ep}$.
\end{proposition}

Notice that  the uniform estimates in $|\lambda|$ is used  in the proof of (ii) of Theorem \ref{thm.unique}.

\begin{proof} The uniqueness is well known and we focus on the estimate \eqref{est.prop.smooth.appendix}.
The Helmholtz decomposition implies $g=h + \nabla p_g$, where $h\in L^2_\sigma (\R^d_+)$ and $\nabla p_g \in L^2(\R^d_+)^d$ with $p_g\in L^2_{loc}(\R^d_+)$. Since $g\in C_0^\infty (\R^d_+)^d$, $\nabla p_g$ and $h=g-\nabla p_g$ are also smooth and bounded in $\R^d_+$ including their derivatives. 
Then the pressure $\nabla p$ is constructed in the form $\nabla p = \Delta' \nabla p_g + \nabla p_u$, where $(u,\nabla p_u)$ is the unique solution to \eqref{e.resol} with $f=\Delta' h$.
We first show that 
\begin{align}\label{proof.prop.smooth.appendix.1}
\nabla^\beta \nabla' \nabla p_g (x), ~\nabla^\beta \nabla' h(x) = \mathcal{O}(|x|^{-d-1}),\qquad |x|\gg 1
\end{align}
for any multi-index $\beta$.
The estimate of $h$ follows from the one of $\nabla p_g$ by the relation $h=g-\nabla p_g$.
To show \eqref{proof.prop.smooth.appendix.1} for $\nabla p_g$, we recall that $p_g$ is constructed as the solution to the Neumann problem 
$\Delta p_g = {\rm div}\, g$ in $\R^d_+$ and $\partial_{x_d} p_g =g_d =0$ on $\partial\R^d_+$, which is given by the formula
\begin{align*}
p_g (x) = - \int_{\R^d_+} \big ( E(x-y) + E(x-y^*) \big ) {\rm div}\, g (y) \, d y,
\end{align*}
where $y^*=(y',-y_d)$, and $E(x)$ is the Newton potential in $\R^d$.
Then, the integration by parts and the condition $g\in C_0^\infty (\R^d_+)^d$ yield
\begin{align*}
p_g (x) & = - \nabla' \cdot  \int_{\R^d_+} \big ( E(x-y) + E(x-y^*) \big )  g' (y) \, d y\\
& \quad  - \partial_{x_d} \int_{\R^d_+} \big ( E(x-y) - E(x-y^*) \big )  g_d (y) \, d y.
\end{align*}
Hence, we obtain the formula
\begin{align}\label{proof.prop.smooth.appendix.3}
\begin{split}
\nabla' \nabla p_g (x) & = - \nabla' \nabla \nabla' \cdot  \int_{\R^d_+} \big ( E(x-y) + E(x-y^*) \big )  g' (y) \, d y \\
& \quad -\nabla' \nabla \partial_{x_d} \int_{\R^d_+} \big ( E(x-y) - E(x-y^*) \big )  g_d (y) \, d y.
\end{split}
\end{align}
Since $|\nabla^\beta \nabla^3 E(x)| \leq C |x|^{-d-1-|\beta|}$, we verify the estimate \eqref{proof.prop.smooth.appendix.1} for $|x|\gg 1$ when $g\in C_0^\infty (\R^d_+)^d$.
Next we consider the estimate of $(u,\nabla p_u)$. We can now apply the results of Section \ref{sec.resolhp}, in particular the integral representation formulas and the kernel estimates given there. 
That is, $u$ is written as $u(x) = \int_{\R^d_+} K_\lambda (x'-y', x_d,y_d) \Delta' h (y) d y$ with the kernel $K_\lambda$ whose pointwise estimates are given in Proposition \ref{l.kernel} for $k_{i,\lambda}$ with $i=1,2$ and in Proposition \ref{prop.estkernel} for $r_\lambda$.
Since $\partial_j$ with $j=1,\ldots, d-1$ commutes with the Stokes operator in $\R^d_+$, we verify the bound 
\begin{align}\label{proof.prop.smooth.appendix.4}
| u(x)| \leq \frac{C}{|\lambda|^\frac14} \int_{\R^d_+} \frac{1}{|x-y|^{d-1}(1+|\lambda|^\frac12 |x-y|)}  |{\nabla'} h(y)| d y
\end{align}
with $C$ independent of $\lambda\in S_{\pi-\ep}$.
Indeed, in virtue of \eqref{r.bound.2'}, the kernel $|\nabla' r_\lambda  (y,z_d)|$ is bounded from above by 
$$\frac{C}{|\lambda|^\frac14(y_d+z_d + |y'|)^{d-\frac12} (1+|\lambda|^\frac12 (y_d + z_d + |y'|)) (1+|\lambda|^\frac12 (y_d+z_d))^\frac12}.$$
A similar bound is valid also for the kernel $\nabla' k_{i,\lambda}$, $i=1,2$, by Proposition \ref{l.kernel}.
Estimate \eqref{proof.prop.smooth.appendix.4} implies from \eqref{proof.prop.smooth.appendix.1} that 
\begin{align}\label{proof.prop.smooth.appendix.2}
| u(x) | \leq C (\frac{1}{|\lambda|^\frac14} + \frac{1}{|\lambda|^\frac34} ) \frac{1}{(1+|x|)^{d+\frac12}}.
\end{align} 
Since the tangential derivatives commute with the kernel, we have the same estimate for $|{\nabla'}^\alpha u(x)|$ as in \eqref{proof.prop.smooth.appendix.2}. The estimate for the $x_d$ derivative is a bit more complicated.
We decompose the kernel $K_\lambda$ as $K_\lambda (y',y_d,z_d) =\chi(y') K_\lambda (y',y_d,z_d)   + \big (1-\chi (y') \big ) K_\lambda (y',y_d,z_d)$, where $\chi(y')$ is a smooth cut-off such that $\chi (y')=1$ for $|y'|<1$ and $\chi(y')=2$ for $|y'|\geq 2$. 
Then we compute as
\begin{align*}
\partial_{x_d} u (x) & = \int_{\R^d_+}   \chi \partial_{x_d} K_\lambda (x'-y',x_d,y_d) \Delta' h (y) d y \\
& \quad + \int_{\R^d_+}   \nabla_x' \big ( (1-\chi) \partial_{x_d} K_\lambda (x'-y',x_d,y_d) \big ) \cdot \nabla' h (y) d y ,
\end{align*}
and then, the former term is estimated from above by 
\begin{align*}
\frac{C}{|\lambda|^\frac14} \int_{|x'-y'|<1} \frac{1}{|x-y|^{d-\frac12} \big ( 1 + |\lambda|^\frac12 |x_d-y_d|\big )^\frac32 } |\Delta' h(y)| d y,
\end{align*}
and the latter is bounded by 
\begin{align*}
C \int_{|x'-y'|\geq 1} \frac{1}{(1+|x-y|)^d \big ( 1+ |\lambda|^\frac12 (x_d+y_d + |x'-y'|)\big )} |\nabla' h(y)| d y
\end{align*}
These bounds follow again from Propositions \ref{l.kernel} and \ref{prop.estkernel}.
Then it is straightforward to see 
\begin{align}\label{proof.prop.smooth.appendix.3}
|\nabla u(x)| \leq C (\frac{1}{|\lambda|^\frac14} + \frac{1}{|\lambda|^\frac34} ) \frac{1}{(1+|x|)^{d+\frac12}}
\end{align}
with $C$ depending on $h$. The same bound holds also for $|{\nabla'}^\alpha \nabla u(x)|$. 
Since 
\begin{align}\label{proof.prop.smooth.appendix.3'}
\partial_{x_d} p_u = -\lambda u_d + \Delta u_d+ \Delta' h_d  = -\lambda u_d + \Delta' u_d - \partial_{x_d} \nabla' \cdot u' + \Delta' h_d,
\end{align} we obtain $|\partial_{x_d}  p_u (x) | \leq C_{\epsilon,\lambda} (1+|x|)^{-d-\frac12}$ from the above results. Thus, the similar bound is valid also for $|{\nabla'}^\alpha \partial_{x_d} p_u(x)|$ since $\nabla'$ commutes with the kernel.
Next we consider the estimate of $\nabla' p$.
We apply the argument as in the estimate of $\nabla u$, that is, with the cut-off $\chi$, we write $\nabla' p_u$ as 
\begin{align*}
\nabla' p_u  (x) & = \int_{\R^d_+} \big ( \chi q_\lambda (x'-y',x_d,y_d) \big ) \nabla_y' \Delta' h (y) d y\\
& \quad  + \int_{\R^d_+} {\nabla_x'}^2 \big ( (1-\chi)  q_\lambda (x'-y',x_d,y_d) \big ) \cdot \nabla' h (y) d y.l
\end{align*}
Then Proposition \ref{prop.kernpressure} yields 
\begin{align*}
|\nabla' p_u (x) | & \leq C \int_{|x'-y'|<1} \frac{e^{-c|\lambda|^\frac12 y_d}}{(|x'-y'|+ x_d + y_d )^{d-1}} |\Delta' h (y) | d y \\
& \quad + C \int_{|x'-y'|\geq 1} \frac{e^{-c|\lambda|^\frac12 y_d}}{(1+|x'-y'| + x_d + y_d)^{d-1} (1+|x'-y'|)^2} |\nabla' h(y)| d y.
\end{align*}
From this bound it is not difficult to derive the estimate 
\begin{align}\label{proof.prop.smooth.appendix.4}
|\nabla' p_u (x)|\leq C (\frac{1}{|\lambda|^\frac14} + \frac{1}{|\lambda|^\frac34} ) \frac{1}{(1+|x|)^{d+\frac12}}.
\end{align}
The details are omitted. Then we also obtain the same estimate for $|{\nabla'}^\alpha \nabla' p (x)|$.
It remains to estimate $\partial_{x_d}^2 u$, but from the divergence-free condition we have already obtained the estimate for $\partial_{x_d}^2 u_d$, and thus, it suffices to consider $\partial_{x_d}^2 u'$. 
But the decay estimate immediately follows from the equation $\partial_{x_d}^2 u' = \lambda u' - \Delta' u' + \nabla' p_u - \Delta' h'$. The proof is complete.
\end{proof}

\begin{remark}\label{rem.prop.smooth.appendix}{\rm Let $(u,\nabla p_u)$ be the solution to the resolvent problem \eqref{e.resol} with $f=\Delta' h$ as in the proof of Proposition \ref{prop.smooth.appendix}. 
From \eqref{proof.prop.smooth.appendix.4} we have shown that
\begin{align}
|{\nabla'}^\alpha \nabla' p_u (x)|\leq C (\frac{1}{|\lambda|^\frac14} + \frac{1}{|\lambda|^\frac34} ) \frac{1}{(1+|x|)^{d+\frac12}}
\end{align}
for any $\lambda\in S_{\pi-\epsilon}$ and $|\alpha|\leq 2$, where $C$ depends only on $\epsilon$ and $h$. 
On the other hand, from \eqref{proof.prop.smooth.appendix.3} and \eqref{est.prop.smooth.appendix}, we have 
\begin{align}
|{\nabla'}^\alpha \partial_{x_d} p_u (x)|\leq C ( |\lambda|^\frac34 + \frac{1}{|\lambda|^\frac14} + \frac{1}{|\lambda|^\frac34} ) \frac{1}{(1+|x|)^{d+\frac12}}
\end{align}
for any $\lambda\in S_{\pi-\epsilon}$ and $|\alpha|\leq 2$, where $C$ depends only on $\epsilon$ and $h$. 
}
\end{remark}

\section{A Liouville theorem for the nonstationary problem in $L^1_{uloc}$ spaces}
\label{sec.liouville.nonsteady}

The class of weak solutions for the non steady Stokes system \eqref{eq.s.appendix} is stated as follows.
Let $u_0\in L^1_{uloc,\sigma} (\R^d_+)$ and $f\in L^1_{loc} ((0,T) \times \overline{\R^d_+}))^d$.
We say that $(u,\nabla p)$ is a solution to \eqref{eq.s.appendix} in the sense of distributions if 

\noindent (i) $u\in L^\infty (0,T; L^1_{uloc,\sigma} (\R^d_+))$, $\nabla p \in L^1_{loc} ((0,T)\times \R^d_+)^d$ with $p\in L^1_{loc} ((0,T)\times \R^d_+)$, and 
\begin{align}\label{condition.p.B}
\sup_{x\in \R^d_+} \int_\delta^T \| \nabla p (t) \|_{L^1 (B(x)\cap \R^d_+)} d t <\infty \qquad \text{for any}~~0<\delta<T\,.
\end{align}
Here $B (x)$ is the ball of radius $1$ centered at $x$.

\noindent (ii) The map $t\mapsto \int_{\R^d_+} u(t,x) \cdot \varphi (x) d x$ belongs to $C([0,T))$ for any $\varphi \in C_0^\infty (\overline{\R^d_+})^d$. In particular, the initial condition is satisfied in this sense.

\noindent (iii) For all $t',t\in (0,T)$ with $t>t'$ and for all $\varphi \in C_0^\infty ((0,T)\times \overline{\R^d})^d$ with $\varphi|_{x_d=0} =0$,
\begin{align}\label{weak.sol.B}
\begin{split}
& \int_{\R^d_+} u(t,x) \cdot \varphi (t,x) d x - \int_{t'}^t \int_{\R^d_+} u(t,x) \cdot \big ( \partial_s \varphi + \Delta \varphi \big ) (t,x) - \nabla p (t,x)  \cdot \varphi (t,x) d x  d s \\
& \qquad = \int_{\R^d_+} u(t',x) \cdot \varphi (t',x) d x + \int_{t'}^t  f(t,x) \cdot \varphi (t,x) d x d s\,.
\end{split}
\end{align}

\begin{proof}[Proof of Theorem \ref{thm.unique.time}] By considering the mollification $(u^\kappa, \nabla p^\kappa)$ instead of $(u,\nabla p)$ as in the proof of Theorem \ref{thm.unique}, we may assume in addition that $(u,\nabla p)$ is smooth in $x'$ and $\nabla u_d$ is bounded. We denote by $\langle \cdot, \cdot \rangle$ the usual inner product of $L^2 (\R^d_+)^d$. Take arbitrary $t,t'\in (0,T)$ with $t>t'$ and  $g\in C_0^\infty (\R^d_+)$. 
We introduce a mollification in time, and set $u^\rho=j_\rho * u$ and $p^\rho=j_\rho * p$, where $*$ here is the convolution in time and $j_\rho (\tau)\in C_0^\infty ((-\rho,\rho))$ is the mollifier with a small parameter $\rho>0$. The parameter $\rho>0$ is taken small enough so that $t'>2\rho$ and $t<T-2\rho$.
Then, we have the bound such as $\nabla p^\rho \in L^\infty (t',t; L^1_{uloc} (\R^d_+))$.
Note also that $(u^\rho, \nabla p^\rho)$ satisfies \eqref{weak.sol.B} for $t,t'$.
Fix arbitrary $\epsilon\in (0,1)$. Let $R\geq 1$ and $\chi_R$ be a smooth cut-off such that $\chi_R=1$ for $|x|\leq R$ and $\chi_R=0$ for $|x|\geq 2R$. Let $\mathbb{P}=I-\mathbb{Q}$ be the Helmholtz projection in $L^2(\R^d_+)$, where $\mathbb{Q}g = \nabla p_g$ is defined as in the proof of Proposition \ref{prop.smooth.appendix}. Note that ${\Delta'}^2 p_g$ and ${\Delta'}^2 \nabla p_g$ are smooth and decay fast enough so that $O(|x|^{-d-2})$ as $|x|\rightarrow \infty$.
Then one can verify that $\langle u^\rho (t), {\Delta'}^2 \mathbb{Q} g\rangle =0$, and thus,
$\langle u^\rho (t), \chi_R {\Delta'}^2\mathbb{Q} g \rangle = - \langle u^\rho (t) , (1-\chi_R) {\Delta'}^2 \mathbb{Q} g\rangle$. Hence, we can take $R=R_\epsilon$ large enough so that $|\langle u^\rho (t), \chi_{R_\epsilon} {\Delta'}^2 \mathbb{Q} g\rangle|\leq \epsilon$. We may also assume that ${\rm supp}\, g\subset \{|x|<R_\epsilon\}$.
Next, since $u^\rho (t) \chi_{R_\epsilon}\in L^1 (\R^d_+)$, there exists $u^{\rho, \epsilon}  (t)\in C_0^\infty (\R^d_+)^d$ such that $\| u^\rho (t)\chi_{R_\epsilon}- u^{\rho,\epsilon} (t) \|_{L^1 (\R^d_+)} \leq \epsilon$. We take $\tilde t>t$ which will be chosen later. 
Let $v$ be the velocity field defined by 
\begin{equation*}
v (s) = e^{-(\tilde t -s) \mathbf{A}} \Delta' \mathbb{P} g = \frac{1}{2\pi i} \int_\Gamma e^{(\tilde t -s) \lambda} (\lambda + \mathbf{A})^{-1} \Delta' \mathbb{P} g d\lambda
\end{equation*}
for $0\leq s<\tilde t$, where $\Gamma$ is the curve as in the proof of Proposition \ref{prop.L^p_u-L^q_u.semigroup}. 
Then $v$ satisfies $\partial_s v + \Delta v - \nabla q =0$ for $0\leq s < \tilde t$, where the associated pressure $\nabla q (s)$ is given by the formula $\nabla q (s) = \frac{1}{2 \pi i} \int_\Gamma e^{(\tilde t -s ) \lambda} \nabla q_\lambda d \lambda$. Here $\nabla  q_\lambda$ is the pressure for each resolvent problem. 
Note that for $\nabla q_\lambda$ we can apply the pointwise estimate stated in Remark \ref{rem.prop.smooth.appendix}, which gives the bounds 
\begin{align}\label{proof.thm.unique.time.1}
|{\nabla'}^\alpha \nabla' q (s,x)| & \leq \frac{C_{\tilde t-t'}}{(\tilde t -s)^\frac34 (1+|x|)^{d+\frac12}} ,\\
|{\nabla'}^\alpha \partial_{x_d} q (s,x) | & \leq \frac{C_{\tilde t-t'}}{(\tilde t-s)^\frac74 (1+|x|)^{d+\frac12}},\label{proof.thm.unique.time.2}
\end{align}
for $t'\leq s <\tilde t$ and $|\alpha|\leq 2$.
We see
\begin{align*}
\langle u^\rho (t), \Delta'^2 g \rangle & = \langle u^\rho (t), \chi_{R_\epsilon}  \Delta'^2 g \rangle \\
& = \langle u^\rho (t), \chi_{R_\epsilon} \Delta' v(t) \rangle\\
&   \quad - \langle u^\rho (t), \chi_{R_\epsilon} \big ( e^{-(\tilde t-t)\mathbf{A}} {\Delta'}^2 \mathbb{P} g - {\Delta'}^2 \mathbb{P} g \big ) \rangle + \langle u^\rho (t), \chi_{R_\epsilon} \Delta' \mathbb{Q} \Delta' g\rangle\,.
\end{align*} 
Then, from the identity 
\begin{align*}
\langle u^\rho (t), \chi_{R_\epsilon} \big ( e^{-(\tilde t-t)\mathbf{A}} {\Delta'}^2 \mathbb{P} g - {\Delta'}^2 \mathbb{P} g \big ) \rangle  & = \langle u^\rho (t) \chi_{R_\epsilon} - u^{\rho,\epsilon} (t) , \big ( e^{-(\tilde t-t)\mathbf{A}} {\Delta'}^2 \mathbb{P} g - {\Delta'}^2 \mathbb{P} g \big ) \rangle\\
& \quad + \langle  u^{\rho,\epsilon} (t) , \big ( e^{-(\tilde t-t)\mathbf{A}} {\Delta'}^2 \mathbb{P} g - {\Delta'}^2 \mathbb{P} g \big ) \rangle,
\end{align*}
we have 
\begin{align*}
|\langle u^\rho (t), \chi_{R_\epsilon} \big ( e^{-(\tilde t-t)\mathbf{A}} {\Delta'}^2 \mathbb{P} g - {\Delta'}^2 \mathbb{P} g \big ) \rangle | & \leq \epsilon \|  e^{-(\tilde t-t)\mathbf{A}} {\Delta'}^2 \mathbb{P} g - {\Delta'}^2 \mathbb{P} g \|_{L^\infty} \\
& \quad + \| u^{\rho,\epsilon} (t)\|_{L^2}  \| e^{-(\tilde t-t)\mathbf{A}} {\Delta'}^2 \mathbb{P} g - {\Delta'}^2 \mathbb{P} g \|_{L^2}\\
& \leq C \epsilon \| {\Delta'}^2 \mathbb{P} g \|_{L^\infty} \\
& \quad + \| u^{\rho,\epsilon} (t)\|_{L^2}  \| e^{-(\tilde t-t)\mathbf{A}} {\Delta'}^2 \mathbb{P} g - {\Delta'}^2 \mathbb{P} g \|_{L^2}.
\end{align*}
Here we have used the fact that the Stokes semigroup is a bounded semigroup in $L^\infty_\sigma (\R^d_+)$.
Note that $\|{\Delta'}^2 \mathbb{P} g\|_{L^\infty}$ is finite since $g\in C_0^\infty (\R^d_+)^d$, and 
there exists  $\tilde t_\epsilon >t$ such that $\| u^{\rho,\epsilon} (t)\|_{L^2}  \| e^{-(\tilde t-t)\mathbf{A}} {\Delta'}^2 \mathbb{P} g - {\Delta'}^2 \mathbb{P} g \|_{L^2} \leq \epsilon$ for any $\tilde t\in (t,\tilde t_\epsilon)$.
Hence we have 
\begin{align*}
|\langle u^\rho (t), \chi_{R_\epsilon} \big ( e^{-(\tilde t-t)\mathbf{A}} {\Delta'}^2 \mathbb{P} g - {\Delta'}^2 \mathbb{P} g \big ) \rangle | & \leq C \epsilon.
\end{align*}
The term $\langle u^\rho (t), \chi_{R_\epsilon} \Delta' v(t) \rangle$ is estimated by using the definition of the solution in the sense of distributions, for $\chi_{R_\epsilon} \Delta' v$ is admissible as a test function on the time interval $[t',t]$. Then we observe that 
\begin{align*}
\langle u^\rho (t), \chi_{R_\epsilon} \Delta' v(t) \rangle & = \int_{t'}^t \langle u^\rho, (\partial_s + \Delta) \chi_{R_\epsilon} \Delta' v(s) \rangle - \langle \nabla p^\rho, \chi_{R_\epsilon} \Delta' v(s) \rangle d s \\
& \quad + \langle u^\rho (t'), \chi_{R_\epsilon} \Delta' v(t') \rangle\\
& =\int_{t'}^t \langle u^\rho, \big ( \Delta \chi_R + 2 \nabla \chi_{R_\epsilon} \cdot \nabla  \big ) \Delta' v(s)  - \chi_{R_\epsilon} \Delta' \nabla q (s) \rangle  d s \\
& \quad - \int_{t'}^t \langle \nabla p^\rho, (\chi_{R_\epsilon} -1) \Delta' v(s) \rangle d s  + \langle u^\rho (t'), \chi_{R_\epsilon} \Delta' v(t') \rangle\\
&=\int_{t'}^t \langle u^\rho, \big ( \Delta \chi_{R_\ep} + 2 \nabla \chi_{R_\epsilon} \cdot \nabla  \big ) \Delta' v(s) \rangle d s\\
& \quad + \int_{t'}^t \langle u^\rho ,  (\nabla \chi_{R_\epsilon}) \Delta'  q (s) \rangle  d s  - \int_{t'}^t \langle \nabla p^\rho, (\chi_{R_\epsilon} -1) \Delta' v(s) \rangle d s \\
& \quad + \langle u^\rho (t'), (\chi_{R_\epsilon} -1) \Delta' v(t') \rangle + \langle u^\rho (t'),  \Delta' v(t') \rangle.
\end{align*}
Here we have used the fact $\langle u^\rho (s) ,   \nabla \big (\chi_{R_\epsilon} \Delta' q (s) \big ) \rangle =0$ and  
$\langle \nabla p^\rho (s) ,  \Delta' v(s) \rangle=0$ for each $s\in (t',t)$, where the latter is verified from ${\nabla'}^\alpha \nabla p^\rho (s) \in L^1_{uloc} (\R^d_+)$ with $|\alpha|\leq 2$ and the pointwise estimate such as  
\begin{align}\label{proof.thm.unique.time.3}
|{\nabla'}^\alpha v (s,x)| + |{\nabla'}^\alpha \nabla v (s,x) | \leq C_{\tilde t-t'} (\tilde t-s)^{-\frac34} (1+|x|)^{-d-\frac12},
\end{align}
which are obtained from Proposition \ref{prop.smooth.appendix} for the resolvent problem and the representation $v(s) = \frac{1}{2\pi i} \int_\Gamma e^{(\tilde t-s)\lambda} (\lambda + \mathbf{A})^{-1} \Delta' \mathbb{P} g d\lambda$. 
From \eqref{proof.thm.unique.time.1} and \eqref{proof.thm.unique.time.3},
we also observe that 
\begin{align*}
|\int_{t'}^t \langle u^\rho, \big ( \Delta \chi_{R_\ep} + 2 \nabla \chi_{R_\epsilon} \cdot \nabla  \big ) \Delta' v(s) \rangle d s|  & \leq C R_{\ep}^{-1} \| u^\rho\|_{L^\infty (t',t; L^1_{uloc}(\R^d_+))} \int_{t'}^t (\tilde t -s)^{-\frac34} d s\\
&  \leq C R_\ep^{-1},\\
|\int_{t'}^t \langle u^\rho ,  (\nabla \chi_{R_\epsilon}) \Delta'  q (s) \rangle  d s |  & \leq C R_\ep^{-1} \| u^\rho \|_{L^\infty (t',t; L^1_{uloc} (\R^d_+))}   \int_{t'}^t (\tilde t -s)^{-\frac34} d s \\
& \leq C R_\ep^{-1},\\
|  \int_{t'}^t \langle \nabla p^\rho, (\chi_{R_\epsilon} -1) \Delta' v(s) \rangle d s  | & \leq C R_\ep^{-\frac14} \| \nabla p^\rho \|_{L^\infty(t',t; L^1_{uloc}(\R^d_+))} \int_{t'}^t (\tilde t -s)^{-\frac34} d s \\
& \leq C R_\ep^{-\frac14},
\end{align*}
and similarly,
\begin{align*}
|\langle u^\rho (t'), (\chi_{R_\epsilon} -1) \Delta' v(t') \rangle|\leq C (\tilde t-t)^{-\frac34} R_\ep^{-\frac14}.
\end{align*}
Therefore, we can take the limit $\ep\rightarrow 0$, which leads to $R_\ep\rightarrow \infty$ and $\tilde t\rightarrow t$, resulting in the identity
\begin{align*}
\langle u^\rho (t) , {\Delta'}^2 g\rangle = \langle u^\rho (t'), {\Delta'} e^{-(t-t')\mathbf{A}} \Delta' \mathbb{P} g \rangle.
\end{align*}
Then we take the limit $\rho\rightarrow 0$, which gives 
\begin{align}\label{proof.thm.unique.time.4}
\langle u (t) , {\Delta'}^2 g\rangle = \langle u (t'), {\Delta'} e^{-(t-t')\mathbf{A}} \Delta' \mathbb{P} g \rangle.
\end{align}
Finally, we take the limit $t'\rightarrow 0$ in \eqref{proof.thm.unique.time.4}. Then the time continuity in the weak sense, which is assumed in the definition of solutions, together with the pointwise estimate for $e^{-(t-t')\mathbf{A}} \Delta' \mathbb{P} g$ similar to \eqref{proof.thm.unique.time.3} implies that 
\begin{align}\label{proof.thm.unique.time.5}
\langle u (t) , {\Delta'}^2 g\rangle = 0.
\end{align}
Since $g\in C_0^\infty (\R^d_+)^d$ is arbitrary, we conclude that for a.e. $t \in (0,T), x_d>0$, $u(t,x',x_d) = (a' (t,x_d), 0)^\top$ by arguing as in the proof of Theorem \ref{thm.unique}.
Once this is shown, the argument is parallel to the proof of  Theorem \ref{thm.unique}; we can show from \eqref{weak.sol.B} that $p$ is independent of $x_d$ and also $\Delta' p=0$ for a.e. $t\in (0,T)$, which implies $p (t,x) = D(t) \cdot x' + c(t)$ for some $D\in L^1_{loc} (0,T)^{d-1}$ and $c\in L^1_{loc} (0,T)$. The last statement for the conclusion $u=0$ is proved in the same manner as in Theorem \ref{thm.unique}, so the details are omitted. The proof is complete.
\end{proof}

\section*{Acknowledgement}
The first author is partially supported by JSPS Program for Advancing Strategic International Networks
to Accelerate the Circulation of Talented Researchers, 'Development of Concentrated Mathematical Center Linking to Wisdom of  the Next Generation', which is organized by Mathematical Institute of Tohoku University.
The second author is partially supported by JSPS grant 25707005.
The third author acknowledges financial support from the French Agence Nationale de la Recherche under grant ANR-16-CE40-0027-01, as well as from the IDEX of the University of Bordeaux for the BOLIDE project.

\small
\bibliographystyle{abbrv}
\bibliography{resolventuloc}

\end{document}